\documentclass[11pt]{article}
\usepackage{amsfonts,amsmath,amssymb,amsthm,mathrsfs,euscript,enumerate,fullpage,color}
\usepackage[utf8]{inputenc}
\usepackage[T1]{fontenc}
\usepackage{natbib}
\usepackage{graphicx} 
\usepackage{hyperref} 
\hypersetup{
colorlinks=true,
breaklinks=true, 
urlcolor= blue,
linkcolor= blue,
citecolor=blue,
}
\usepackage{bbold}

\usepackage[titletoc,title]{appendix}

\newtheorem{theorem}{Theorem}

\newtheorem{lemma}[theorem]{Lemma}
\newtheorem{corollary}[theorem]{Corollary}
\newtheorem{proposition}[theorem]{Proposition}

\title{Integral estimation based on Markovian design}
\author{Romain Aza\"{\i}s$^a$, Bernard Delyon$^b$ and Fran\c{c}ois Portier$^c$}
\date{\small $^a$\textit{Inria Nancy -- Grand Est, Team BIGS and Institut \'Elie Cartan de Lorraine, Nancy, France}\\
$^b$ \textit{Institut de recherches math\'ematiques de Rennes, Universit\'e de Rennes 1}\\
$^c$ \textit{LTCI, CNRS, T\'el\'ecom ParisTech, Universit\'e Paris-Saclay}}

\begin{document}

\maketitle

\begin{abstract}
Suppose that a mobile sensor describes a Markovian trajectory in the ambient space. At each time the sensor measures an attribute of interest, e.g., the temperature. Using only the location history of the sensor and the associated measurements, the aim is to estimate the average value of the attribute over the space. In contrast to classical probabilistic integration methods, e.g., \textit{Monte Carlo},  the proposed approach does not require any knowledge on the distribution of the sensor trajectory. Probabilistic bounds on the convergence rates of the estimator are established. These rates are better than the traditional ``root $n$''-rate, where $n$ is the sample size, attached to other probabilistic integration methods. For finite sample sizes, the good behaviour of the procedure is demonstrated through simulations and an application 
to the evaluation of the average temperature of oceans is considered.
\end{abstract}


\section{Introduction}

For the last decades, climate scientists have been interested in the evolution of different physical 
attributes of Earth to quantify the effects of global warming. For instance, attributes such as temperature, 
acidity and salinity of oceans or the concentration of greenhouse gases in the atmosphere are 
important indicators of global warming. Scientists measurements are often provided by sensors placed on drifting buoys 
in the oceans or weather balloons in the atmosphere, each describing an area or a volume. Whenever the data have been collected, 
a crucial quantity is the average of the measurements over a given space. 
As the sensors are eventually subjected to unpredictable effects such as marine currents or winds, their trajectories 
are modelled as random sequences. The approach taken here is concerned with trajectories satisfying the Markov 
property, meaning roughly that the distribution of the location at time $t+1$ is fully determined by the location at $t$ and an independent random noise.  For the sake of realism, the underlying transition probability and the invariant probability measure associated to the Markov chain are supposed unknown. In summary, the aim is to evaluate the average value of a physical quantity over some space when the measurements are taken along the path of a Markov chain.

More formally, let $Q$ denote a given bounded and open set of $\mathbb R^d$ and suppose that 
$ \varphi: Q\rightarrow \mathbb R$ represents a physical attribute to each location in $Q$, e.g., the temperature in the air over a volume or the wind velocity on the sea over a surface. 
For simplicity, the Lebesgue measure of $Q$ is set to be $1$. Hence, we are interested in 
the average value of $\varphi$ over $Q$, defined as
\begin{align*}
I_0=\int_Q \varphi(x) dx.
\end{align*} 
In most examples of interest the function $\varphi$ is unknown and only some images of the function are obtained from measurement instruments. 
Suppose that we observe $n\in\mathbb N^*$ points from the trajectory of a 
time-homogeneous Harris recurrent Markov chain $X_1,  X_2,\ldots$ \citep{meyn+t:2009} with state space 
$\mathcal E \supseteq Q$. Suppose moreover that we know the associated images by the map $\varphi $, i.e., 
$\varphi(X_1),\ldots \varphi(X_n)$. Let $\pi$ denote the density of the stationary measure of the chain. 
If $\pi$ were known, it would be tempting to compute the \textit{Monte Carlo} estimator of $I_0$,
\begin{align*}
\widehat I_{\text{mc}}= n^{-1} \sum_{i=1}^n \frac{\varphi(X_i)}{\pi (X_i) },
\end{align*} 
which satisfies, under standard conditions \citep[chapter 17]{meyn+t:2009}, a central limit theorem, 
i.e., $n^{1/2}(\widehat I_{\text{mc}} - I_0)$ converges weakly to a centered Gaussian distribution. 
As the previous estimator requires the knowledge of $\pi$, which is not the case in our framework, 
we rather consider the following \textit{kernel smoothing} estimator of $I_0$,
\begin{align*}
\widehat I_{\text{ks}} =  n^{-1} \sum_{i=1}^n \frac{\varphi(X_i)}{\widehat \pi(X_i) },  
\end{align*}
where $\widehat \pi$ is the classical kernel estimator of the density \citep{silverman:1986}, given by,
\begin{align*}
\widehat \pi (x) = (nh_n^d)^{-1} \sum_{i=1} ^n K( (x-X_i)/h_n ), \qquad x\in \mathbb R^d,
\end{align*}
with $K : \mathbb R^d\rightarrow \mathbb R$, a symmetric function, called kernel, that integrates to $1$, 
and $(h_n)_{n\in \mathbb N^*}$, a sequence of positive numbers, called bandwidth, that goes to $0$ as $n\rightarrow +\infty$. 

As the stationary measure is unknown, we can not rely on Monte Carlo integration techniques, often used in simulation-based approximation, such as \textit{importance sampling}, \textit{control variates} or \textit{Metropolis-Hasting integration}. We refer the reader to the books \cite{evans:2000} and \cite{robert:2004} on integral approximation techniques.

The estimator $\widehat I_{\text{ks}}$ has been introduced in 
\cite{delyon2015} where the authors established bounds on the rate of convergence, in probability, in the case of independent 
and identically distributed sequence $X_1,X_2,\ldots$ Their main observation is that the convergence rate of $\widehat I_{\text{ks}}$ 
to $I_0$ is faster than the convergence rate of the Monte Carlo estimator $\widehat I_{\text{mc}}$ to $I_0$ (even though $\widehat I_{\text{mc}}$ requires the knowledge of $\pi$). 
In contrast to standard Monte Carlo methods, the main ingredient of their proposal is the evaluation 
of the image of the design points by the kernel estimator, i.e., $\widehat \pi(X_1), \ldots \widehat \pi(X_n)$. 
These quantities capture an essential information : the isolation of each point. 
Basically, the more isolated $X_i$, the larger the weight $1/\widehat \pi(X_i)$ (and conversely). 
Hence these weights realize an adaptation to the \textit{design points} by attributing more weight to the lonely points. Such a strategy takes

The main theoretical objective of the paper is to extend the results of \cite{delyon2015} when the sequence $(X_i)_{i\in \mathbb N^*}$ is a
time-homogeneous Harris recurrent Markov chain. Denote by $s$ and $r$ the (Nikolski) 
regularity of the functions $\varphi$ and $\pi$, respectively. For any set $B\subset \mathcal E$, let $\tau_B$ denote the return-time 
of the chain to $B$. If there exists $A\subset \mathcal E$ and $p_0>3$ such that
\begin{align}\label{assump:tau_A}
\sup_{x\in A}\mathbb {E}_x[\tau_A^{p_0}]<+\infty,
\end{align}
where $ \mathbb {E}_x$ is the expectation for the Markov chain starting at $X_0=x$, and if, as $n\rightarrow +\infty $,
\begin{align*}
\frac{nh_n^{d(p_0/p_0-1)}}{|\log (n)|} \rightarrow +\infty,
\end{align*}
we show (Theorem \ref{th:bigOpth}), under mild additional conditions, that, as $n\rightarrow +\infty $,
\begin{align*}
\widehat I_{\text{ks}} - I_0 = O_{\mathbb P} \left( h^r_n + n^{-1/2}h_n^s + n^{-1} h_n^{-d}  \right).
\end{align*}
This is the same convergence rate as the one provided in \cite{delyon2015} for independent and identically distributed sequences $(X_i)_{i\in\mathbb N^*}$. The previous rate is better than the rate of $\widehat I_{\text{mc}}$ whenever $n^{1/2}h_n^r \rightarrow  0$  and $n^{-1/2} h_n^{-d}\rightarrow 0$, as $n\rightarrow +\infty$. Taking $h_n \propto n^{-1/(r+d)}$, we obtain a rate in $  n^{-r/(r+d)} + n^{-1/2} n^{-s/(r+d)}  $ which is negligible before $n^{-1/2} $ if and only if $r>d$. Consequently, in addition of being \textit{consistent} when facing Markovian design, the kernel smoothing integral estimator might give an acceleration of the rate of convergence of the Monte Carlo estimator. This acceleration is unfortunately subjected to the well-known curse of the dimension as one needs that $r>d$. In contrast, a nice feature of the method is that only mild constraints are on the regularity of $\varphi$. Finally, there exists a theoretical lower bound for random integration methods \cite[Theorem 3]{Novak} which takes the following form $n^{-1/2} n^{-s/d}$ while our proposal achieves $n^{-1/2} n^{-s/2(s+d)}\geqslant n^{-1}$. This gap in efficiency might be explained by the design distribution which is imposed in our framework.

The mathematical proofs follow from a mixture between the \textit{Nummelin splitting technique} 
for Markov chains \citep{nummelin:1978}, \textit{Hoeffding-type decompositions} 
for $U$-statistics \citep[section 11.4]{vandervaart:1998}  and uniform bounds 
for kernel density estimators in the case of independent observations \citep{einmas}. More specifically, 
the Nummelin splitting technique, also called \textit{regeneration theory} and presented in section \ref{s2}, allows for dividing the chain into $l_n$ independent blocks. 
Assumption (\ref{assump:tau_A}) implies that $l_n$ and $n$ have the same order allowing us 
to mimic the approach of \cite{delyon2015} taken in the independent case: 
\begin{enumerate}[(i)]

\item  Linearise the terms $1/\widehat \pi (X_1),\ldots 1/\widehat \pi (X_n)$ by the help of a Taylor expansion. 
This is typically used in semi-parametric problems as for instance the \textit{single-index model} \citep{hardle:1989,vial:2003}.
\item Find a probabilistic bound on some degenerate $U$-statistic depending on the sequence 
$K( (X_j-X_i)/h )/h^d $, $(i,j)\in \{ 1,\ldots n\}^2$. We shall follow \cite{bertail2011} 
by using an Hoeffding-type decomposition based on the blocks.

\item Too small values of the denominator in $ \widehat I_{\text{ks}} $ are avoided by showing that 
$\inf _{x\in Q} \widehat \pi (x)$ is bounded away from $0$, with large probability. 
In particular, we show (Theorem \ref{prop:unif_conv_density}) that, as $n\rightarrow +\infty$,
\begin{align*}
\sup_{x\in \mathbb R^d } |\widehat \pi(x) -    \pi_{h_n} (x)  | \longrightarrow 0, \quad \text{in probability},
\end{align*} 
where $\pi_{h_n}(x)$ is the expectation of $\widehat \pi(x)$ under stationarity. We rely on empirical process theory and more precisely, 
on a formulation of Talagrand's inequality established in \cite{einmas}. From the best of our knowledge, the previous result in the case of general time-homogeneous Markov chains 
is new. Consistency results (non-uniform) for time-homogeneous Markov chains can be found in \cite{roussas:1969}. In the case of mixing-type dependency, uniform convergence rates are given in \cite{hansen:2008}.
\end{enumerate} 
Steps (i) and (ii) are directly developed in the proof of Theorem  \ref{th:bigOpth}, while the consistency result (iii) is presented in section \ref{s3}.

In contrast with the framework of \cite{delyon2015}, in which the density needs to be continuously differentiable on $\mathbb R^d$, we have been able to include density functions that possibly jumps at the boundary of $Q$ (see the discussion before the statement of Theorem \ref{th:bigOpth}).

To compute $\widehat I_{\text{ks}}$, the bandwidth $h_n$ and the kernel $K$ need to be chosen. 
Preliminary numerical experiments show that $\widehat I_{\text{ks}}$ is quite sensible to the values of $h_n$ whereas 
the choice of $K$ has no strong influence. In \cite{delyon2015}, $h_n$ is chosen according to both the independent points of the design and the function $\varphi$. In the present paper, we propose to use the multivariate plug-in bandwidth selection developed in \cite{Chacon2010}. A simulation study illustrates the good behaviour of the estimator with this choice of the bandwidth in various settings.

The organization of the paper is as follows. 
In section \ref{s2}, we present quickly the regeneration approach for Markov chains. 
The notations and the concepts introduced there will be useful in the rest of the paper. 
Section \ref{s3} is concerned about the uniform convergence of kernel density estimators for Markov chains. 
In section \ref{s4}, we provide the main theoretical statement of the paper which consists in a bound on the rate of convergence of $\widehat I_{\text{ks}}- I_0$.
In section \ref{s5}, we offer a large simulation study as well as a real data analysis performed from sea surface temperature data of the 3 major oceans.
Technical details about regeneration-based bounds for expectations and about the initial measure, as well as the proofs of the results of section \ref{s3}, are presented in appendices \ref{s31}, \ref{s32} and \ref{s33}.

\section{Regeneration}\label{s2}

In this section we give a short account of the regeneration theory, also referred to as the Nummelin splitting technique, as discovered 
in \cite{MR511425} and \cite{nummelin:1978}, 
extensively studied in \cite{nummelin:1984} and \cite{meyn+t:2009}.

We consider a Markov chain $X_0,X_1,X_2,\ldots$
with state space $\mathcal E$ and transition probabilities $P(x,dy)$. 
The notation $\mathbb {E}_\nu$ denotes the expectation according to the chain under $X_0\sim\nu$, 
and $\mathbb {E}_x$ in the case $\nu=\delta_x$. 
The associated probabilities are denoted by $\mathbb {P}_\nu$ and $\mathbb {P}_x$, respectively. 
We assume that for some set $A$ the hitting time
\begin{align*}
\tau_A=\min\{i\geqslant 1: X_i\in A\},
\end{align*}
satisfies
\begin{align}\label{taupr}
&\forall x\in \mathcal E, \mathbb {P}_x(\tau_A<\infty)=1,\\
&\sup_{x\in A}\mathbb {E}_x[\tau_A]<\infty.\label{tauint}
\end{align}
We assume also that for some probability measure $\psi$, some $\lambda_0>0$, and some $m_0\geqslant 1$
\begin{align}\label{ren1_0}
\forall x\in A,~\forall B~\text{measurable},\quad P^{m_0}(x,B)\geqslant \lambda_0\psi(B).
\end{align}
The previous equation means that  $A$ is a ``petite set'' in the terminology of \cite{meyn+t:2009}, section 5.5.2. 
In particular the set $A$ is $\psi$-communicating in the sense 
that \citep[Definition 2.2 p.11]{nummelin:1984}
\begin{align*}
\forall x\in A,~\forall B~\text{measurable},~\psi(B)>0\quad \Rightarrow\quad  \exists m\geqslant 1, P^{m}(x,B)>0.
\end{align*}
As by (\ref{taupr}) the time to reach $A$ is finite with probability $1$, the chain is 
$\psi$-irreducible, i.e., the whole space $\mathcal E$ is $\psi$-communicating. 
An irreducible Markov chain is called Harris recurrent if
\begin{align*}
\forall B\subset \mathcal E\ \text{such that}\ \psi(B)>0,\ \forall x \in \mathcal E,\quad  \mathbb {P}_x( \{X_n\in B\}\  i.o. ) =1 .
\end{align*} 
A consequence of (\ref{taupr}), is that for all $x \in \mathcal E$, $\mathbb {P}_x( \{X_n\in A\}\  i.o. ) =1$. 
Starting from $A$ and if $\psi(B)>0$, from (\ref{ren1_0}) one can deduce that the chain reaches $B$ with 
positive probability. Consequently, under (\ref{taupr}) and (\ref{ren1_0}), the chain is Harris 
recurrent (see \cite{meyn+t:2009}, Proposition 9.1.7, or \cite{nummelin:1984}, Proposition 4.8). From Theorem 10.0.1 in \cite{meyn+t:2009} (see also Corollary 5.3 (ii) in \cite{nummelin:1984}), 
the chain admits an invariant measure and equation (\ref{tauint}) allows to prove that this measure is finite.

If $m_0=1$ the regeneration theory, detailed below, allows to split the chain into independent 
subsequences. This is obviously of great technical interest as many results can be adapted from the independent setting. 
The case $m_0>1$ is somewhat different and we shall say a few words
about it later.

When $m_0=1$, i.e.,
\begin{align}\label{ren1}
\forall x\in A,~\forall B~\text{measurable},\quad P(x,B)\geqslant \lambda_0\psi(B),
\end{align}
 each time the chain hits $A$, it can be restarted with probability $\lambda_0$ with
the measure $\psi$. It should be noted that this assumption is weaker than the well-known \textit{Doeblin condition}
which requires (\ref{ren1}) to hold for every $x\in \mathcal E$. In order to make these regeneration times stopping times,
the chain has to be extended and redefined as the so-called split chain 
$Z_i=(X_i,Y_i)$, $i=1,2\ldots$ having the following transitions:
\medskip

- generation of $Y_i$ given $X_i$
\begin{align*}
\begin{array}{lll}
X_i\notin A&&\longrightarrow Y_i=0,\\
X_i\in A&&\longrightarrow Y_i\sim\EuScript B(1,\lambda_0), \end{array}
\end{align*}

- and generation of $X_{i+1}$
\begin{align*}
\begin{array}{ll}
X_i\notin A&\longrightarrow ~X_{i+1}\sim P(X_i,dx),\\
X_i\in A,~Y_i=0&\longrightarrow ~X_{i+1}\sim(1-\lambda_0)^{-1}(P(X_i,dx)-\lambda_0\psi(dx)),\\
X_i\in A,~Y_i=1&\longrightarrow ~X_{i+1}\sim\psi(dx).
\end{array}
\end{align*}
It is easily checked that the chain $X_0,X_1,X_2,\ldots$ has the right transition probability, $P$. 
In addition, the set $a=A\times \{1\}$ is now an atom for $Z_0,Z_1,Z_2,\ldots$ in the sense that 
(the transition probability of $Z_0,Z_1,Z_2,\ldots$ is abusively still denoted by $P$)
\begin{align}\label{ren2}
\forall z\in a,~\forall C~\text{measurable},~~~P(z,C)=\Psi(C) ,
\end{align}
where $\Psi$ depends only on the measure $\psi$ and $\lambda_0$\footnote{The measure $\Psi$ is given by 
$P(z,B\times \{1\} ) = \int_{B} \psi(x) \lambda_0  \mathbb{1}_{\{x\in A \}}dx$ and 
$P(z,B\times \{0\} ) =  \int_{B } \psi(x)(1- \lambda_0  \mathbb{1}_{\{x\in A \}})dx $.}. 
In particular, the chain regenerates as soon as it gets in $a$, i.e., whenever $Z_i\in a$, 
the distribution of $Z_{i+1},Z_{i+2},\ldots $ is always the same.
We denote the expectation under this measure as $\mathbb {E}_a$. We also set
\begin{align*}
\theta_a=\inf\{i\geqslant 1: Z_i\in a\}.
\end{align*}
As a consequence of (\ref{taupr}) and (\ref{tauint}) (see Lemma~\ref{lemma:tau} in appendix \ref{s31}), 
\begin{align}\label{taufi1}
&\forall z\in \mathcal E\times \{0,1\},\quad  \mathbb {P}_z(\theta_a <\infty)=1,\\
&\alpha_0 = \mathbb {E}_a[\theta_a ]<\infty.\label{taufi2}
\end{align}
Two essential consequences of (\ref{ren2}), (\ref{taufi1}) and (\ref{taufi2}) are the following.
Let $\theta_a(k)$ stand for the $k$-th hitting time of $a$ ($\theta_a(1)=\theta_a>0$),
then the variables
\begin{align}\label{block}
B_k=(Z_{\theta_a(k)+1},\dots Z_{\theta_a(k+1)}),\quad k\in \mathbb N^*,
\end{align}
form an identically and independently distributed sequence of random variables valued 
in $\bigcup_{i \geqslant 1} \mathbb R^i$. These random variables are called ``blocks''. 
And secondly, the chain has a unique invariant probability $\pi$ and
we have the classical formula \citep[equation (5.7)]{nummelin:1984}, for any bounded function $g$,
\begin{align}\label{bert0}
\mathbb {E}_a\Big[\sum_{i=1}^{\theta_a}  g(Z_i)\Big]=\alpha_0 \pi(g).
\end{align}
Based on this, many properties of independent sequences can be extended to Markov chains. As it is useful in our study, we derive in appendix \ref{s31} a bound on the order-$2$ moments of certain empirical sums over Markov chains satisfying (\ref{taupr}), (\ref{tauint}) and (\ref{ren1}).

\paragraph{Control of the recurrence.} As we see with equation~(\ref{tauint}) 
above, a key point for the application of this theory is 
the control of moments of $\tau_A$. This can be classically done through the following result 
(Theorem 3.6 in \cite{MR1890063}): {\it If there exists a function $V\geqslant 1$ such that 
for all $x\in \mathcal E$
\begin{align}\label{vcon}
\mathbb {E}_x[V(X_1)]\leqslant V(x)-cV(x)^{1-\frac1{p}}+c^{-1}\mathbb{1}_{A}(x)
\end{align}
with $c>0$, then for some $c'>0$, for all $x\in \mathcal E$
\begin{align*}
\mathbb {E}_x[\tau_A^{p}]\leqslant c'V(x).
\end{align*}}\paragraph{The case $m_0>1$.}  
Consider for example the chain $X_i=(A_i,B_i)$, $i\in \mathbb N$, with the following transition:
given $X_{i-1}$, draw $U_i\sim\EuScript B(1,1/2)$, $A'_i,B'_i\sim \EuScript N(0,1)$
and set $X_i=(A'_i,B_{i-1})$ if $U_i=0$, and otherwise $X_i=(A_{i-1},B'_i)$.
Then $(X_i)_{i\in \mathbb N}$ does not satisfy (\ref{ren1_0}) with $m_0=1$, but with $m_0=2$. This may induce serious 
complications since the block theory actually fails for the chain $(X_i)_{i\in \mathbb N}$.

However, for $k=0,\dots, m_0-1$, the chain $(X_{im_0+k})_{i\in \mathbb N} $, satisfies  (\ref{ren1}).
Consequently, some properties when $m_0>1$ might be directly deduced from the case $m_0=1$, e.g., for obtaining bounds on empirical sums.
%

\section{Convergence of density estimators}\label{s3}

This section includes some results on kernel estimators of the density of the invariant measure associated to a Markov chain. We start by giving approximation results in $L_p$-spaces and then we consider the question of uniform convergence with the help of empirical process theory.

As the proofs of certain results are technical their proofs are postponed in appendix \ref{s32}. 


\subsection{Approximation in $L_p$-spaces}

We denote by $\lfloor s \rfloor  $ the greater integer smaller than $s$, e.g., $\lfloor 3\rfloor=2$. Following \cite{tsybakov}, we define the Nikolski class
of functions $\EuScript H_q(s,M)$ of regularity $s$ with constant $M>0$ and order $q\geqslant 1$, 
as the set of bounded by $M$ and $\lfloor s\rfloor $-times differentiable functions
$\psi$ whose derivatives of order $\lfloor s\rfloor$ satisfy, for every $u\in \mathbb R^d$,
\begin{align*}
\int \big|\psi^{(l)}(x+u)-\psi^{(l)}(x)\big|^qdx\leqslant M^q|u|_1^{q(s-\lfloor s\rfloor)},\qquad l=(l_1,\dots , l_d)\in\mathbb N^d, \qquad \sum_{i=1}^d l_i\leqslant \lfloor s\rfloor,
\end{align*}
where $\psi^{(l)} = \partial _{x_1}^{l_1}\ldots \partial _{x_d}^{l_d}\psi$ and $|\cdot |_1$ stands for the $\ell _1$-norm. Notice that $0<s-\lfloor s\rfloor\leqslant 1$. When $s<1$, the Nikolski class contains discontinuous functions whereas the more classical H\"older regularity class does not \citep[Lemma 9]{delyon2015}. As a result, the Nikolski class is too large to guarantee pointwise convergence of kernel density estimators. It still ensures convergence in $L_q(\pi)$-norm which is enough for our purpose. While the usual definition of the Nikolski class is with $q=2$, considering different values of $q$ helps when treating the bias of the density estimator along the blocks of the chain.
  
We say that $K$ is a kernel with order $p\in\mathbb N^*$ whenever $K:\mathbb R^d\rightarrow \mathbb R$ is symmetric about $0$, bounded and satisfies 
\begin{align*}
&\int K(x)dx =1,\qquad \int x ^l K(x) dx =0,\ \ l=(l_1,\dots , l_d), ~~0<\sum_{i=1}^d l_i\leqslant p-1,
\end{align*}
with the notation $x^l =(x_1^{l_1},\ldots , x_d^{l_d})$. 

For every $h>0$, we introduce the notation 
\begin{align*}
K_h(\cdot)=h^{-d}K(\cdot/h).
\end{align*}
For any other function $\psi:\mathbb R^d \rightarrow \mathbb R$, the convolution between $\psi$ and $K_h$ is given by
\begin{align*}
\psi_h(x) =(\psi \ast K_h)(x)  =  \int \psi (x-hu)K(u)   du.
\end{align*}
The following Lemma asserts that for kernels with sufficiently high order, the larger the Nikolski regularity of $\psi$ and $\pi$ the better the rate of convergence of $\psi_h$ to $\psi$ in $L_q(\pi)$-norm. For any bounded real-valued function $g$ defined on some space $\mathcal X$, we set 
\begin{align*}
g_{\infty} = \sup_{x\in \mathcal X} |g(x)|. 
\end{align*}

\begin{lemma}\label{lemma:regularizationrates}
Let $s>0$, $q\geqslant 1$ and suppose that $K$ has order (strictly) greater than $\lfloor s \rfloor$ such that $\int |u|_1^s|K(u)|du<+\infty$ and $\psi:\mathbb R^d \rightarrow \mathbb R $ belongs to $\EuScript H_q(s,M_1)$, then for any bounded density $\pi$ on $\mathbb R^d$, and every $h>0$,
\begin{align}
&\big\| \psi -\psi_{h} \big\|_{L_q(\pi)} \leqslant C_1 M_1 \pi_\infty^{1/q}  h^{s},\label{eq:L2regularization}
\end{align}
where $C_1$ depends on $K$ and $s$. Suppose the previous assumptions hold with $q=1$. Let $r>0$ and assume moreover that $K$ has order (strictly) greater than $\lfloor r \rfloor$ such that $\int |u|_1^r|K(u)|du<+\infty$, $\pi$ belongs to $\EuScript H_1(r,M_2)$ and $\int |\psi(x)|dx <+\infty$, then there exists $C_2>0$ such that, for every $h>0$,
\begin{align}\label{eq:L1regularization}
& |\pi(\psi -\psi_{h})|  \leqslant C_2 (M_1\pi_\infty + M_2 \psi_\infty) h^{r\vee s},
\end{align}
where $ C_2$ depends on $K$, $s$ and $r$.
\end{lemma}

\subsection{Uniform concentration}

The considered approach is based on empirical process theory and more precisely on the following result from \cite{einmas}. Given independent and identically distributed 
 random variables $\xi_1, \xi_2,\ldots$, it provides a bound on the expected value of 
\begin{align*}
\sup_{f\in \mathcal F}\Big| \sum_{i=1}^n (f(\xi_i)-\mathbb E [f(\xi_1)] ) \Big|,
\end{align*}
whenever the class of function $\mathcal F$ is a VC class of functions (see Theorem \ref{theorem:gine+g} below). A class $\mathcal F$ is VC whenever there exist $A>0$ and $v>0$ such that, for every probability measure $Q$ satisfying
$\|F\|_{L_2(Q)}<\infty$, and every $0<\epsilon<1$,
\begin{align*}
\mathcal N \left(\mathcal F,\, L_2(Q),\, \epsilon\|F\|_{L_2(Q)}\right)\leqslant \left(\frac{A}{\epsilon}\right)^{v}, 
\end{align*}
where $F$ is an envelope for $\mathcal F$, i.e., for any $f\in \mathcal F$, $|f(x)|\leqslant F(x)$, and $\mathcal N(T,d,\epsilon)$ denotes the $\epsilon$-covering number of the metric space $(T,d)$ \citep{wellner1996}. Many classes of interest turn out to be VC, e.g., polynomials and indicators, and several preservation properties are available (see Proposition \ref{theorem:preservationprop}, \ref{proposition:pollard+nollan} and \ref{corollary:vcclass_example} below). 

The following statement is actually a slight modification of Proposition 1 in \cite{einmas}. Comments are given below.

\begin{theorem}[Einmahl and Mason (2005)] \label{theorem:gine+g}
Let $\xi_1,\dots \xi_n$ be an i.i.d. sequence and $\mathcal F$ be a VC class of functions with envelope $F$ 
and characteristics $(A,v)$ with $A\geqslant e$ and $v\geqslant 1$, and set $\beta^2=\mathbb E[F(\xi_1)^2]$. 
Let $\sigma^2$ be such that
\begin{align}\label{einmasas1}
& \sigma^2\geqslant\sup_{f\in \mathcal F} \mathbb E[f(\xi_1)^2]),\\
&\sigma^2\geqslant 16vn^{-1}\log\big(A(\tfrac{\beta}{\sigma}\vee 1)\big)\sup_{f\in \mathcal F,\, x\in \mathcal X} f(x)^2,\label{einmasas2}
\end{align}
then
\begin{align}
\mathbb E  \sup_{f\in \mathcal F} \Big|  \sum_{i=1} ^n (f(\xi_i)- \mathbb E[f(\xi_1)]) \Big| \leqslant 
C_0\sqrt {vn\sigma^2 \log\big(A(\tfrac{\beta}{\sigma}\vee 1)\big)},\label{einmaseq}
\end{align}  
where $C_0$ is a universal constant.
\end{theorem}


In  \cite{einmas} the left hand side is actually a Rademacher sum, but then (\ref{einmaseq}) follows from
the Symmetrization Lemma, e.g., Lemma~2.3.1 in \cite{wellner1996}. Another difference is that it is
stated only in the case $\sigma\leqslant\beta$. But if
$\sigma\geqslant\beta$, one can increase $F$, e.g., $F\rightarrow a\vee F$, in such a way that
$\beta$ will be equal to $\sigma$ ($A$ and $v$ do not change) and apply the previous result;
this leads to (\ref{einmaseq}).

Preservation properties of the covering number's size will be useful in the sequel to show that some classes are VC. 
The following proposition asserts that locally Lipschitz transformations of VC classes are still VC. This result is a slight variation of Theorem 2.10.20 in \cite{wellner1996} in which the authors consider uniform entropy numbers with respect to discretely finite probability measures.

\begin{proposition}\label{theorem:preservationprop}
Let $ \mathcal F_1, \ldots  \mathcal F_d$ be VC classes of functions defined on a common space $\mathcal X$ such that 
each $f\in\mathcal F_j$ is valued in the set $I_j\subset \mathbb R$ and $\mathcal F_j$ has envelope $F_j$. 
Let $\Psi: I_1\times \ldots \times I_d\rightarrow \mathbb R$ be such that
for any $A=(A_1,\dots A_d)\in \mathbb R_+^d$
\begin{align}\label{def:locally_lipschitz}
&|\Psi(z) - \Psi(\widetilde z)| \leqslant \sum_{j=1}^d C_j( A) |z_j-\widetilde z_j|
 ,\qquad \forall z,\widetilde z\in 
 \big([-A_1,A_1]\cap I_1\big)\times \ldots \times \big([-A_d,A_d]\cap I_d\big),
\end{align} 
where $C_j:\mathbb R^d \rightarrow \mathbb R$, $j=1,\ldots d$, are non-negative functions. 
Let $\mathcal G$ denote the class of functions $x\mapsto \Psi(f_1(x),\ldots f_d(x) )$ 
when $(f_1,\ldots f_d) $ ranges over $ \mathcal F_1\times \ldots \times  \mathcal F_d$.  
The class $\mathcal G$ is a VC class of functions with envelope 
\begin{align*}
&G=|\Psi (f_0)| +2 \sum_{j=1}^d  (1\vee F_j){C_j(F) } ,
\end{align*}
where $F=(F_1,\ldots F_d)$, and $f_0$ is an arbitrary function in $ \mathcal F_1\times \ldots \times  \mathcal F_d$.
\end{proposition}

The following proposition, which includes a result from \cite{nolan+p:1987}, provides interesting examples of  uniformly bounded VC classes of functions. 
We shall consider a kernel function $K:\mathbb R^d \rightarrow \mathbb R$ 
that takes one of the two following forms,
 \begin{align}\label{assump:kernel}
  (i) \quad K(x) =  K^{(0)} (|x|), \qquad \text{or} \qquad (ii) \quad  K(x) = \prod_{k=1}^d  K^{(0)}(x_k),
 \end{align}
where  $K^{(0)} $ a bounded real function of bounded variation. We denote by $K_\infty$ the supremum of $K$.

\begin{proposition}\label{proposition:pollard+nollan}
The class of functions $\{ x\mapsto \mathbb{1}_{x\leqslant M}\,:\,M\in\mathbb R \}$ is a uniformly bounded VC class of functions. Assume that (\ref{assump:kernel}) holds. The class of functions $\{ x\mapsto K ( h^{-1}(y-x))\, : \, y\in \mathbb R^d,\, h>0 \}$ is a uniformly bounded VC class of functions.
\end{proposition}

By applying Proposition \ref{theorem:preservationprop} to the VC classes of the previous proposition, we establish the VC property for some class of functions which will be of great interest in the sequel.

\begin{proposition}\label{corollary:vcclass_example}
Assume that (\ref{assump:kernel}) holds. The class of functions  
\begin{align}\label{grosvc}
\big\{\, (t,x) \mapsto t \mathbb{1}_{t\leqslant M} K ( h^{-1}(y-x))\, : \, y\in \mathbb R^d,\, h>0 ,\, M\in\mathbb R \,\big\},
\end{align}
defined on $\mathbb R\times \mathbb R^d$ is a VC class of functions 
with envelope $(t,x)\mapsto 2 ((1\vee K_\infty  ) |t| +  (1\vee |t|) K_\infty )$.
\end{proposition}

Based on Proposition \ref{proposition:pollard+nollan}, if the random variables $X_1,X_2,\ldots$ 
used in the construction of $\widehat \pi$, were independent, 
then we would have, under the assumptions of Theorem \ref{theorem:gine+g} and Proposition \ref{proposition:pollard+nollan}, that 
\begin{align*}
\sup_ {y\in \mathbb R} | \widehat  \pi(y)-\pi_{h_n}(y)|  = O_{\mathbb P}\left (\sqrt{ \frac{\log n}{nh_n^{d}}} \right),
\end{align*}
whenever $h_n\rightarrow 0$ and $ {nh_n^{d}}/\log (n) \rightarrow +\infty $, as $n\rightarrow +\infty$. For Markov chains, we require the stronger condition on the sequence of bandwidth,
\begin{align}\label{nhdp}
 h_n\rightarrow 0, \qquad nh_n^{dp_0/(p_0-1)}/\log (n)\rightarrow +\infty,
 \end{align}
for some $p_0>2$ such that
\begin{align}\label{c0tau}
\xi(p_0) = \sup_{x\in A}\mathbb {E}_x[\tau_A^{p_0}]<+\infty.
\end{align}
In addition our approach only permits to obtain the convergence to $0$ in probability, not any sharp bound on the rate of convergence.

\begin{theorem}\label{prop:unif_conv_density}
Let $(X_i)_{i\in \mathbb N }$ be a Markov chain satisfying (\ref{taupr}), (\ref{ren1}) and (\ref{c0tau}) for some $p_0>2$.
Suppose that $K$ satisfies (\ref{assump:kernel}) and that (\ref{nhdp}) holds true
for the same $p_0>2$. If $\pi$ is bounded, and $\int (|K(x)|+K(x)^2)dx<+\infty$, we have
\begin{align*}
\sup_{y\in \mathbb R^d}| \widehat\pi(y)-  \pi_{h_n}(y)| \longrightarrow 0,\quad \text{in $\mathbb P_\pi$-probability}.
\end{align*} 
\end{theorem}

Working further on the difference between $\pi$ and $ \pi_{h_n}$ leads to the following statement which prevents the estimated density of being too close to $0$.

\begin{corollary}\label{prop:coro:lowerbound_density}
Under the assumptions of Theorem~\ref{prop:unif_conv_density}, suppose that $Q\subset \mathbb R ^d $ is a compact set such that $\pi$ is continuous on $Q$ and $\inf_{y\in Q} \pi(y)\geqslant b>0$. If $K$ has bounded support and if there exists $c>0$ and $h_0>0$ such that for every $x\in Q$, $0<h<h_0$, it holds that $(\mathbb 1_{\{ Q\}} \ast K_h)(x) \geqslant c$, then
\begin{align*}
\mathbb {P}_\pi\left(\inf_{y\in Q}\widehat  \pi(y) \geqslant \frac{cb}{2} \right) \rightarrow 1.
\end{align*}
\end{corollary}



\section{Main result}\label{s4}

We now provide the rate of convergence of the estimator $\widehat I_{\text{ks}}$ of $I_0$. We rely largely on the regenerative framework described in the previous section. In particular, the following set of assumptions ensures the statements of Theorem \ref{prop:unif_conv_density} and Corollary \ref{prop:coro:lowerbound_density}.

\begin{enumerate}[(\text{A}1)]
\item \label{ash1} For some $s>0$ and $M_1>0$, the support of $\varphi$ is a compact set $Q\subset \mathbb R^d$ and $\varphi$ belongs to $\EuScript H_q(s,M_1)$ for any $q\geqslant 1$.
\item \label{ash2} For some $r>0$ and $M_2>0$, $\pi$ is continuous, bounded on $Q$ and belongs to $\EuScript H_q(r,M_2)$ for any $q\geqslant 1$. Moreover, there exists $b>0$ such that $\inf_{y\in Q} \pi(y)\geqslant b$.
\item \label{ash4} Let $K$ be a kernel satisfying (\ref{assump:kernel}) with order (strictly) greater than $r$ and $s$. There exists $c>0$ and $h_0>0$ such that for every $x\in Q$ and $0 < h<h_0$,
\begin{align*}
(\mathbb 1_{\{ Q\}} \ast K_h)(x) \geqslant c.
\end{align*}
\item \label{ash6:bandwidth} Let $(X_i)_{i\in \mathbb N}$ be a Markov chain satisfying (\ref{taupr}) and (\ref{ren1}) and initial measure $\nu$ absolutely continuous with respect to $\pi$. There exists $p_0>3$ such that 
\begin{align*}
\sup_{x\in A} \mathbb {E}_{x}[ \tau_A^{p_0}]<+\infty,
\end{align*}
where $A$ is the recurrent set introduced in (\ref{taupr}), and, as $n\rightarrow +\infty$, the sequence of bandwidth $(h_n)_{n\in \mathbb N^*}$ satisfies, as $n\rightarrow +\infty$,
\begin{align*}
h_n\rightarrow 0,\qquad \frac{nh_n^{dp_0/(p_0-1)}}{\log (n)} \rightarrow +\infty. 
\end{align*} 
\end{enumerate}

Most stable Markov chains satisfy (A\ref{ash6:bandwidth}). This has been the subject of many studies as presented in \cite{meyn+t:2009} where the drift condition (\ref{vcon}) is used to bound the moments of the return times. Examples include for instance auto-regressive models \cite[Theorem 16.5.1, equation 16.43]{meyn+t:2009} or the Metropolis-Hasting algorithm \cite[Example 5.2]{MR1890063}. Since the invariant measure $\pi$ is solution to $\pi(y)=\int \pi(x)P(x,y)dx$,
where $P$ is the transition density, the smoothness of $y\mapsto P(x,y)$ will essentially ensure the smoothness
of $\pi$ as required in (A\ref{ash2}).  Whenever $\pi>0$ on the support of $\varphi$ (e.g., as soon as $P(x,y)>0$ for all $(x,y)$) and continuous, the lower bound in (A\ref{ash2}) holds.

High order kernels can be constructed using radial kernel (\ref{assump:kernel})(i) or using product-type kernel (\ref{assump:kernel})(ii) following for instance \cite{gasser:1985} or \cite[section 1.11]{liracine}. The condition that $(\mathbb 1_{\{ Q\}} \ast K_h)$ is lower bounded (uniformly for $x$ and small $h$) intervenes in Corollary \ref{prop:coro:lowerbound_density} which is a key ingredient to control the small values of $\widehat \pi$. This condition cannot trivially verified as it involves the boundary of $Q$ and the regions where $K<0$. A first example is when $Q$ is the hypercube and $K$ is a product type kernel with initial kernel $K^{(0)}$ such that $\int_{-x}^{+\infty}  K^{(0)}(u)du>0$ for all $x>0$. A second example is when the boundary of $Q$ is smooth and $K$ is such that $\int _{\mathcal H} K(u) du >0$ for every half-space $\mathcal H$ containing $0$.

The following theorem extends the results of \cite{delyon2015} for independent sequences of random variables to Harris recurent Markov chains. A secondary improvement with respect to \cite{delyon2015} concerns the  requirements on the regularity of $\pi$. In \cite{delyon2015}, the density $\pi$ is assumed to be at least continuously differentiable on $\mathbb R^d$ and bounded away from $0$ on $Q$, excluding the case where $ \pi$ is supported on $Q$, and possibly discontinuous on the boundary. In the present approach, we include such cases by supposing that $\pi$ is in some Nikolski's regularity class. This informs us on the effect of jumps in the shape of $\pi$.  As the Nilkolski's regularity of such functions is smaller than $1/2$ \citep[Lemma 11]{delyon2015}, a bias term in  $h_n^{1/2}$ shall appear in the asymptotic decomposition.

\begin{theorem}\label{th:bigOpth}
 If moreover, (A\ref{ash1}) to (A\ref{ash6:bandwidth}), we have for every initial measure 
\begin{align*}
\widehat I_{\text{\normalfont ks}} -I_0&=O_{\mathbb P_{\nu}} (h_n^r+n^{-1}h_n^{-d}+n^{-1/2}h_n^s).
\end{align*}
\end{theorem}

\begin{proof}


We consider the split chain $(Z_i)_{i\in \mathbb N}$ introduced in section \ref{s2} with initial distribution $\nu$. We are interested in showing that $\mathbb E_{\nu} \mathbb{1}_{\{ | \widehat I_{\text{ks}} - I_0 |>a_n\}}\rightarrow 0$ for some sequence $a_n\rightarrow 0$. By applying Lemma \ref{lemma:initial_measure}, it suffices to prove the result in the case when $\nu$ equal $\pi$. 

By (\ref{lnconv}), we have that $l_n/n$ converges to its expectation $\alpha_0^{-1}>0$. We shall use several times that $n/l_n =O_{\mathbb P_\pi}(1)$ and that the product of two $O_{\mathbb P_\pi} (1)$ remains $ O_{\mathbb P_\pi}(1)$. 


Without loss of generality, we can assume that $l_n>2$. Indeed, the complementary event occurs with probability going to $0$ as $n$ increases.  

A convenient scaling in the sequel is to put $\alpha_0(l_n-1)$ and $\alpha_0(l_n-2)$ instead of $n$, in some places, because it simplifies many terms of our expansion. Hence, instead of $\widehat I_{\text{ks}}$, we rather study 
\begin{align*}
\widetilde  I_{\text{ks}} = (l_n-1)^{-1}\alpha_0^{-1}\sum_{i=1}^n \frac{\varphi(X_i)}{  \widehat \pi_i},
\end{align*}
with 
\begin{align*}
 \widehat \pi_i =\alpha_0^{-1}(l_n-2)^{-1} \sum_{j=1} ^{  n} K_{ij}.
\end{align*}
and $K_{ij}=K_{h_n}(X_i-X_j)$. Since $\widehat  I_{\text{ks}} =\left(\frac{l_n-1}{l_n-2} \right) \widetilde  I_{\text{ks}} $ and $\left(\frac{l_n-1}{l_n-2} \right) =O_{\mathbb P_\pi} (1)$, the rates of convergence of $\widetilde I_{\text{ks}}$ and $\widehat I_{\text{ks}}$, in probability, are the same.

We now introduce the notation 
\begin{align*}
&\psi_q(x)=\tfrac{\varphi(x)}{\pi(x)^q},~~~q\in\mathbb N.
\end{align*}
The following development, reminiscent of the Taylor expansion of $\widehat \pi_i$ around $\pi(X_i)$,
\begin{align*}
\frac {1}{\widehat \pi_i } =  \frac{2} { \pi(X_i)} 
- \frac {\widehat \pi_i } { \pi(X_i)^2 } 
+\frac {(\pi(X_i) -\widehat \pi_i  )^2} {\widehat \pi_i \pi(X_i)^2},
\end{align*}
allows us to expand $\widetilde  I_{\text{ks}}$ as follows
\begin{align*}
\widetilde  I_{\text{ks}} = (\alpha_0 (l_n-1))^{-1} \sum_{i=1}^n 2 \psi_{1}(X_i) 
 -( \alpha_0^2 (l_n-1)(l_n-2)  )^{-1} \sum_{i=1}^n \sum_{j=1}^n \psi_{2}(X_i)K_{ij}+R_{1,n} ,
\end{align*}
with 
\begin{align*}
R_{1,n} = (\alpha_0(l_n-1))^{-1} \sum_{i=1}^n\frac{\psi_2(X_i)(\pi(X_i)-\widehat \pi_i)^2}{\widehat \pi_i} .
\end{align*}
Reorganizing the first two terms according to the blocks leads to 
\begin{align*}
\widetilde  I_{\text{ks}} = -( \alpha_0^2 (l_n-1)(l_n-2)  )^{-1}  \sum_{ k = 0} ^{l_n} \sum_{ l = 0} ^{l_n}  H_{kl}+ (\alpha_0 (l_n-1))^{-1}\sum_{k=0}^{l_n} 2 G_k +R_{1,n},
\end{align*}
with for any $(k,l)\in \mathbb N^2$,
\begin{align*}
&H_{kl}= \sum_{i\in B_k ,j\in B_l}  \psi_{2}(X_i)K_{ij}\\
&G_k =\sum_{i\in B_k } \psi_{1}(X_i).
\end{align*}
The notation $i\in B_k$ is a short-cut for $Z_i\in B_k$ and the block $B_0$ is the first (incomplete) 
block given by $Z_1,Z_2,\ldots Z_{\theta_a}(1)$.
Diagonal terms of the above $U$-statistic and terms related to the first and last block are 
treated as remainder, we write
\begin{align}\label{decomp10}
\widetilde  I_{\text{ks}} = -( \alpha_0^2 (l_n-1)(l_n-2)  )^{-1} \sum_{ k = 1} ^{l_n-1} \sum_{ l = k+1} ^{l_n-1} \{H_{kl}^*\}+ (\alpha_0 (l_n-1))^{-1} \sum_{k=1}^{l_n-1} \{2 G_k +R_{1,n}+R_{2,n}\},
\end{align}
with $H_{kl}^*= H_{kl}+H_{lk}$ and
\begin{align*}
R_{2,n}= &-( \alpha_0^2 (l_n-1)(l_n-2)  )^{-1}  \big(H_{00}+H_{l_nl_n}+H_{0l_n}^* +\sum_{k=1}^{l_n-1} \{H_{0k}^*+ H_{l_nk}^*+H_{kk} \}\big) \\
&+ (\alpha_0 (l_n-1))^{-1} 2 ( G_0+ G_{l_n}) .
\end{align*}
The first term in (\ref{decomp10}) is a $U$-statistic whose fluctuations can be controlled by using an Hoeffding-type decomposition with respect to the blocks. Denoting 
\begin{align*}
\widetilde H_k^*=\mathbb E_ a[H_{1k}^* |B_k],
\end{align*}
we can rewrite
\begin{align*}
\widetilde  I_{\text{ks}}
 -I_{\varphi} =U_n+M_n+B_n+R_{1,n}+R_{2,n},
\end{align*}
with (we use that $\sum_{1\leqslant k< l\leqslant l_n-1}\{\widetilde H_k^*+\widetilde H_l^*\}=(l_n-2)\sum_{1\leqslant  k\leqslant l_n-1} \widetilde H_k^* $ and we underbrace terms which have been deliberately introduced and removed)
\begin{align*}
&U_n = -( \alpha_0^2 (l_n-1)(l_n-2)  )^{-1} \sum_{ k = 1} ^{l_n-1} \sum_{ l = k+1} ^{l_n-1} \{H_{kl}^*\underbrace{-\widetilde H_k^*-\widetilde H_l^*}_{(1)}+\underbrace{E[H_{12}^*]}_{(2)}\},\\
&M_n= (\alpha_0 (l_n-1) )^{-1} \sum_{k=1}^{l_n-1}\{  2G_k - \underbrace{\alpha_0^{-1} \widetilde H_k^*}_{(1)} -\underbrace{E(2G_1 -\alpha_0^{-1} \widetilde H_1^*)}_{(3)}\}, \\
&B_n = \alpha_0 ^{-1}\underbrace{E(2G_1 -\alpha_0^{-1} \widetilde H_1^*)}_{(3)}\} + \underbrace{( 2\alpha_0^2)^{-1} E[H_{12}^*] }_{(2)} - \int \varphi(x) dx.
\end{align*}
The notations follow from the fact that $U_n$ is a U-statistic, $M_n$ is a martingale, $B_n$ is a bias term (nonrandom), $R_{1,n}$ comes from the remainder of the Taylor expansion and $R_{2,n}$ corresponds to uncompleted blocks and diagonal terms. We shall now compute bounds for each term separately.

\paragraph{Step~1.} $U_n=O_{\mathbb P}(n^{-1}h_n^{-d/2})$.

Let $\widetilde U_n = (\tfrac {\alpha_0^2(l_n-1)(l_n-2)}{n^2}) U_n$, we have, since $l_n\leqslant n$
\begin{align*}
|\widetilde U_n |
&\leqslant n^{-2} \max_{1\leqslant L\leqslant n} \left|S_L\right| ,
\end{align*}
with $S_L = \sum_{1\leqslant k< l\leqslant L}  \{H_{kl}^*-\widetilde H_k^*-\widetilde H_l^*+E[H_{12}^*]\}$. The independence between the blocks $(B_k)_{k=1,\ldots n}$ defined in (\ref{block}), 
implies that the process $L\mapsto S_L$ is a martingale. 
Then by Doob's inequality, we know that
\begin{align*}
\mathbb {P}_\pi(|\widetilde U_n |>\epsilon) &\leqslant \frac{ES_n^2}{\epsilon^2 n^4},
\end{align*}
and it remains to develop the squared sum inside the expectation. 
By construction, the terms in the sum defining $S_L$ are all orthogonal. As a consequence, we find
  \begin{align*}
\mathbb {P}_\pi(|\widetilde U_n |>\epsilon) &\leqslant \frac{ n(n-1) \mathbb E_a [\{H_{12}^*-\widetilde H_1^*-\widetilde H_2^*+E[H_{12}^*]\}^2] }{2\epsilon^2 n^4}\leqslant \frac{ \mathbb E_a [H_{12}^{*2}] }{2\epsilon^2 n^2}.
\end{align*}
Because of the symmetry of $K$ and the boundedness of $\psi_2$ and $ K  $, we have, denoting by $\Delta_k$ the length of $B_k$ for $k\in\mathbb N$ (as introduced in the proof of Theorem \ref{prop:unif_conv_density}),
  \begin{align*}
 \mathbb E_a [H_{12}^2] 
&= \mathbb E_a \Big( \sum_{i\in B_1, j\in B_2}\{\psi_2(X_i)K_{ij}+\psi_2(X_j)K_{ji}\}\Big)^2 \\
&\leqslant \psi_{2,\infty}^2 \mathbb {E}_a\Big(\sum_{i\in B_1, j\in B_2} | K_{ij}|+|K_{ji}|  \Big)^2\\
&\leqslant \psi_{2,\infty}^2 K_\infty h_n^{-d}  
\mathbb {E}_a\Big(\Delta_1\Delta_2 \sum_{i\in B_1, j\in B_2}|K_{ij}|+|K_{ji}|\Big)\\
&\leqslant\psi_{2,\infty}^2K_\infty h_n^{-d}  
\mathbb {E}_a\Big(  (\Delta_1^2+ \Delta_2^2)\sum_{i\in B_1, j\in B_2} | K_{ij}|\Big)\\
&=2\psi_{2,\infty}^2 K_\infty h_n^{-d}\mathbb {E}_a\Big(\Delta_1^2 \sum_{i\in B_1,j\in B_2}| K_{ij}|  \Big).
\end{align*}
The independence between the blocks permits to integrate with respect to $B_2$ knowing $B_1$, 
and that yields, using (\ref{bert0}),
\begin{align*}
\mathbb {E}_a\Big(  \Delta_1^2  \sum_{i\in B_1, j\in B_2} | K_{ij}| \Big) 
&=\alpha_0 \mathbb {E}_a\Big(  \theta_a^2  \sum_{i\in B_1} \int | K_{h_n} (X_i-y)|\pi(y)dy \Big) \\
&\leqslant \alpha_0 \pi_\infty \int | K( x)| dx ~\mathbb {E}_a [\theta_a^3].
\end{align*}
From Lemma~\ref{lemma:tau} and Assumption~(A\ref{ash6:bandwidth}), $\mathbb {E}_a  \theta_a^3$ is finite. Conclude using that $ U_n = O_p(1) \widetilde U_n$.

\paragraph{Step~2.} $M_n=O_{\mathbb P}(n^{-1/2}h_n^{s \wedge r})$.

Consider $\widetilde M_n = \big(\frac {\alpha_0 (l_n-1)} { n }\big) M_n$, we have
\begin{align*}
|\widetilde M_n |&\leqslant n^{-1} \max_{1\leqslant L\leqslant n} \Big| \sum_{1\leqslant k \leqslant L} \{2G_k - {\alpha_0^{-1} \widetilde H_k^* } -{E(2G_k -\alpha_0^{-1} \widetilde H_k^*)}\} \Big| ,
\end{align*}
and Doob's inequality yields
\begin{align*}
\mathbb {P}_\pi (|\widetilde M_n |>\epsilon) &\leqslant \frac{\mathbb E_a\big(\sum_{1\leqslant  k\leqslant n} \{2G_k - {\alpha_0^{-1} \widetilde H_k^*} -{E(2G_1 -\alpha_0^{-1} \widetilde H_1^*)}\} \big)^2}{\epsilon^2 n^2}\\
&= \frac{\mathbb E_a \big( 2G_1 - {\alpha_0^{-1} \widetilde H_1^*} -{E(2G_1 -\alpha_0^{-1} \widetilde H_1^*)} \big)^2}{\epsilon^2 n}\\
&\leqslant \frac{\mathbb E_a \big( 2G_1 - {\alpha_0^{-1} \widetilde H_1^*} \big)^2}{\epsilon^2 n}.
\end{align*}
Because of (\ref{bert0}), we have
\begin{align}\nonumber
\alpha_0^{-1} \widetilde H_1^*&= \sum_{i\in B_1} \int \Big(\psi_{2}(X_i)K_{h_n}(X_i-y)+\psi_{2}(y)K_{h_n}(y-X_i) \Big) \pi (y) dy\\
&=  \sum_{i\in B_1}  \{\psi_{2}(X_i)\pi_{h_n}(X_i) + \psi_{1h_n}(X_i)\} ,\label{psi2psi1}
\end{align}
hence it holds
\begin{align}\label{equationtermMn}
2G_1 - {\alpha_0^{-1} \widetilde H_1^*}  =   \sum_{i\in B_1}  \{\psi_{2}(X_i)(\pi(X_i)-\pi_{h_n}(X_i)) + (\psi_1(X_i)- \psi_{1h_n}(X_i))\}.
\end{align}
Then from Minkowski's inequality and Lemma \ref{lemma:insideblock}, we get for some $2<p<p_0-1$ (see Assumption (A\ref{ash6:bandwidth})),
\begin{align*}
\| 2G_1 - {\alpha_0^{-1} \widetilde H_1^*} \|_2 &\leqslant \psi_{2,\infty} \big\| \sum_{i\in B_1} | \pi(X_i)-\pi_{h_n}(X_i)| \big\|_2+ \big\| \sum_{i\in B_1} | \psi_1(X_i)- \psi_{1h_n}(X_i) | \big\|_2 \\
&\leqslant C \big( \big\|( \pi(X_0)-\pi_{h_n}(X_0))\tau_A^{p/2}\big\|_2+\big\|( \psi_1(X_0)-\psi_{1h_n}(X_0))\tau_A^{p/2}\big\|_2\big) ,
\end{align*}
where $C $ is a constant that depends on $p$ and on the chain and $\| \cdot \|_2$ stands for the $L_2(\pi)$-norm. Now we use H\"older's inequality, with conjugates $u$ and $v$, to obtain
\begin{align*}
\| 2G_1 - {\alpha_0^{-1} \widetilde H_1^*} \|_2 &\leqslant C \| \tau_A^{p/2}\big\|_{2v} \left( \big\| \pi(X_0)-\pi_{h_n}(X_0)\big\|_{2u}+\big\| \psi_1(X_0)-\psi_{1h_n}(X_0)\big\|_{2u}\right) .
\end{align*}
Now choose $v$ sufficiently close to $1$ to ensure, using (A\ref{ash6:bandwidth}), (\ref{eq:inequality_invariant_measu}) 
and (\ref{c0tau0prime}), that $\mathbb {E}_\pi[\tau_A^{pv}]\leqslant \mathbb {E}_\pi [\tau_A^{p_0-1}]\leqslant \mathbb {E}_\pi [\theta_a^{p_0-1}]< +\infty $. Use Lemma \ref{lemma:regularizationrates} to obtain the desired rate, 
$h_n^r+h_n^s$, for the two other quantities.

\paragraph{Step~3.} $B_n=O(h_n^{r} )$.

By (\ref{bert0}) and (\ref{equationtermMn}), we have that 
\begin{align*}
\alpha_0^{-1} \mathbb E_a (2G_1 -\alpha_0^{-1} \widetilde H_1^*) 
=  \int \psi_{2}(x)(\pi(x)-\pi_{h_n}(x))\pi(x)dx  + \int (\psi_1(x)- \psi_{1{h_n}}(x)) \pi(x)dx ,
\end{align*}
and using  (\ref{eq:fubini}) and the definition of $\psi_2$ gives
\begin{align*}
\alpha_0^{-1} \mathbb E_a (2G_1 -\alpha_0^{-1} \widetilde H_1^*) 
=  2\pi(\psi_1- \psi_{1{h_n}}).
\end{align*}
Similarly from (\ref{psi2psi1}), (\ref{bert0}) and (\ref{eq:fubini}), it follows that
\begin{align*}
(2\alpha_0^{2})^{-1} \mathbb E_a [H_{12}^*] = (2\alpha_0^{2})^{-1} \mathbb E_a[\widetilde H_{1}^*]
= \frac 1 2   \int  (\psi_{2}(x)\pi_{h_n}(x) + \psi_{1{h_n}} (x))\pi(x)dx ,
=\pi(\psi_{1{h_n}}),
\end{align*}
Since $\int\varphi(x)dx=\pi(\psi_1)$, this yields
\begin{align*}
B_n &=\pi(  \psi_1- \psi_{1{h_n}}).
\end{align*}
 Because there exists $M$ such that $\psi_1$ belongs to $\EuScript H_1 (r\wedge s,M)$, applying Lemma \ref{lemma:regularizationrates} gives a bound in $h_n^{r\vee \min(r,s)}=h_n^r$ for $B_n$.

\paragraph{Step~4.} $R_{1,n}=O_{\mathbb P}(h_n^{2r}+n^{-1}h_n^{-d})$.

By Corollary~\ref{prop:coro:lowerbound_density}, and because
$n (\alpha_0 (l_n-1))^{-1} = O_{\mathbb P}(1)$, we get
\begin{align*}
R_{1,n}& \leqslant   O_\mathbb {P}(1)\left\{ n^{-1} \sum_{i=1}^n (\pi(X_i) -\widehat \pi_i)^2 \right\}  \\
&\leqslant  O_\mathbb {P}(1)\left\{ n^{-1} \sum_{i=1}^n(\pi(X_i)-\pi_{h_n}(X_i))^2+(\pi_{h_n}(X_i)-\widehat \pi_i)^2\right\}.
   \end{align*}
We compute the expectation of 
the first term inside the brackets. By Lemma \ref{lemma:regularizationrates},  
we obtain a bound  $O_{\mathbb P}(h_n^{2r})$. To treat the second term inside the bracket, 
denote by $J^{(-i)}= \{1\leqslant k\leqslant l_n-1 \ : \  i\notin B_k \}$, 
$l(i)=\{k\in \mathbb N : \ i\in B_k\}$ and $K{(i,B)}= \sum_{j\in B}K_{h_n}(X_i-X_j)$, write
($r_{1,n}$ and $r_{2,n}$ are specified below)
\begin{align*}
&\sum_{i=1}^n (\pi_{h_n}(X_i) -\widehat \pi_i)^2\\
&~=\sum_{i=\theta_a(1)+1}^{\theta_a(l_n)} (\pi_{h_n}(X_i) -\widehat \pi_i)^2 + r_{1,n}\\
&~=\quad r_{1,n}\\
&\quad~+\sum_{i=\theta_a(1)+1}^{\theta_a(l_n)}  \Big(\pi_{h_n}(X_i) -(\alpha_0(l_n-2))^{-1}\big\{K(i,B_0)+K{(i,B_{l_n})}+K{(i,B_{l(i)}) } + \sum_{k\in J^{(-i)}} K{(i,B_k)}  \big\} \Big)^2\\
&~\leqslant  2(l_n-2)^{-2}  \sum_{i=\theta_a(1)+1}^{\theta_a(l_n)}  \Big( \sum_{k\in J^{(-i)}}  \{\pi_{h_n}(X_i) - \alpha_0^{-1}K{(i,B_k)}\}   \Big)^2+ r_{1,n}+r_{2,n}\\
&~\leqslant   2(l_n-2)^{-2}   \sum_{i=1}^n  \Big( \sum_{k\in J^{(-i)}} \{\pi_{h_n}(X_i) - \alpha_0^{-1}K{(i,B_k)}\}   \Big)^2 + r_{1,n}+r_{2,n},
   \end{align*}
with
\begin{eqnarray*}
r_{1,n}&= & \sum_{i=1}^n (\pi_{h_n}(X_i) -\widehat \pi_i)^2 (\mathbb{1}_{\{i\leqslant \theta_a(1)\}}+\mathbb{1}_{\{i>\theta_a(l_n)\}} )\\
&\leqslant &\sup_{y\in\mathbb R^d}|\widehat \pi (y) - \pi_{h_n}(y)| (\theta_a(1)+ \Delta_{l_n} )  
\end{eqnarray*}
and 
\begin{align*}
r_{2,n}&=  2(\alpha_0 (l_n-2))^{-2} \sum_{i=\theta_a(1)+1}^{\theta_a(l_n)}   \Big(K(i,B_0)+K(i,B_{l_n})+K(i,B_{l(i)}) \Big)^2\\
&\leqslant 2(\alpha_0 (l_n-2))^{-2}  K_\infty^2 h_n^{-2d} \,\sum_{i=\theta_a(1)+1}^{\theta_a(l_n)} (\Delta_0+ \Delta_{l_n}+ \Delta_{l(i)})^2 \\
&\leqslant 6(\alpha_0 (l_n-2))^{-2}  K_\infty^2 h_n^{-2d}\, \sum_{i=\theta_a(1)+1}^{\theta_a(l_n)} ( \Delta_0^2+ \Delta_{l_n}^2+ \Delta_{l(i)}^2)\\
&\leqslant 6(\alpha_0 (l_n-2))^{-2}  K_\infty^2 h_n^{-2d}\Big( n( \Delta_0^2+ \Delta_{l_n}^2) +\sum_{i=\theta_a(1)+1}^{\theta_a(l_n)}   \Delta_{l(i)}^2\Big)
\end{align*}
Because $\sum_{i=\theta_a(1)+1}^{\theta_a(l_n)}   \Delta_{l(i)}^2 =  \sum_{k=1}^{l_n-1}  \Delta_k^3\leqslant \sum_{k=1}^{n}  \Delta_k^3$, 
we find that the above term between parentheses has expectation of order 
$n ( \mathbb E _{\pi}  \theta_a^2+\mathbb E _{a} \theta_a ^2 + \mathbb E _{a}\theta_a^3)$. 
As by Lemma \ref{lemma:tau} and Assumption \ref{ash6:bandwidth}, the previous expectations 
are bounded, it follows that $r_{2,n} = O_{\mathbb P_\pi} (n (nh_n^d )^{-2} )$
has a contribution $O_{\mathbb P_\pi} ( (nh_n^d )^{-2} )$ to $R_{1,n}$. 
Moreover, we have that $r_{1,n} =o_{\mathbb P_\pi} (1 )$ by 
Theorem \ref{prop:unif_conv_density}, which gives a contribution 
$o_{\mathbb P_\pi} (n^{-1} )$ to $R_{1,n}$. 
Regarding the objective of the present step, $r_{1,n}$ and $r_{2,n}$ 
are negligible, so that, we can concentrate on 
\begin{align*}
 \sum_{i=1}^n  \Big( \sum_{k\in J^{(-i)}} \{\pi_{h_n}(X_i) - \alpha_0^{-1}K{(i,B_k)}\}   \Big)^2.
\end{align*}
We use the independence between the blocks to compute 
\begin{align*}
\mathbb E_\pi\sum_{i=1}^n  \left( \sum_{k\in J^{(-i)}} \{\pi_{h_n}(X_i) - \alpha_0^{-1}K(i,B_k)\}   \right)^2 &= n\mathbb E_\pi  \left(\sum_{k=1 }^{l_n-2} \{\pi_{h_n}(X_0) - \alpha_0^{-1}K(0,B_k)\} \right)^2\\
&\leqslant n \mathbb E_\pi \left( \max_{1\leqslant l\leqslant n}\Big| \sum_{k=1 }^{l} \{\pi_{h_n}(X_0) - \alpha_0^{-1}K(0,B_k)\} \Big|\right)^2.
\end{align*} 
Since $l\mapsto \sum_{k=1 }^{l} \{\pi_{h_n}(X_0) - \alpha_0^{-1}K(0,B_k)\}$ is a martingale, we get from Doob's inequality that
 \begin{align*}
\mathbb E_\pi \sum_{i=1}^n  \left( \sum_{k\in J^{(-i)}} \{\pi_{h_n}(X_i) - \alpha_0^{-1}K(i,B_k)\}   \right)^2 
&\leqslant 4n \mathbb E_\pi \left[ \left(\sum_{k=1 }^{n} \{\pi_{h_n}(X_0) - \alpha_0^{-1}K(0,B_k)\}\right)^2\right]\\
 &= 4n^2 \mathbb E_\pi \left[ \{\pi_{h_n}(X_0) - \alpha_0^{-1}K(0,B_1)\}^2\right] \\
 &\leqslant 4n^2\alpha_0^{-2} \mathbb E_\pi[K(0,B_1)^2 ]\\
 &\leqslant 4n^2\alpha_0^{-2} h_n^{-d} K_\infty  \mathbb {E}_\pi \Big\{\Delta_1 \sum_{j\in B_1} |K_{0j}| \Big\}.
  \end{align*}
Here we use the independence between $B_1$ and $X_0$ to write
  \begin{align*}
\mathbb {E}_\pi \Big\{\Delta_1 \sum_{j\in B_1} |K_{0j}| \Big\} &=   \mathbb {E}_a \Big\{\theta_a \sum_{j\in B_1}\int \pi (X_j-h_nu) |K(u)| du\Big\}\\
 &\leqslant   \pi_{\infty} \int |K(x)| dx\, \mathbb {E}_a [\theta_a^2 ].
 \end{align*}
This leads to a contribution $O_{\mathbb P_\pi} (n^{-1}h_n^{-d} )$ to $R_{1,n}$.

\paragraph{Step~5.} $R_{2,n}=O_{\mathbb P}(n^{-1} h_n^{-d})$.

Recall that 
\begin{align*}
R_{2,n}= &-( \alpha_0^2 (l_n-1)(l_n-2)  )^{-1}  \big(H_{00}+H_{l_nl_n}+H_{0l_n}^* +\sum_{k=1}^{l_n-1} \{H_{0k}^*+ H_{l_nk}^*+H_{kk} \}\big) \\
&+ (\alpha_0 (l_n-1))^{-1} 2 ( G_0+ G_{l_n}) ,
\end{align*}
with $H_{kl}= \sum_{i\in B_k ,j\in B_l}  \psi_{2}(X_i)K_{ij}$, $H_{kl}^*=H_{kl}+H_{lk}$ and $G_k =\sum_{i\in B_k } \psi_{1}(X_i)$. First, the boundedness of $\psi_1$ yields
\begin{align*}
\mathbb E_\pi  |G_0 |\leqslant \psi_{1,\infty} \mathbb {E}_{\pi}\theta_a\qquad 
\text{and} \qquad \mathbb E_\pi |G_{l_n} |\leqslant  \psi_{1,\infty}  \mathbb {E}_{a}\theta_a,
\end{align*}
leading to a contribution of order $O_{\mathbb P_\pi}(n^{-1})\ll O_{\mathbb P_\pi}((nh_n^d)^{-1})$. Second, we have
\begin{align*}
\mathbb E_{\pi} |\sum_{k=1}^{l_n-1} (H_{0k}^*+ H_{l_nk}^* + H_{kk}) |&\leqslant n\mathbb E_{\pi}  (|H_{01}^*|+|H_{l_n1}^*| + |H_{11}|)\\
&\leqslant nK_\infty\psi_{2,\infty}  h_n^{-d} (\mathbb {E}_a\theta_a \mathbb {E}_\pi\theta_a+(\mathbb {E}_a\theta_a )^2+ \mathbb {E}_a\theta_a^2)  ,
\end{align*}
involving a $(nh_n^d)^{-1}$ in the $O_{\mathbb{P}}$. 
In a similar fashion, the term $H_{00}+H_{l_nl_n}+H_{0l_n}^* $ has order $(n^2h_n^d)^{-1}\ll (nh_n^d)^{-1}$.
\end{proof}

\section{Numerical experiments}\label{s5}

\subsection{Estimation algorithm}
\label{ss:algo}

Let us first recall the framework investigated in the paper. We consider the estimation of the integral of a function $\varphi$ over $Q$ from a dataset $(X_i,\varphi(X_i))_{1\leq i\leq n}$ when the $X_i$'s form a Markov chain. The estimator $\widehat I_{\text{ks}}$ of $I_0=\int_Q\varphi(x) dx$ is given by 
\begin{align*}
\widehat I_{\text{ks}} =  n^{-1} \sum_{i=1}^n \frac{\varphi(X_i)}{\widehat \pi(X_i) }.
\end{align*}

As noticed in \cite{delyon2015} for independent data, the crucial factor for the estimation of $I_0$ is to select the optimal bandwidth parameter $h_n$ appearing in the estimator $\widehat{\pi}$ of the design distribution given by 
\begin{align*}
\widehat \pi (x) = (nh_n^d)^{-1} \sum_{i=1} ^n K( (x-X_i)/h_n ), \qquad x\in \mathbb R^d.
\end{align*} 
In this paper, we propose to use the multivariate plug-in bandwidth selection developed in \cite{Chacon2010}. More precisely, we exploit the implementation of this algorithm in the \verb+R+ package \verb+ks+ (see \cite{duong:2007} for a presentation of a preliminary version). It ensures better results than the $\varphi$-based method proposed in \cite{delyon2015} in both the independent and the Markov frameworks. Moreover, this method is simpler because it provides an optimal bandwidth that only depends on the design (and not on $\varphi$) contrary to the aforementioned competitive strategy. This is a particularly interesting procedure to integrate several functions from the same design points, e.g., temperature and salinity, because it requires only one selection of the bandwidth. We strongly recommend to use this method rather than the one proposed in the previous paper \cite{delyon2015}. Such a choice of the bandwidth does not fit the theoretical framework of Theorem \ref{th:bigOpth} (as it depends on the design points) but does not require any knowledge on the regularity of the functions $\varphi $ and $f$.

\cite{delyon2015} introduce a corrected version $\widehat{I}_\text{ks}^c$ of the integral estimator $\widehat{I}_\text{ks}$ that presents both smaller bias and variance in numerical experiments,
$$\widehat{I}_\text{ks}^c = \frac{1}{n}\sum_{i=1}^n \frac{\varphi(X_i)}{\widehat{\pi}(X_i)}\left(1-\frac{\widehat{v}(X_i)}{\widehat{\pi}(X_i)^2}\right),$$
where
$$\widehat{v}(x) = \frac{1}{n(n-1)}\sum_{i=1}^n \left[\frac{1}{h_n^d} K\left(\frac{x-X_i}{h_n}\right)  -   \widehat{\pi}(x)\right]^2 ,\qquad x\in \mathbb R^d .$$
This new estimator has been chosen in order to make vanish the leading term in the expansion of the estimation error in the independent case. Function $\widehat{v}$ being positive, $\widehat{I}_\text{ks}^c$ is lower than $\widehat{I}_\text{ks}$ which tends to have a positive bias. In the sequel, we compute both $\widehat{I}_\text{ks}$ and $\widehat{I}_\text{ks}^c$ from the same bandwidth $h_n$ depending only on the design points and obtained as aforementioned.

\subsection{Simulation study}

We consider the following 3 models. For each of them, the function $\varphi$ will be integrated on its support given by $Q=[0,1]^d$.
\begin{itemize}
\item $\mathcal{M}_1$: $\varphi(x_1,\dots,x_d) = \prod_{i=1}^d \left[2\sin(\pi x_i)^2 \mathbb{1}_{[0,1]}(x_i)\right]$;
\item $\mathcal{M}_2$: $\varphi(x_1,\dots,x_d) = \prod_{i=1}^d \left[\frac{1+\pi^2}{\pi (1+\exp(1))}\sin(\pi x_i)\exp(x_i) \mathbb{1}_{[0,1]}(x_i)\right]$;
\item $\mathcal{M}_3$: $\varphi(x_1,\dots,x_d) = \prod_{i=1}^d \left[\frac{\pi}{2} \sin(\pi x_i) (1+\cos(5\pi x_i)) \mathbb{1}_{[0,1]}(x_i)\right]$.
\end{itemize}
For improved comparability, the normalizing constant of each model has been chosen in such a way that $I_0=1$. The one-dimensional shape of each model is presented in Figure \ref{fig:models}. The $3$ models are continuous but have their own features. $\mathcal{M}_1$ is symmetric centered on the center of $Q$, while $\mathcal{M}_2$ has a negative skewness. Finally $\mathcal{M}_3$ has 3 distinct modes. Consequently, one may expect that the models are somehow sorted by increasing difficulty in numerical integration.

\begin{figure}[h]
\centering\includegraphics[width=0.55\textwidth]{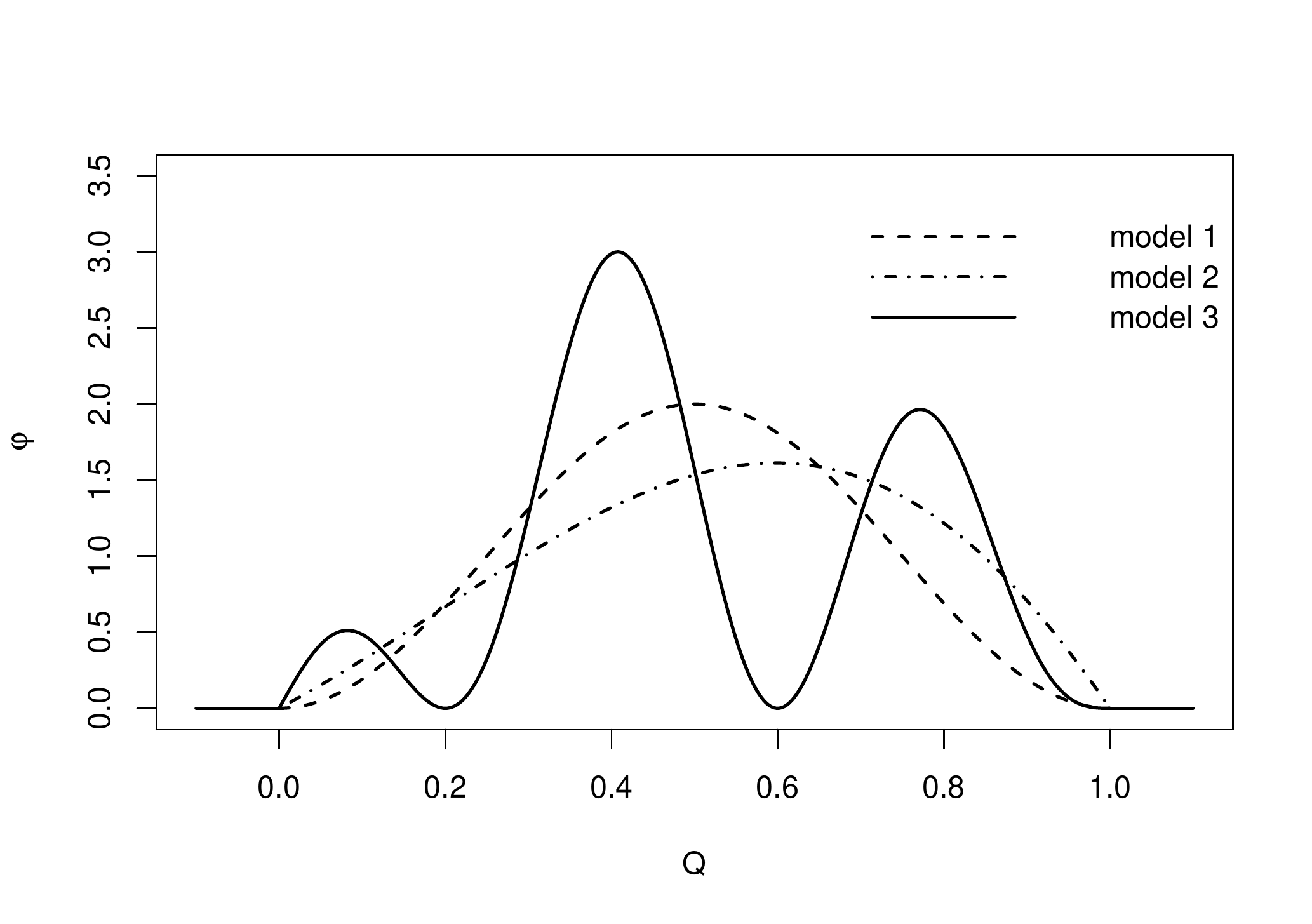}
\caption{Shape of function $\varphi$ for each model $\mathcal{M}_i$, $1\leq i\leq3$, in dimension $d=1$.}
\label{fig:models}
\end{figure}

For each model $\mathcal{M}_i$, $1\leq i\leq3$, we have computed the estimator and its corrected version presented in section \ref{ss:algo} from independent design (data with uniform distribution on $Q$ denoted by $\mathcal{U}_Q$) 
and from Markov design. In the Markov case, the dataset is generated according to the Metropolis-Hastings algorithm 
with proposition kernel $$P_r(x,dy) = \mathcal{U}_{[x-\varepsilon,x+\varepsilon]^d}(dy),$$
with $\varepsilon=0.2$ and target measure $\mathcal{U}_{Q}$. The Markov chain that results from this Metropolis-Hastings algorithm satisfies (A\ref{ash6:bandwidth}). Indeed first note that, the kernel $P^6(x,dy)$ has a density which is lower bounded on $[0,1]^{2d}$ by a positive number
(because $6\varepsilon>1$ and $\varphi(x)/\varphi(y)$ is lower bounded on $[0,1]^{2d}$), i.e., starting from any $x$ and waiting enough time will guarantee that any region is attained by the chain with positive probability. In other words, the uniform Doeblin condition holds for the chains $(X_{mk+1})_{k\geqslant 1}$ with the Lebesgue measure. Applying Theorem 16.0.2, page 394, in \cite{meyn+t:2009}, we obtain that the return time to $A$ of this chain, which is larger than the return time of the initial chain, has an exponential moment.

Independent and Markov designs have thus not been generated according to the same simulation model but share the same distribution, which makes them comparable. This will allow us to evaluate how the Markovian dependency impacts the performance of the methods. For the sake of reference, we have also computed the Monte Carlo estimator
$$\widehat{I}_\text{mc} = n^{-1} \sum_{i=1}^n \frac{\varphi(X_i)}{\pi(X_i) },$$
which can only be done in a simulation study where the distribution $\pi$ is known, and not from real data. Furthermore, we have investigated various sample sizes ($n=500$, $n=1\,000$ and $n=2\,000$) and different dimensions ($d=1$, $d=2$ and $d=3$). All the numerical results over $50$ independent replicates are provided in Figures \ref{fig:boxplots:model1} (model $\mathcal{M}_1$), \ref{fig:boxplots:model2} (model $\mathcal{M}_2$) and \ref{fig:boxplots:model3} (model $\mathcal{M}_3$). In order to make this numerical study reproducible, the \verb+R+ scripts implemented to generate datasets and estimate the integrals of interest are available at the webpage \url{http://iecl.univ-lorraine.fr/~Romain.Azais/}.

\begin{figure}[p]
\centering
\begin{tabular}{cc}
   \includegraphics[width=0.5\textwidth]{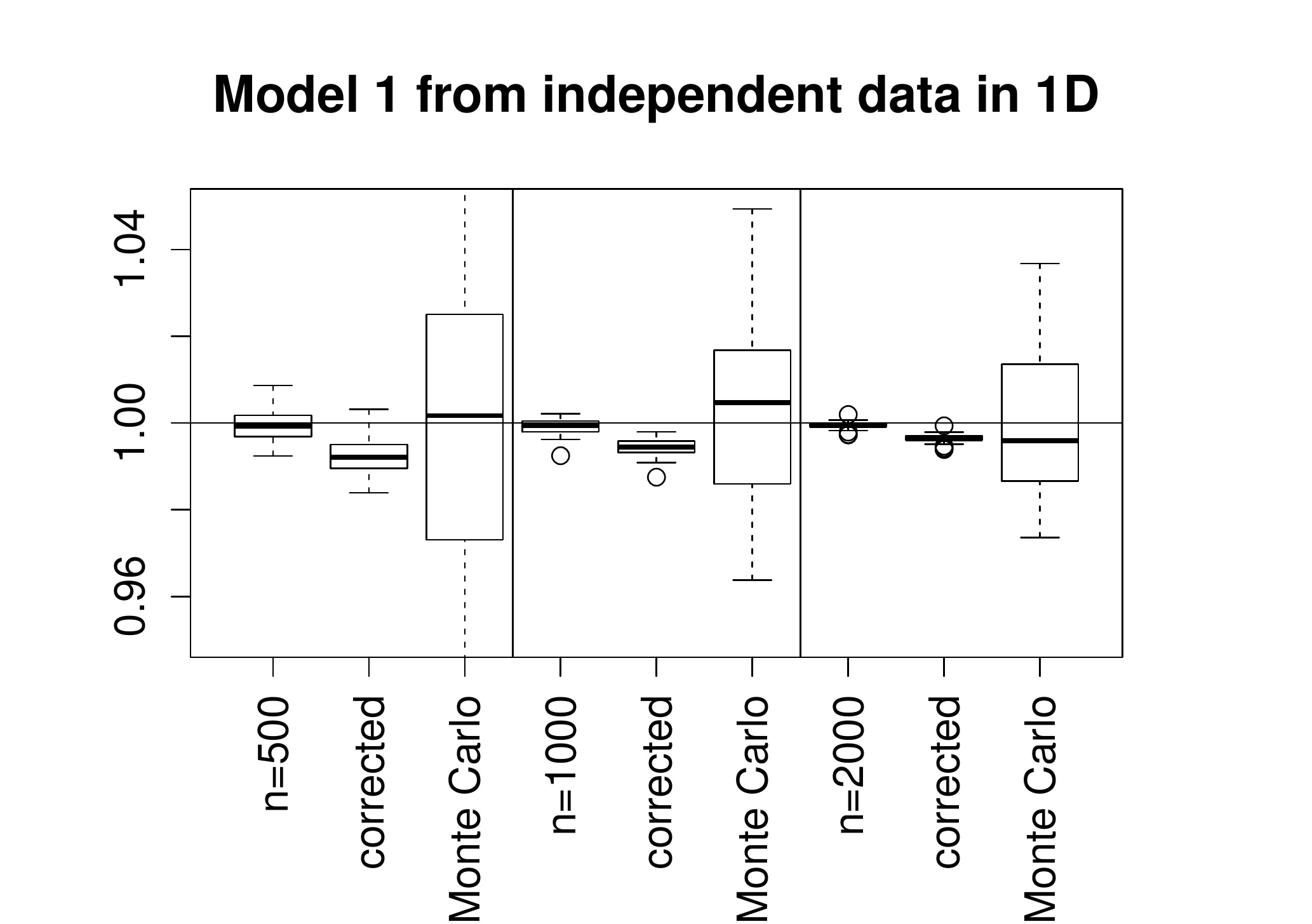} & \includegraphics[width=0.5\textwidth]{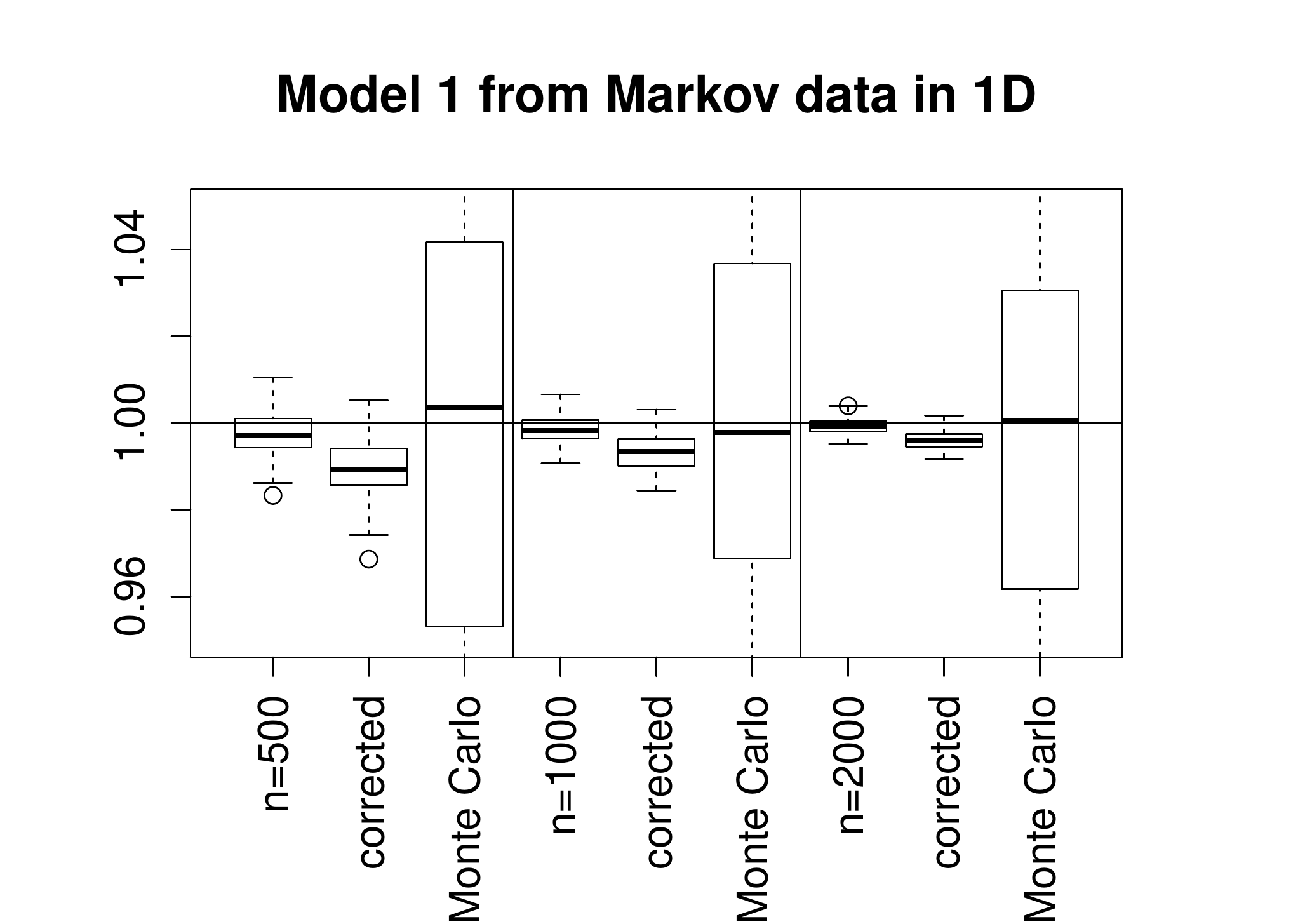} \\
   \includegraphics[width=0.5\textwidth]{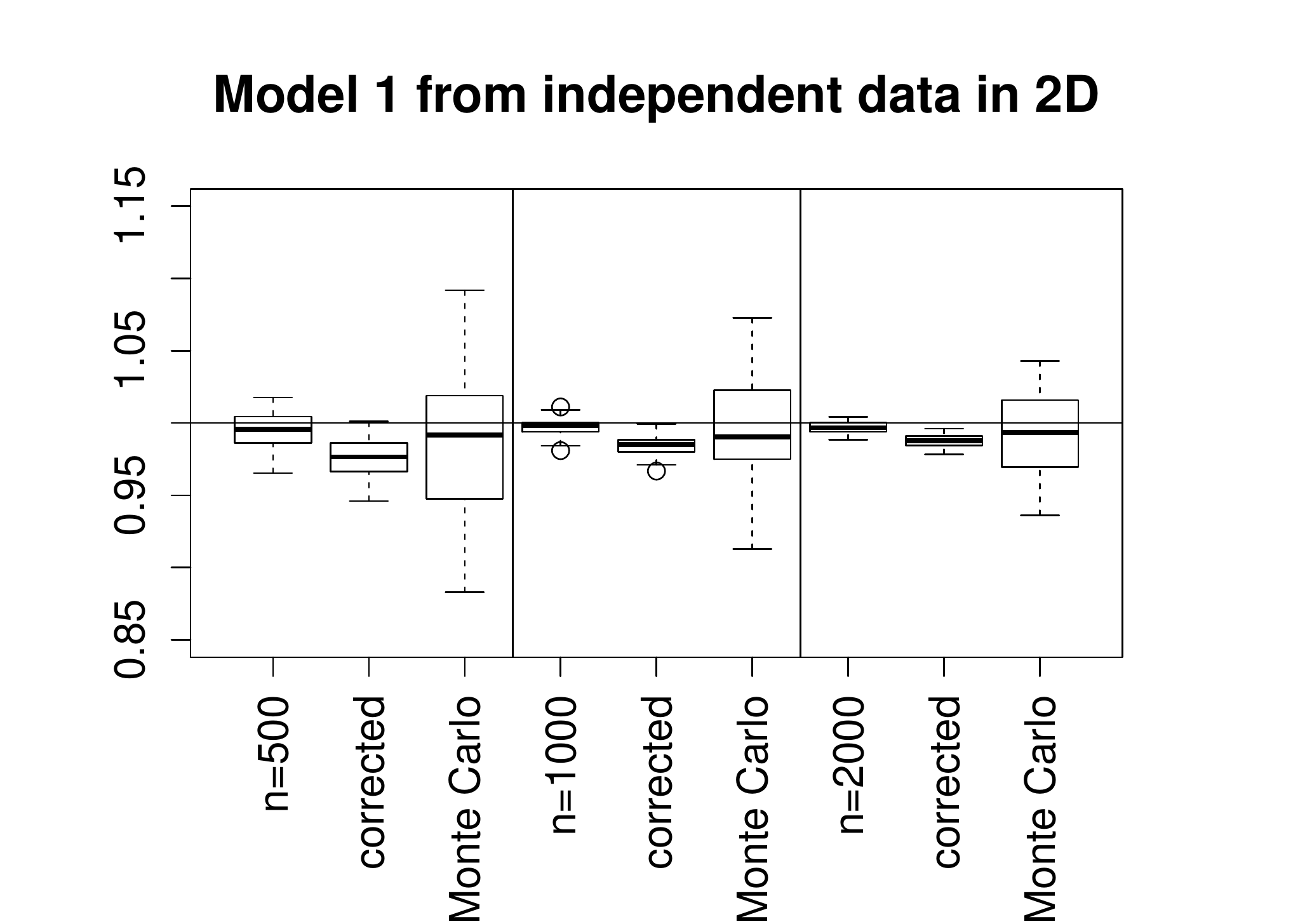} & \includegraphics[width=0.5\textwidth]{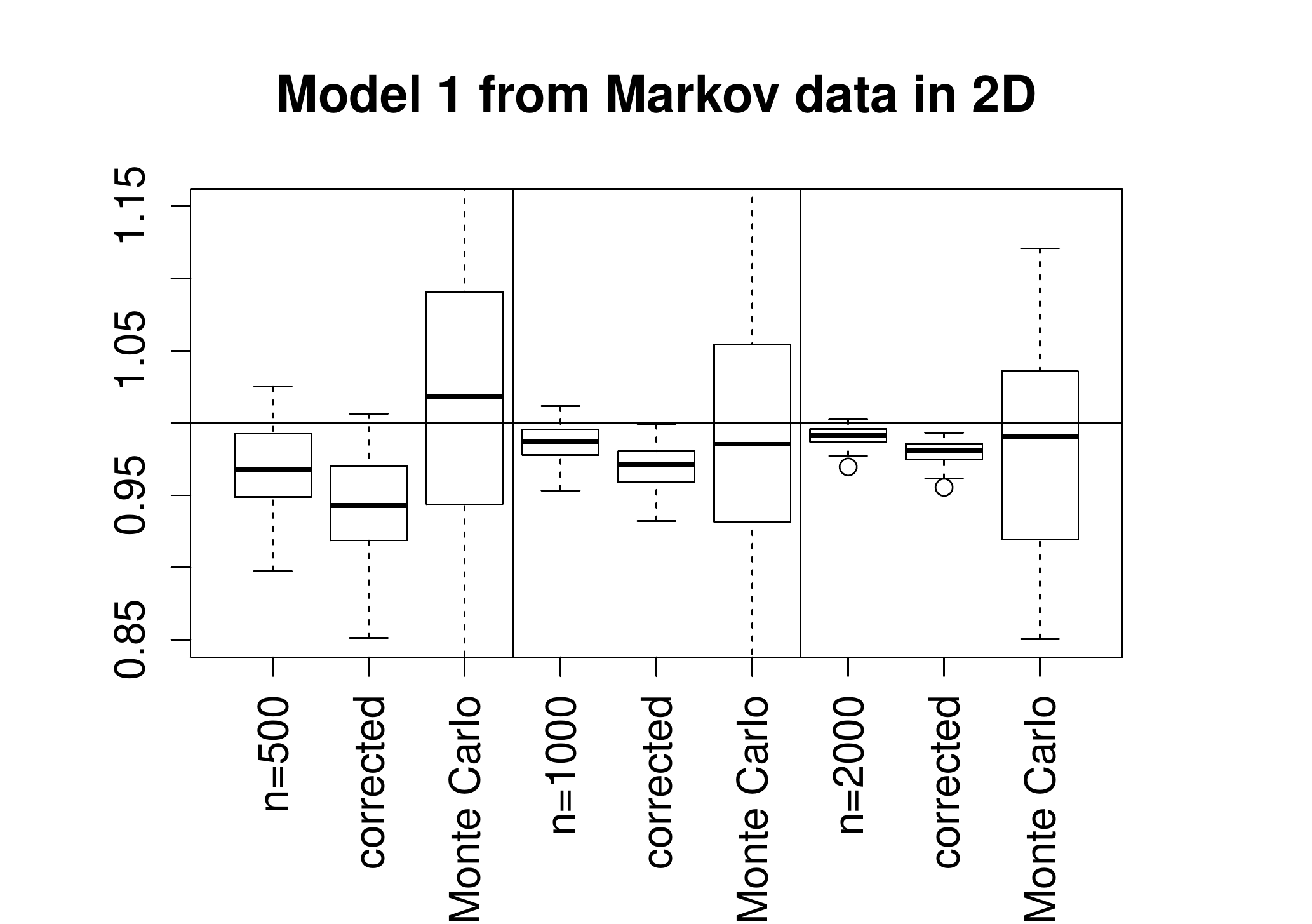} \\ 
   \includegraphics[width=0.5\textwidth]{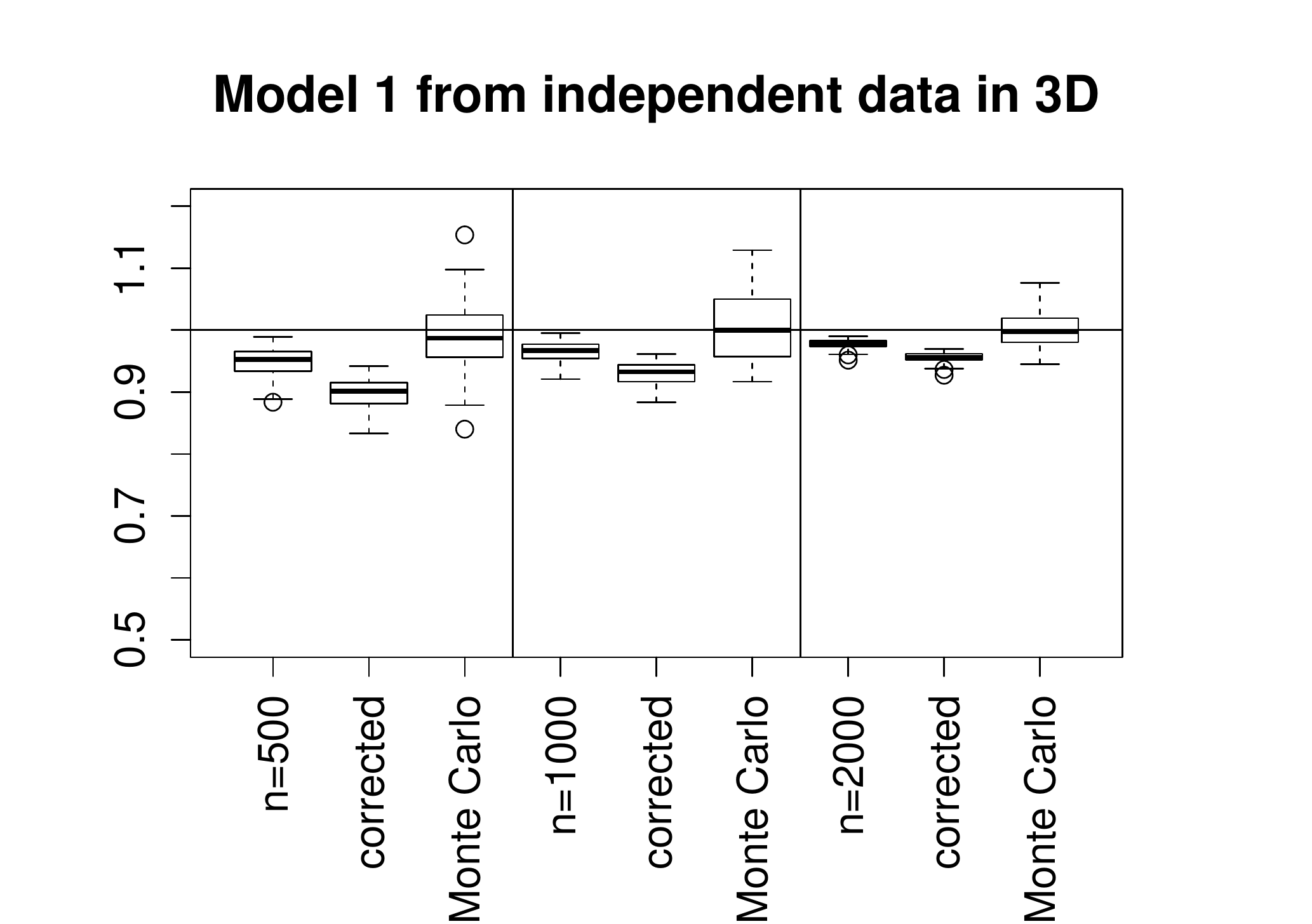} & \includegraphics[width=0.5\textwidth]{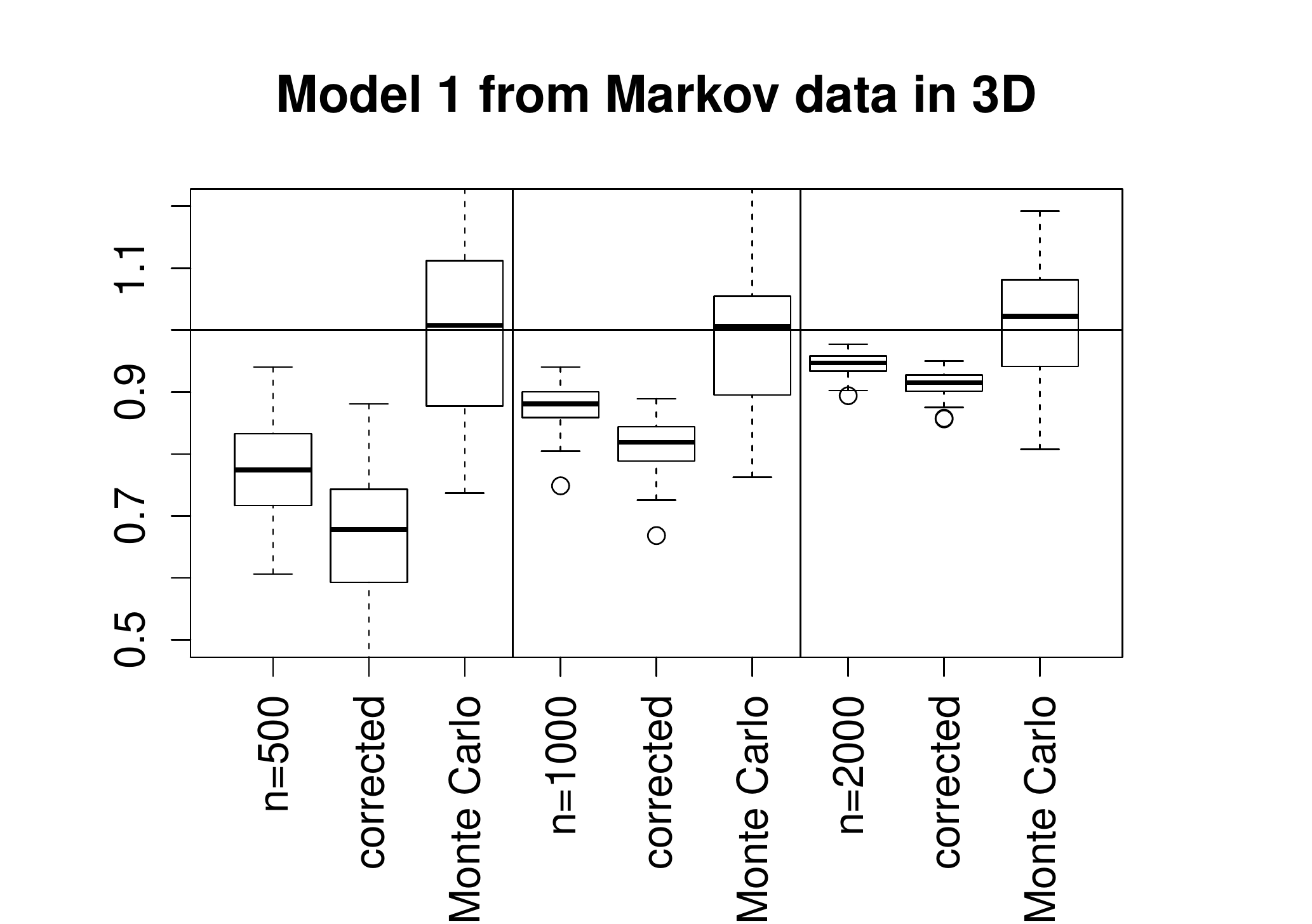} 
\end{tabular}
\caption{Boxplots of $\widehat{I}_{\text{ks}}$, $\widehat{I}_{\text{ks}}^c$ and $\widehat{I}_{\text{mc}}$ computed from $50$ replicates for model $\mathcal{M}_1$ in dimension $d=1$ (top), $d=2$ (middle) and $d=3$ (bottom) from independent data (left) and Markov data (right).}
\label{fig:boxplots:model1}
\end{figure}

\begin{figure}[p]
\centering
\begin{tabular}{cc}
   \includegraphics[width=0.5\textwidth]{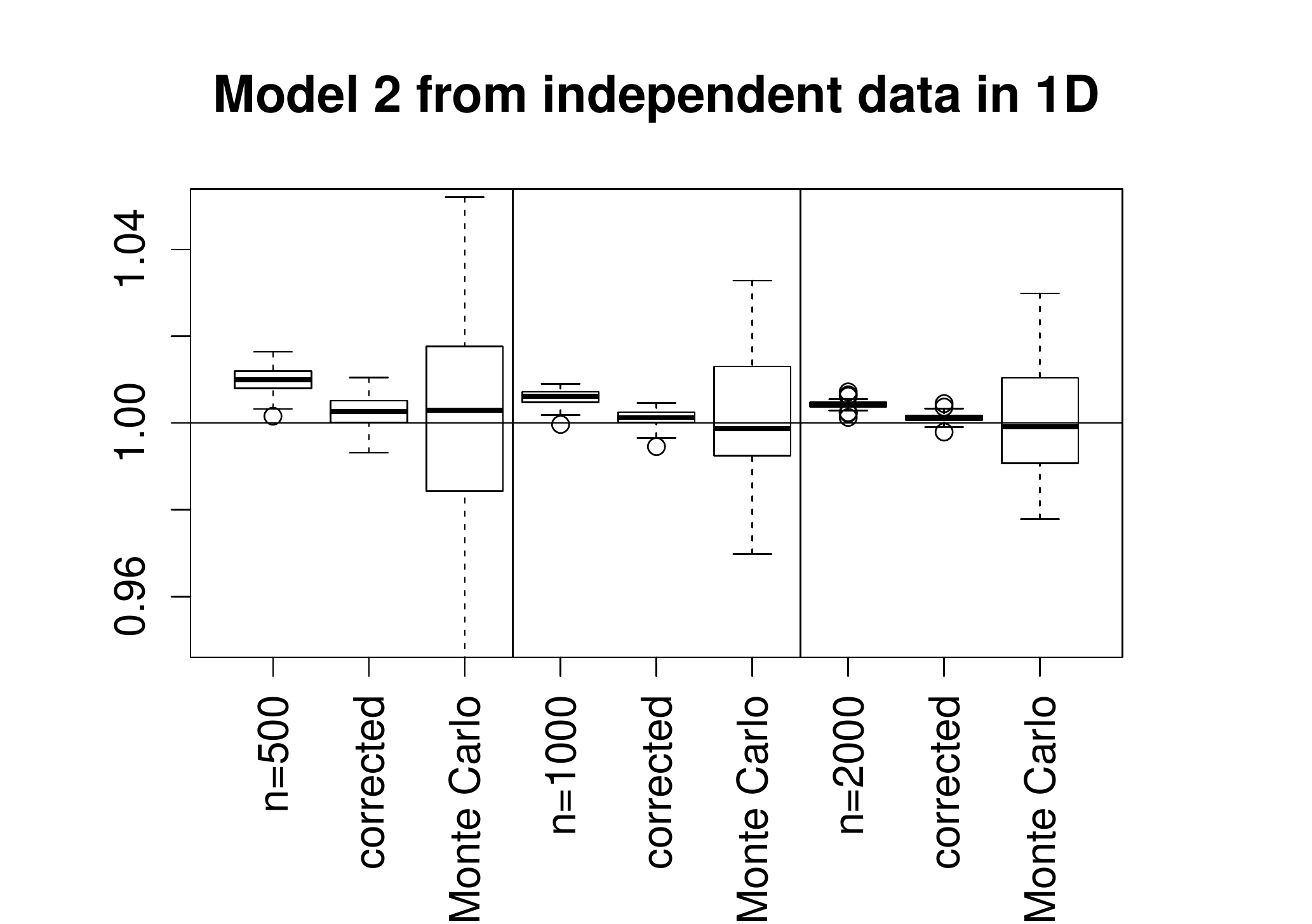} & \includegraphics[width=0.5\textwidth]{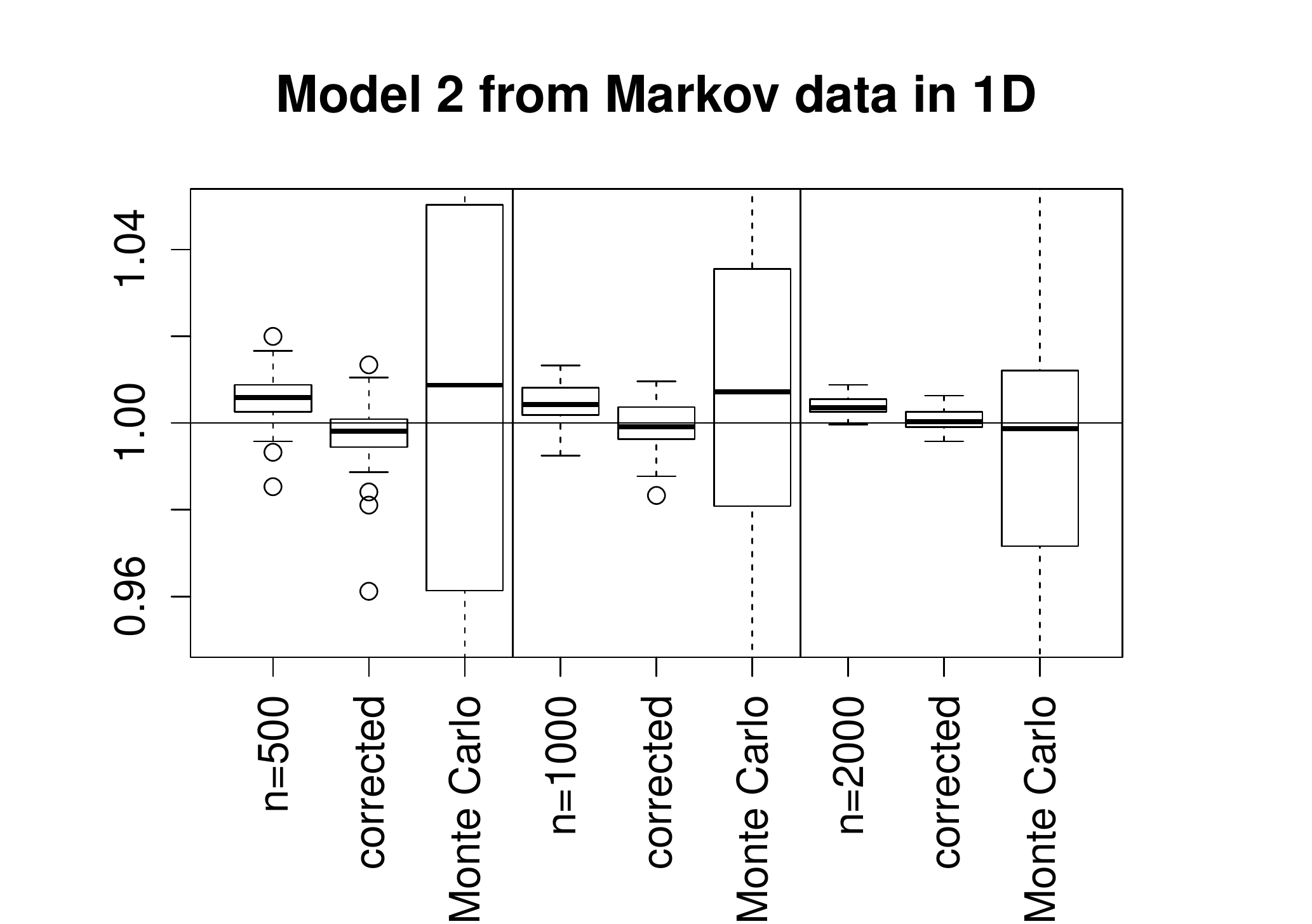} \\
   \includegraphics[width=0.5\textwidth]{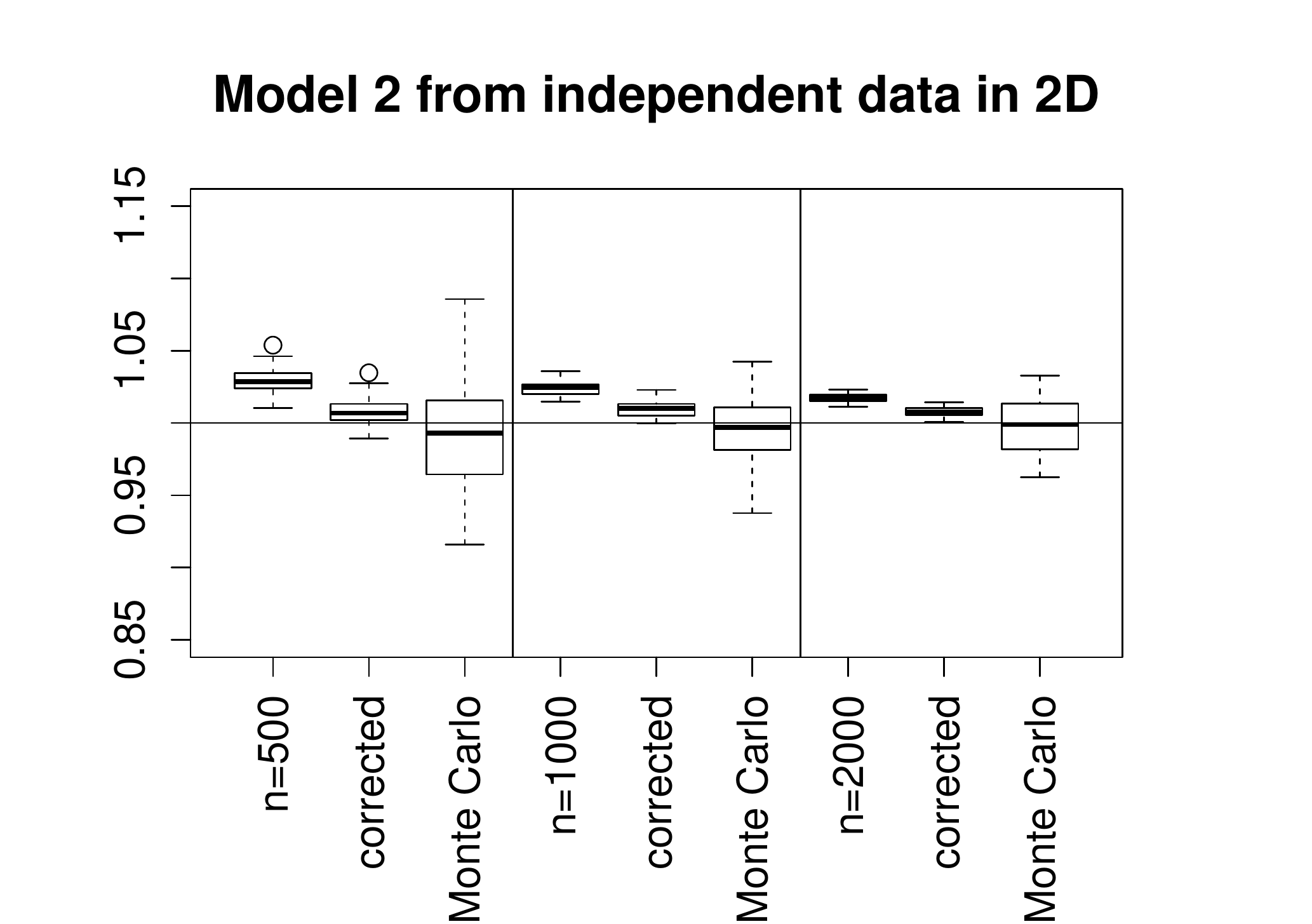} & \includegraphics[width=0.5\textwidth]{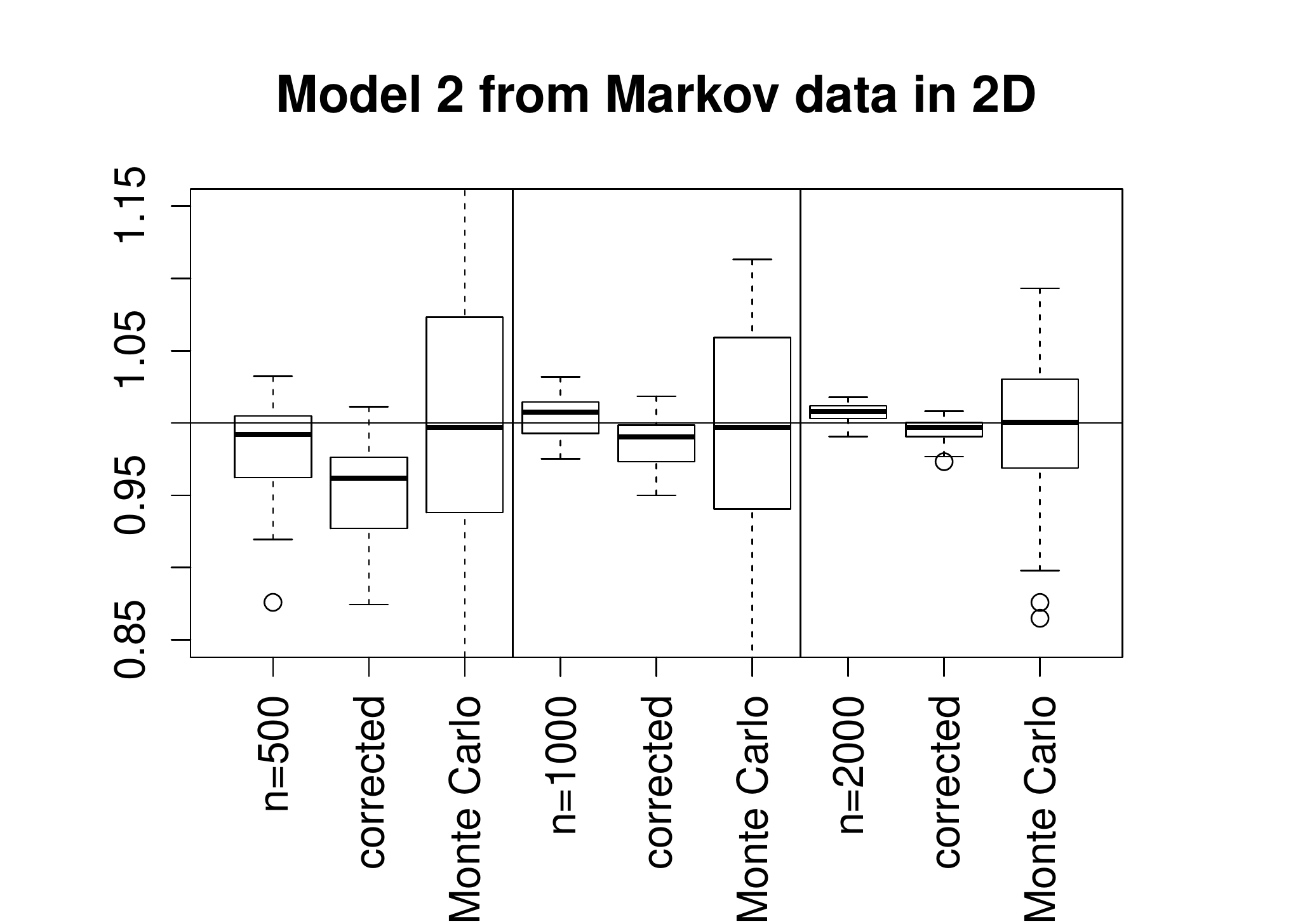} \\ 
   \includegraphics[width=0.5\textwidth]{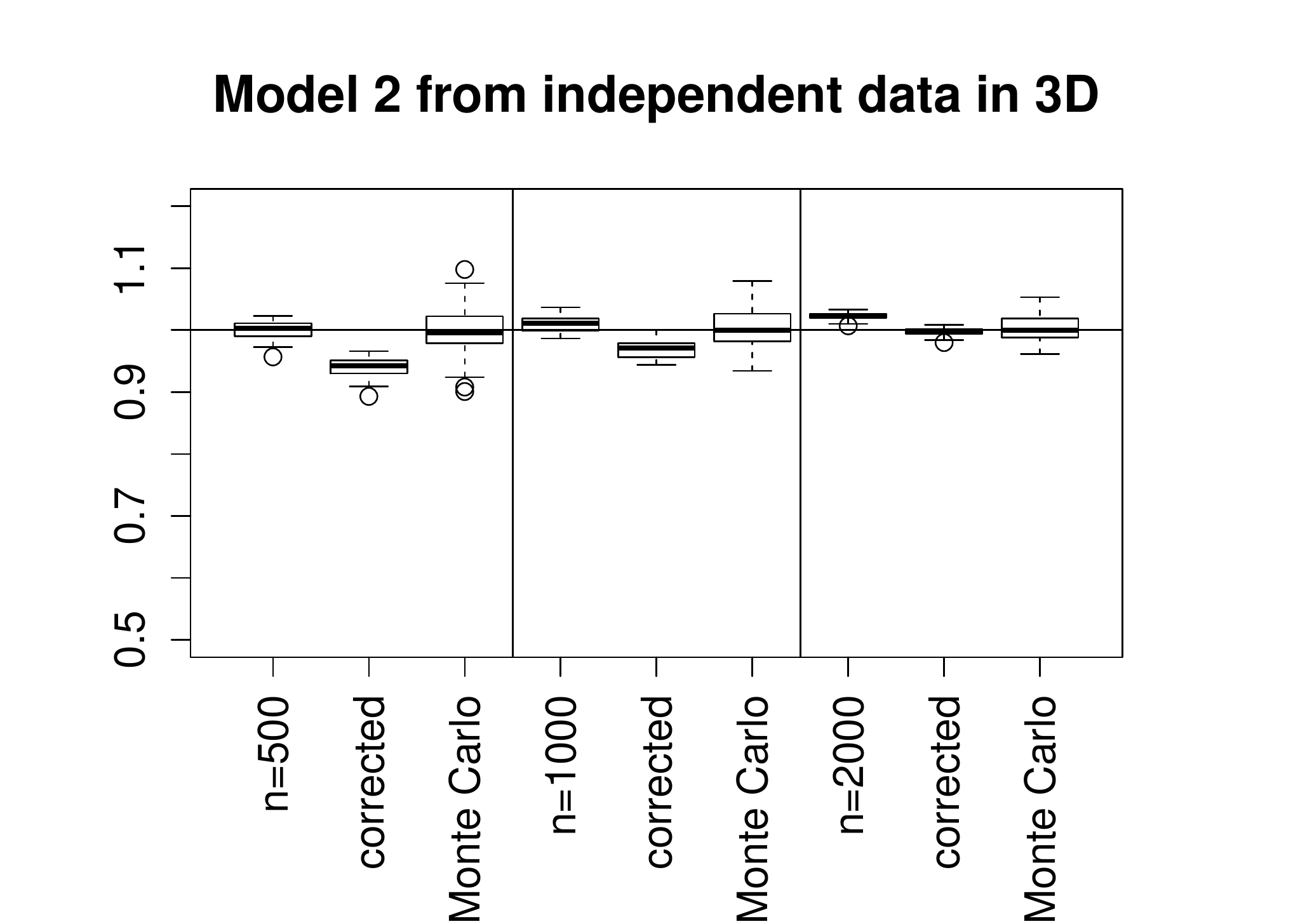} & \includegraphics[width=0.5\textwidth]{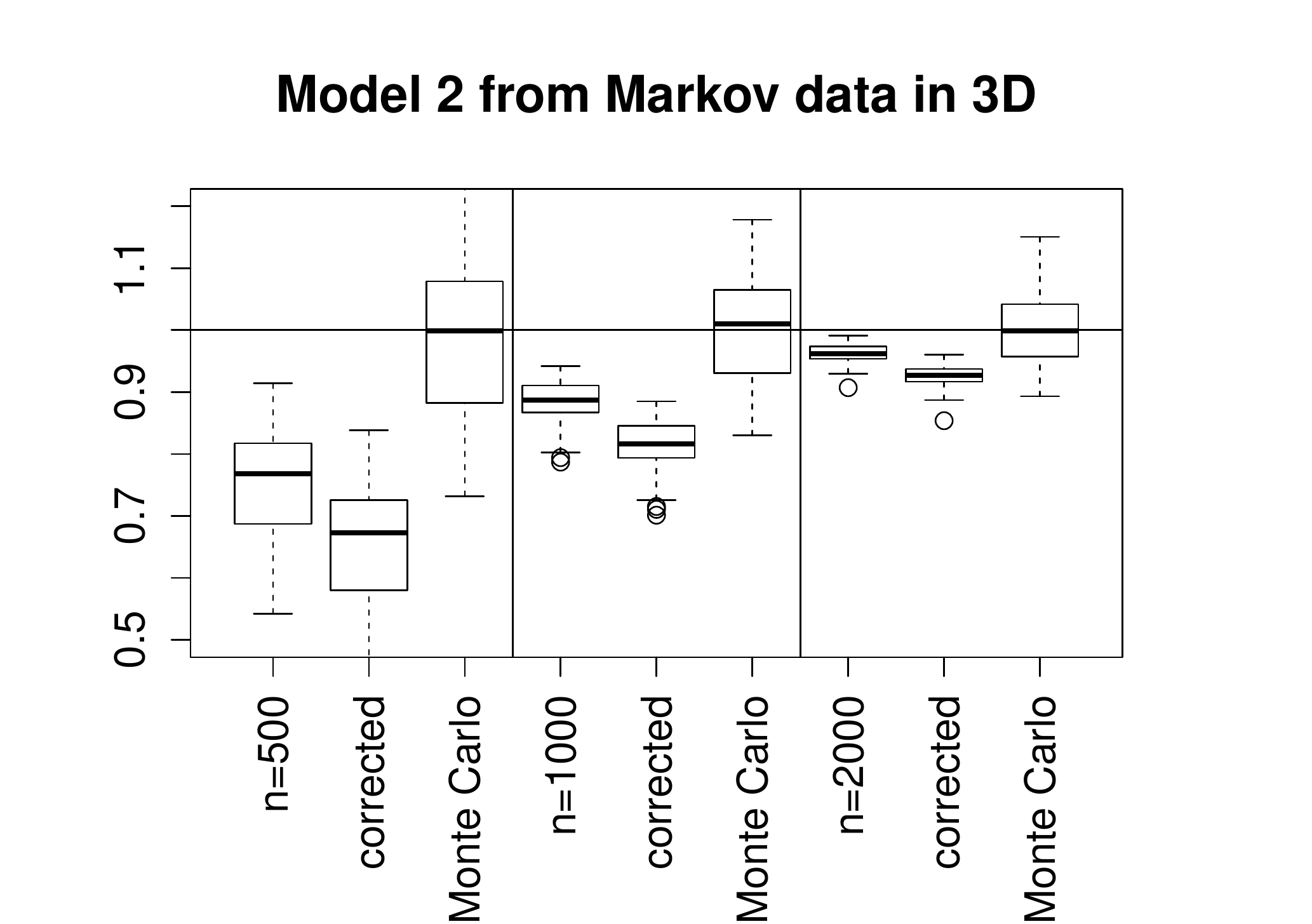} 
\end{tabular}
\caption{Boxplots of $\widehat{I}_{\text{ks}}$, $\widehat{I}_{\text{ks}}^c$ and $\widehat{I}_{\text{mc}}$ computed from $50$ replicates for model $\mathcal{M}_2$ in dimension $d=1$ (top), $d=2$ (middle) and $d=3$ (bottom) from independent data (left) and Markov data (right).}
\label{fig:boxplots:model2}
\end{figure}

\begin{figure}[p]
\centering
\begin{tabular}{cc}
   \includegraphics[width=0.5\textwidth]{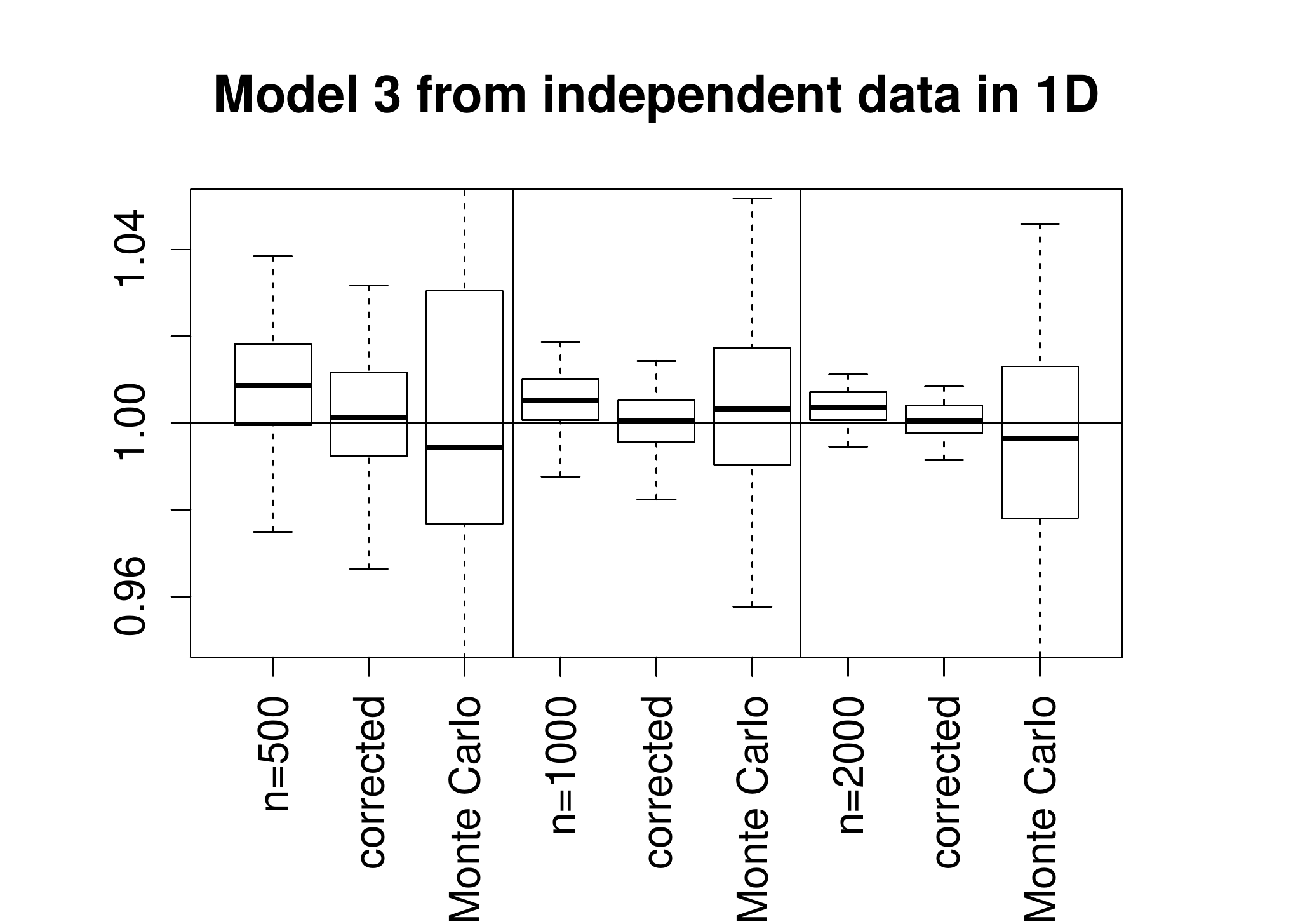} & \includegraphics[width=0.5\textwidth]{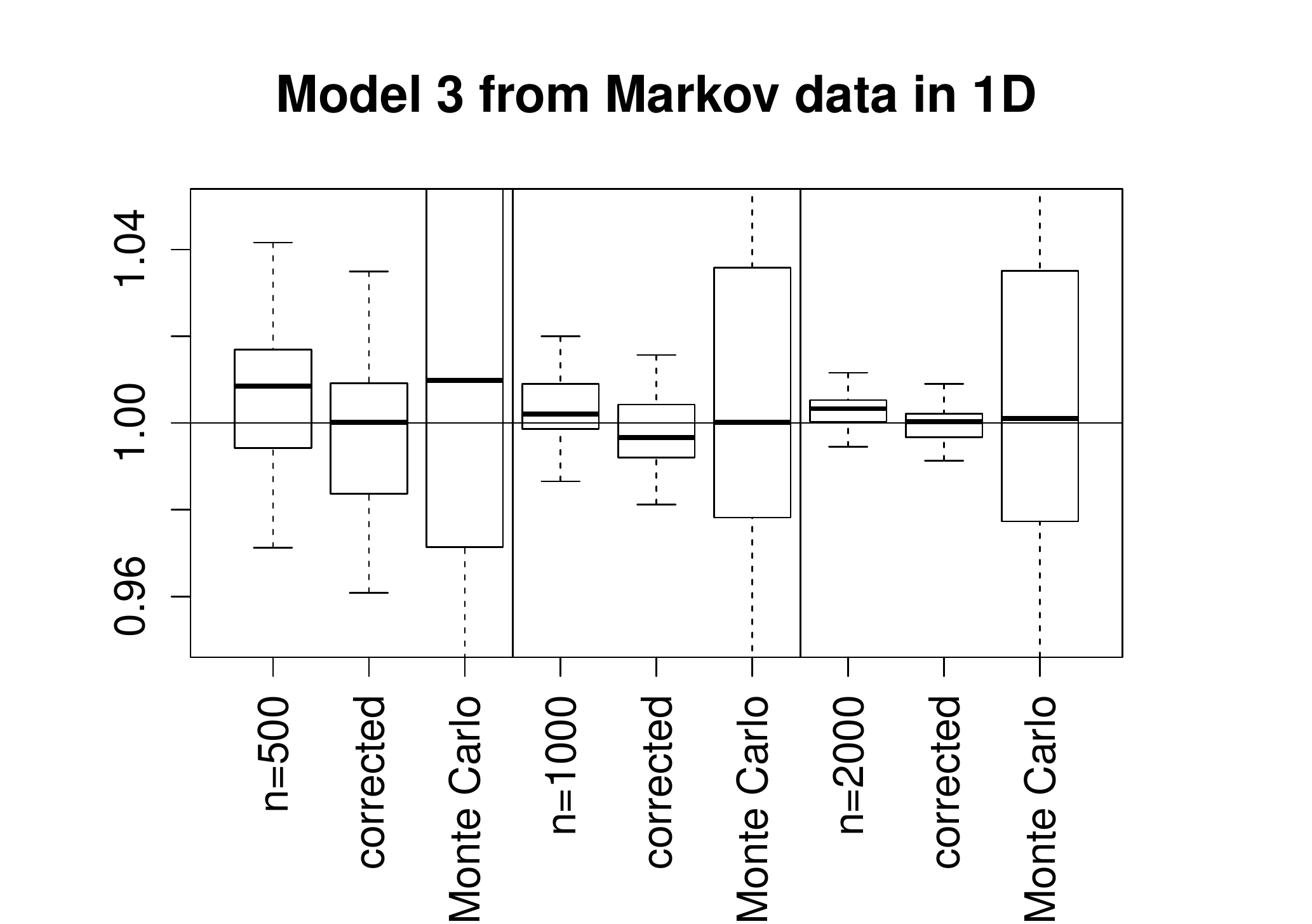} \\
   \includegraphics[width=0.5\textwidth]{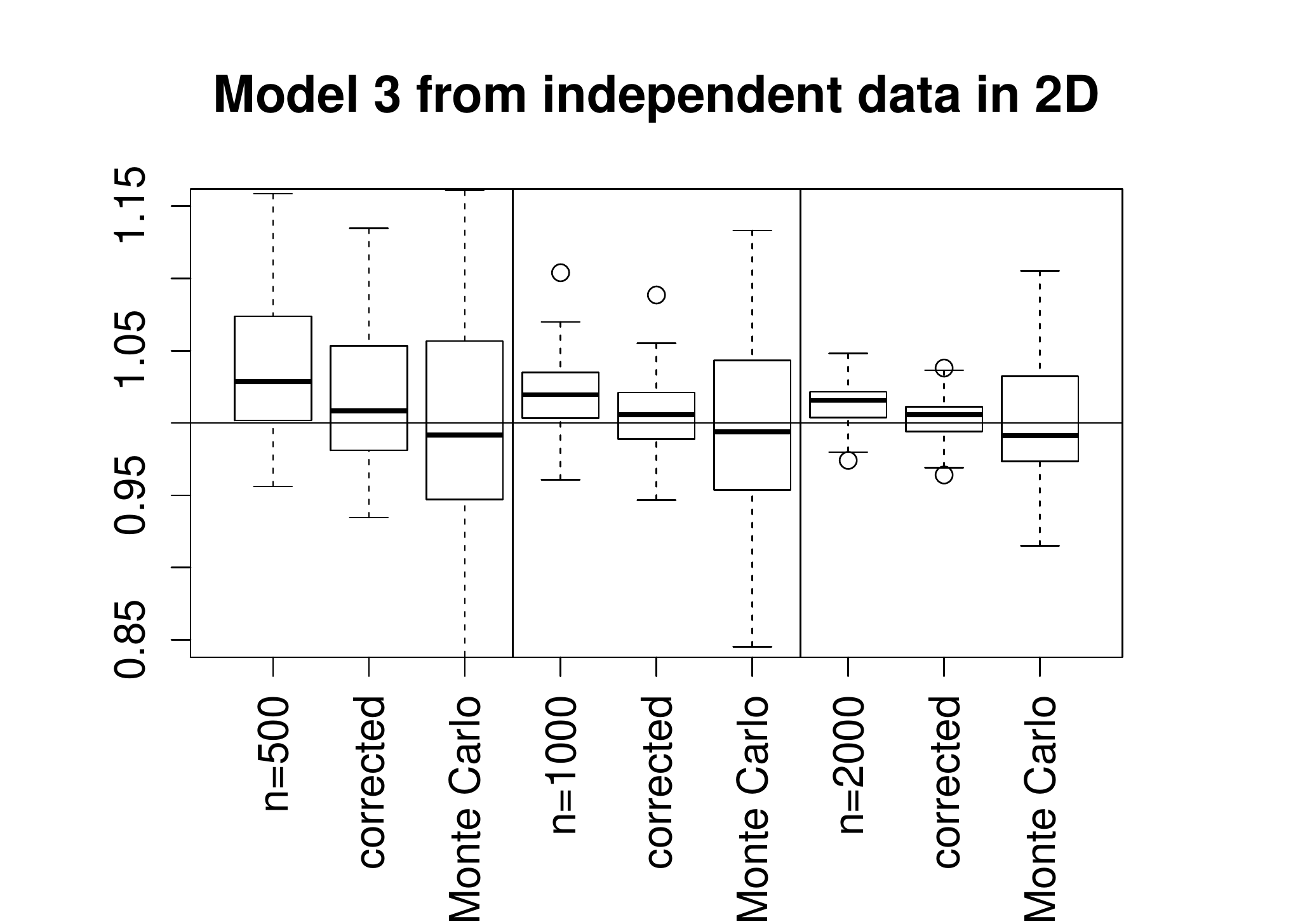} & \includegraphics[width=0.5\textwidth]{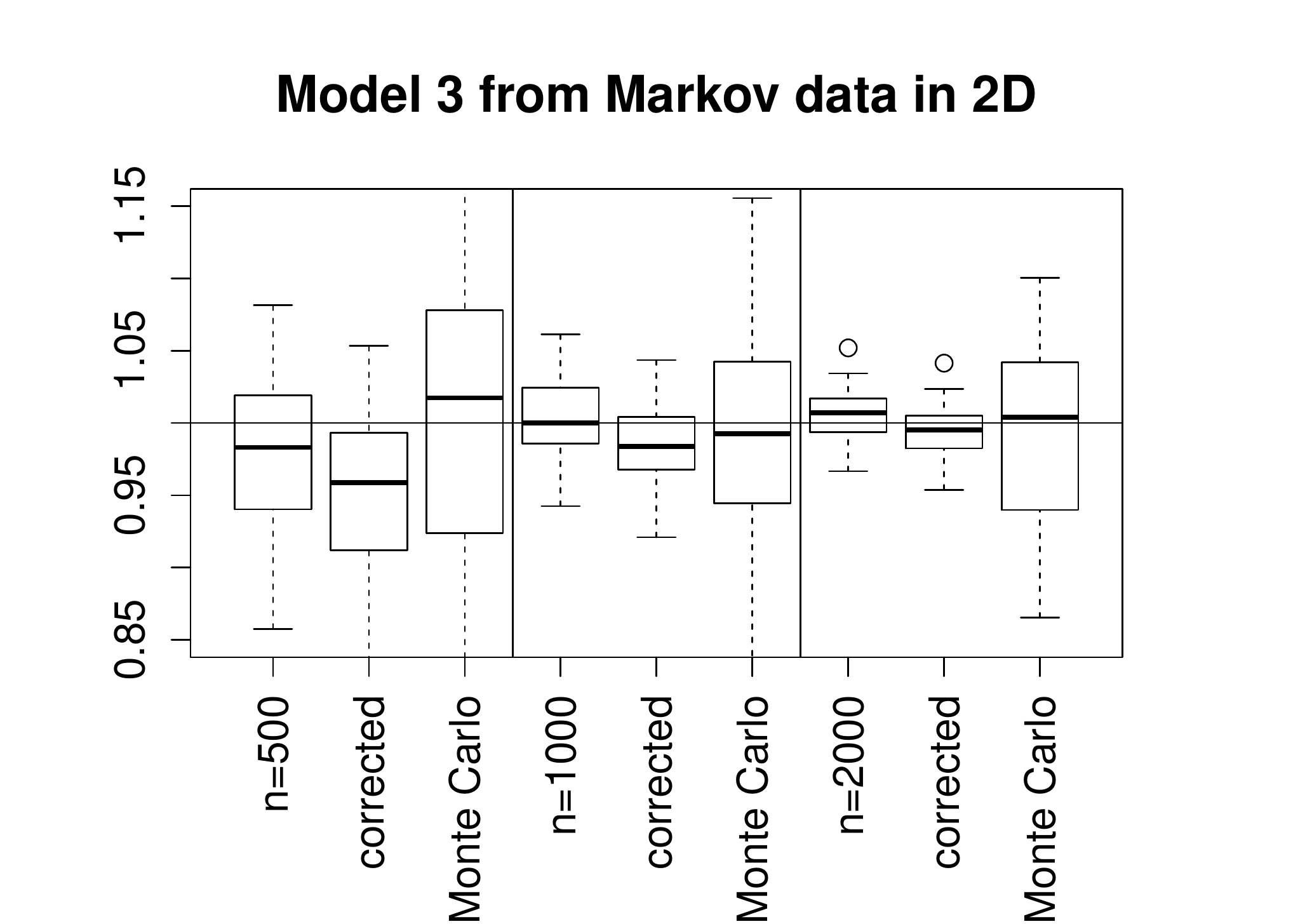} \\ 
   \includegraphics[width=0.5\textwidth]{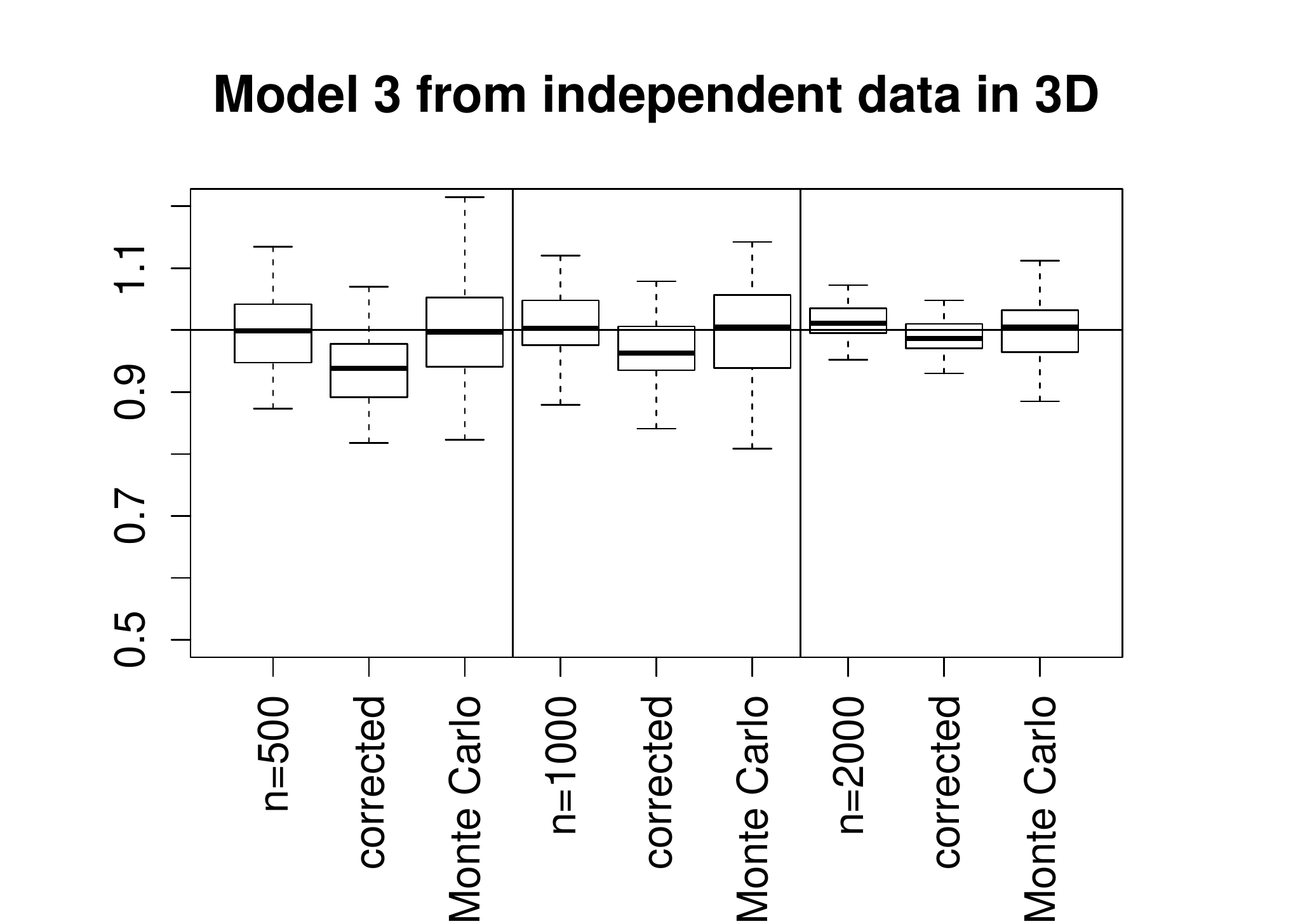} & \includegraphics[width=0.5\textwidth]{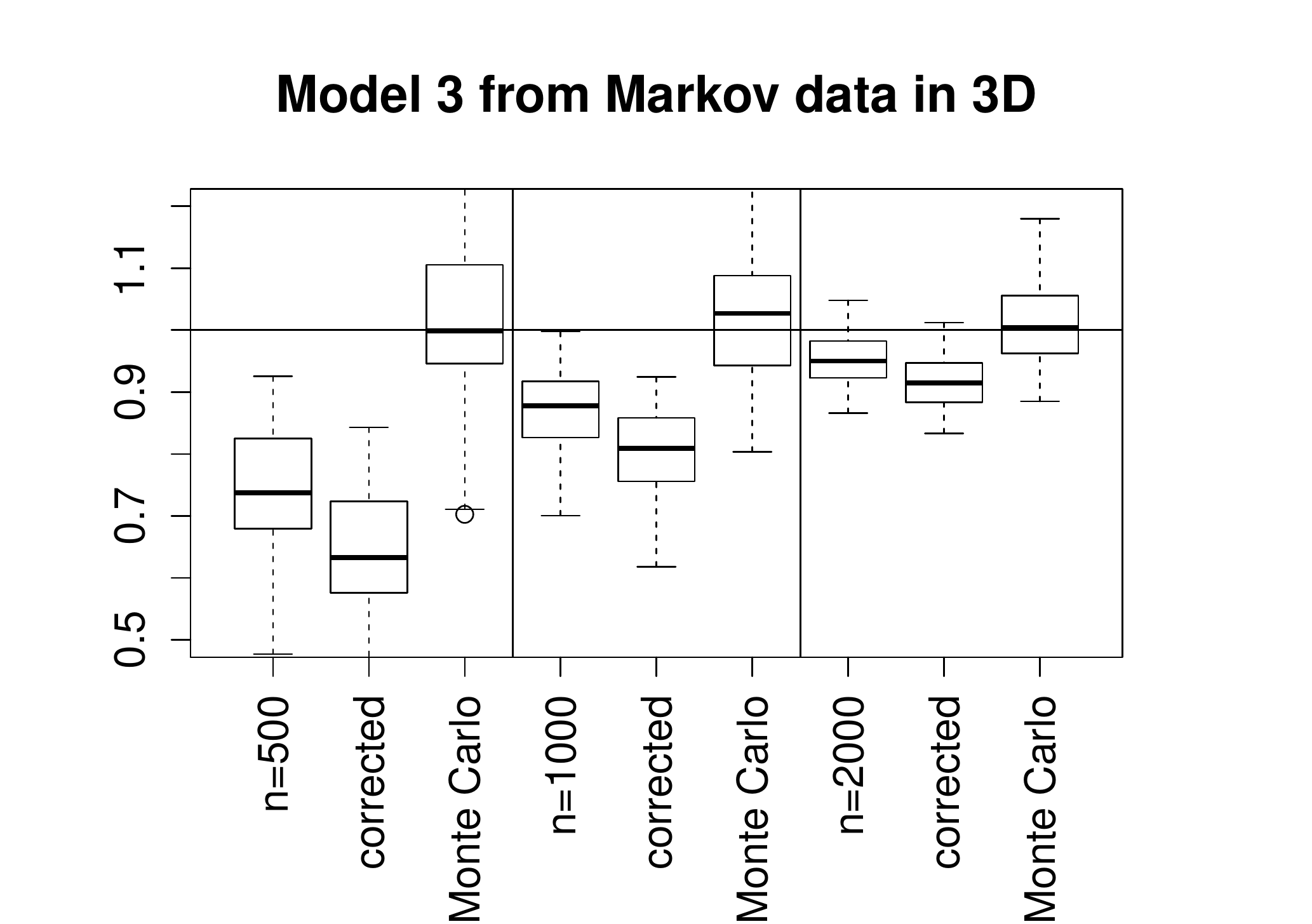} 
\end{tabular}
\caption{Boxplots of $\widehat{I}_{\text{ks}}$, $\widehat{I}_{\text{ks}}^c$ and $\widehat{I}_{\text{mc}}$ computed from $50$ replicates for model $\mathcal{M}_3$ in dimension $d=1$ (top), $d=2$ (middle) and $d=3$ (bottom) from independent data (left) and Markov data (right).}
\label{fig:boxplots:model3}
\end{figure}

First, estimators $\widehat{I}_\text{ks}$ and $\widehat{I}_\text{ks}^c$ have similar dispersions, but the corrected version is more accurate in most cases and should be promoted. Unsurprisingly, the results are better in terms of bias and variance when estimation of the design distribution is computed from independent data rather from Markov data. In addition, the accuracy deteriorates when the dimension increases.
For too small samples, the integral is underestimated (see for example models $\mathcal{M}_1$ and $\mathcal{M}_2$ in dimensions $2$ and $3$), in particular in the Markov framework (see model $\mathcal{M}_3$ in dimension $3$).
Numerical results are quite similar for models $\mathcal{M}_1$ and $\mathcal{M}_2$, which states that the method is not sensitive to skewness. As expected, quality is a little lower for $\mathcal{M}_3$. In the $3$ considered models, the Monte Carlo estimator $\widehat{I}_\text{mc}$ presents no bias but a large dispersion in comparison with $\widehat{I}_\text{ks}$ and $\widehat{I}_\text{ks}^c$, especially in the Markov framework where the dataset does not exactly follow the distribution $\pi$.
The numerical study shows that the methodology is very efficient and applicable in various contexts, in particular compared to Monte Carlo methods that achieve worse results in terms of variance and can not be applied in a statistical framework. Nevertheless, additional numerical experiments point out that both estimators present some bias when function $\varphi$ is not continuous.

As stated in Theorem \ref{th:bigOpth}, the shape of the function $\pi$ (and secondarily $\varphi$) plays an important role in the convergence rate of $\widehat I_{\text{ks}}$: the smoother the better. Hence, the situation when $\pi$ is the uniform density on $Q$ is far from being easy (as the function is not even continuous).
Continuity of $\pi$ is no remedy as it implies the cancellation of $\pi$ at the border and therefore provides too few points near the border. One solution is to consider points that lie slightly outside $Q$, say in $ \tilde Q\supset Q$, in order to stabilize the estimation of $\pi$ at the border of $Q$.  Then compute the kernel estimator $\tilde \pi$ using all these points, and finally calculate
\begin{align*}
\tilde I_{\text{ks}} =  n^{-1} \sum_{i=1}^n \frac{\varphi(X_i)\mathbb{1}_{\{X_i\in Q\}}}{\tilde \pi(X_i) }.
\end{align*}
In the applications where only points in $Q$ are given, one might prefer to consider a different set $Q$, slightly smaller than the original, in order to implement the previous method. If collecting the points has not been done, it might be appropriate to allow the sensor capturing data to get out of $Q$.

\subsection{Real data analysis}

The U.S. National Centers for Environmental Information (NCEI) are parts of National Oceanic and Atmospheric Administration (NOAA). NCEI form the world's largest provider of weather and climate data. The real data analysis presented in the present paper is based on sea surface temperatures obtained all around the world between 2005 and 2015 from profiling floats (PFL dataset) and available on NCEI's website\footnote{World Ocean Database Search and Select (last consulted in July 2016): \url{https://www.nodc.noaa.gov/cgi-bin/OC5/SELECT/builder.pl}}. Sea surface temperatures have a large influence on climate and weather and are therefore used in analyses of climate change. The dataset investigated in this article contains about $1.3$M data and is fully described in Table \ref{tab:dataset} and Figure \ref{fig:worldmap}. Data preprocessing has been implemented in \verb+Python+, while estimation and data analysis have been made with \verb+R+.

\begin{table}[h]
\centering
\begin{tabular}{|c|c|c|c|c|c|}
\hline
\multicolumn{6}{|c|}{Total}\\
\multicolumn{6}{|c|}{1\,343\,094 }\\ \hline\hline
\multicolumn{6}{|c|}{Ocean}\\ \hline
\multicolumn{2}{|c|}{Pacific} & \multicolumn{2}{c|}{Atlantic} & \multicolumn{2}{c|}{Indian}\\
\multicolumn{2}{|c|}{727\,135} & \multicolumn{2}{c|}{336\,180} & \multicolumn{2}{c|}{279\,779 } \\ \hline \hline
\multicolumn{2}{|c|}{Year} & \multicolumn{2}{c|}{Year} & \multicolumn{2}{c|}{Year} \\ \hline
2005 (min) & 2015 (max) & 2005 (min) & 2015 (max) &2005 (min) &2015 (max) \\
35\,773 & 86\,961 &  16\,242 & 45\,488 & 14\,134 & 33\,049 \\
\hline
\end{tabular}
\caption{Size of the sub-datasets extracted from PFL dataset between 2005 and 2015.}
\label{tab:dataset}
\end{table}


\begin{figure}
\centering \includegraphics[width=0.5\textwidth]{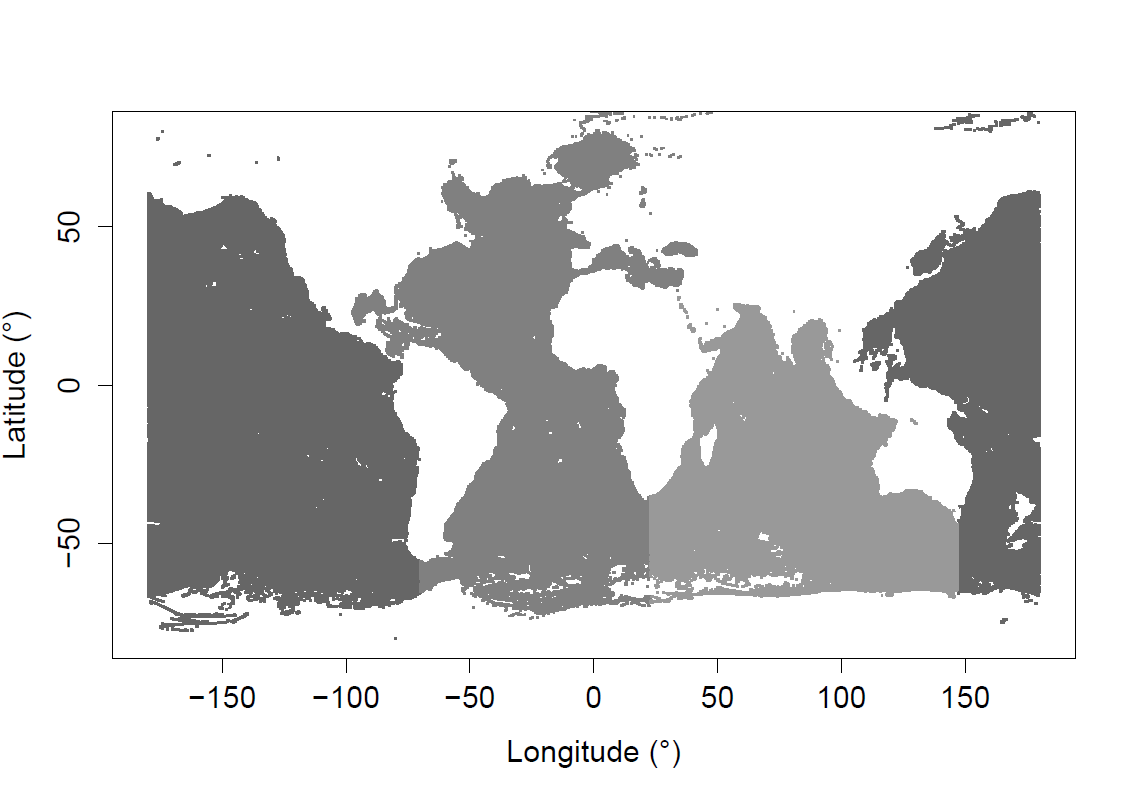}
\caption{Visualization of the 1\,343\,094 points of PFL dataset between 2005 and 2015. Oceans are distinguished using gray shades (darker to lighter: Pacific, Atlantic, Indian).}
\label{fig:worldmap}
\end{figure}

The database of interest consists of spatiotemporal data obtained from measure instruments with unpredictable trajectories, which makes them hardly tractable. We focus here on the estimation of the average sea surface temperature for a given period of time, between 2005 and 2015, and for some given areas in the 3 major oceans. Areas considered in this paper are delimited by the latitude: more than $50^\circ$ (North region), $[30^\circ,50^\circ]$, $[10^\circ,30^\circ]$ (North Tropical region), $[-10^\circ,10^\circ]$ (Equatorial region), $[-30^\circ,-10^\circ]$ (South Tropical region), $[-30^\circ,-50^\circ]$ and less than $-50^\circ$ (South region). For each mentioned spatial region, we have estimated the average sea surface temperature over each month by the corrected algorithm presented in section \ref{ss:algo}. This technique is fully adapted to the problem at hand because measurement locations follow erratic trajectories with unknown distribution.

Local average sea surface temperatures for the 3 oceans are presented in Figure \ref{fig:profils}. One obtains temperature patterns according to the location on the North-South axis. One may observe that the variability of sea surface temperatures in a given region over $11$ years is weak compared to the variations in latitude, especially for the Pacific Ocean. In other words, the temperature mainly depends on the latitude, rather on the period of the year. Unsurprinsingly, sea surface temperatures are the highest under the Equator and near the Tropics, where Earth receives the most direct sunlight.

In Figure \ref{fig:tempo}, we present time series over $11$ years of average sea surface temperatures in $3$ regions: South Tropical Pacific Ocean (latitude between $-30^\circ$ and $-10^\circ$), North Atlantic Ocean (latitude between $50^\circ$ and $60^\circ$) and Equatorial Indian Ocean (latitude between $-10^\circ$ and $10^\circ$). First, it should be noted that we observe an expected seasonal effect on sea surface temperatures of South Pacific and North Atlantic Oceans: the highest temperatures occur in January and February in the Southern Hemisphere, while they occur in August and September in North Atlantic Ocean. In addition, we note a general decrease in sea surface temperature in Southern Pacific between 2006 and 2009 followed by a stable period. This phenomenon has been taken into account in simulations proposed in \cite{pac}. In particular, they show that recent cooling in Pacific Ocean is tied to recent global-warming hiatus. One may also remark that temperature in North Atlantic Ocean has decreased recently. Indeed, there is a region of cooling in the Northern Atlantic. \cite{ncc} suggest that this cooling may be due to changes in the Atlantic meridional overturning circulation in the late twentieth century. Finally we point out that Equatorial Indian Ocean has tended to warm for at least $10$ years. According to \cite{kollroxy}, this warming begun more than a century ago and is linked to the El Ni\~{n}o -- Southern Oscillation periodical phenomenon.

\begin{figure}[h!]
\centering
   \includegraphics[width=0.33\textwidth,height=7cm]{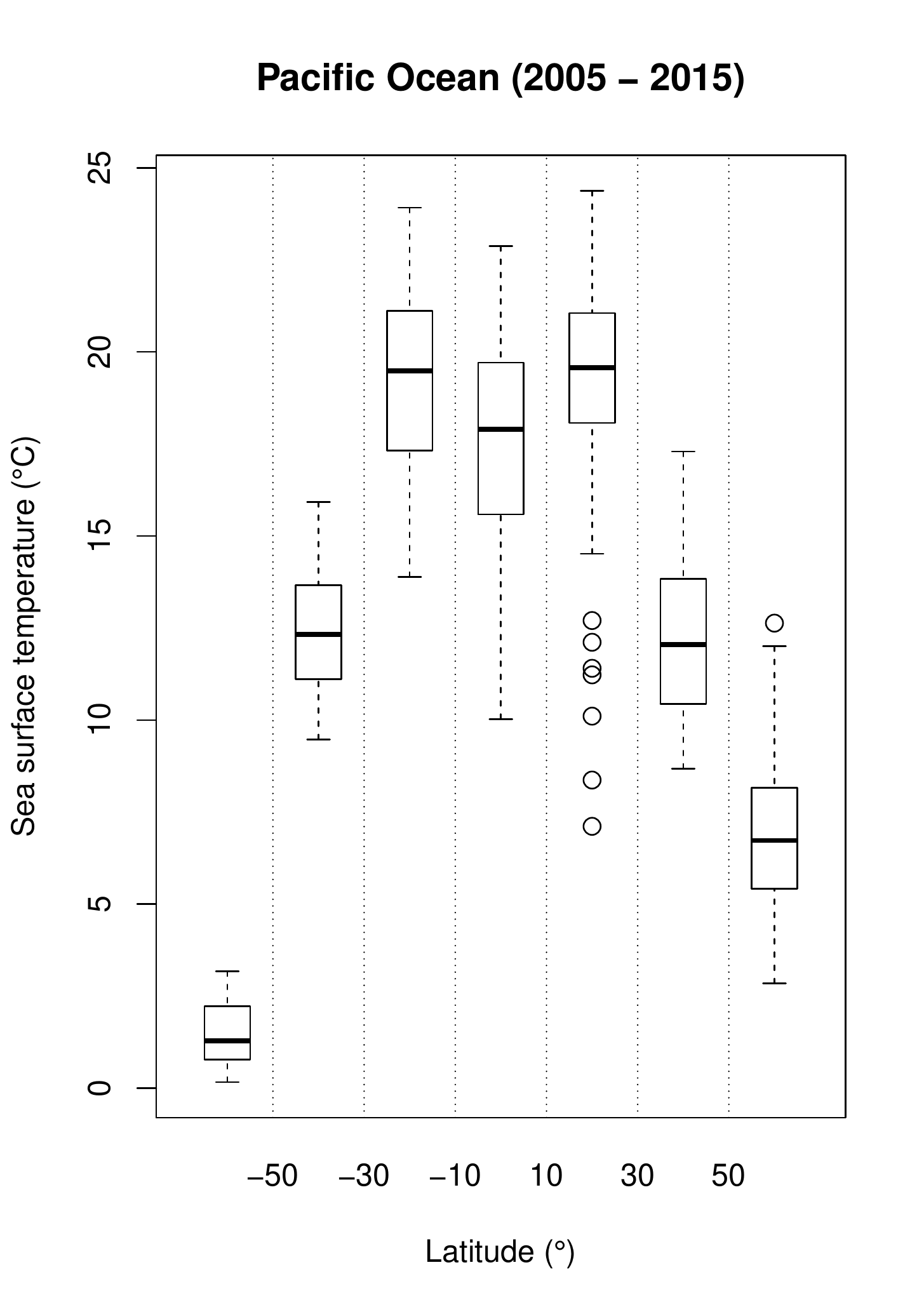}\includegraphics[width=0.33\textwidth,height=7cm]{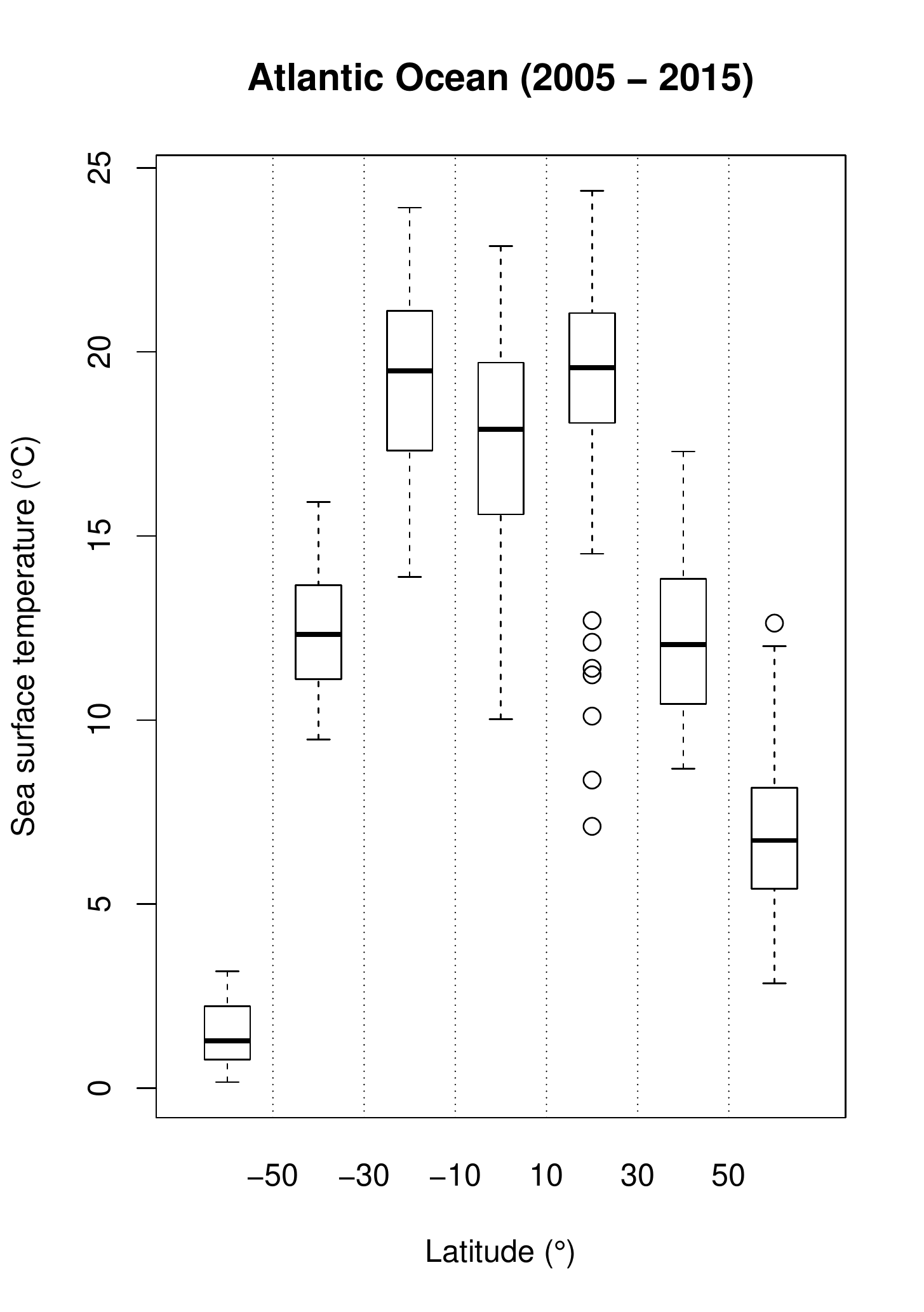}\includegraphics[width=0.33\textwidth,height=7cm]{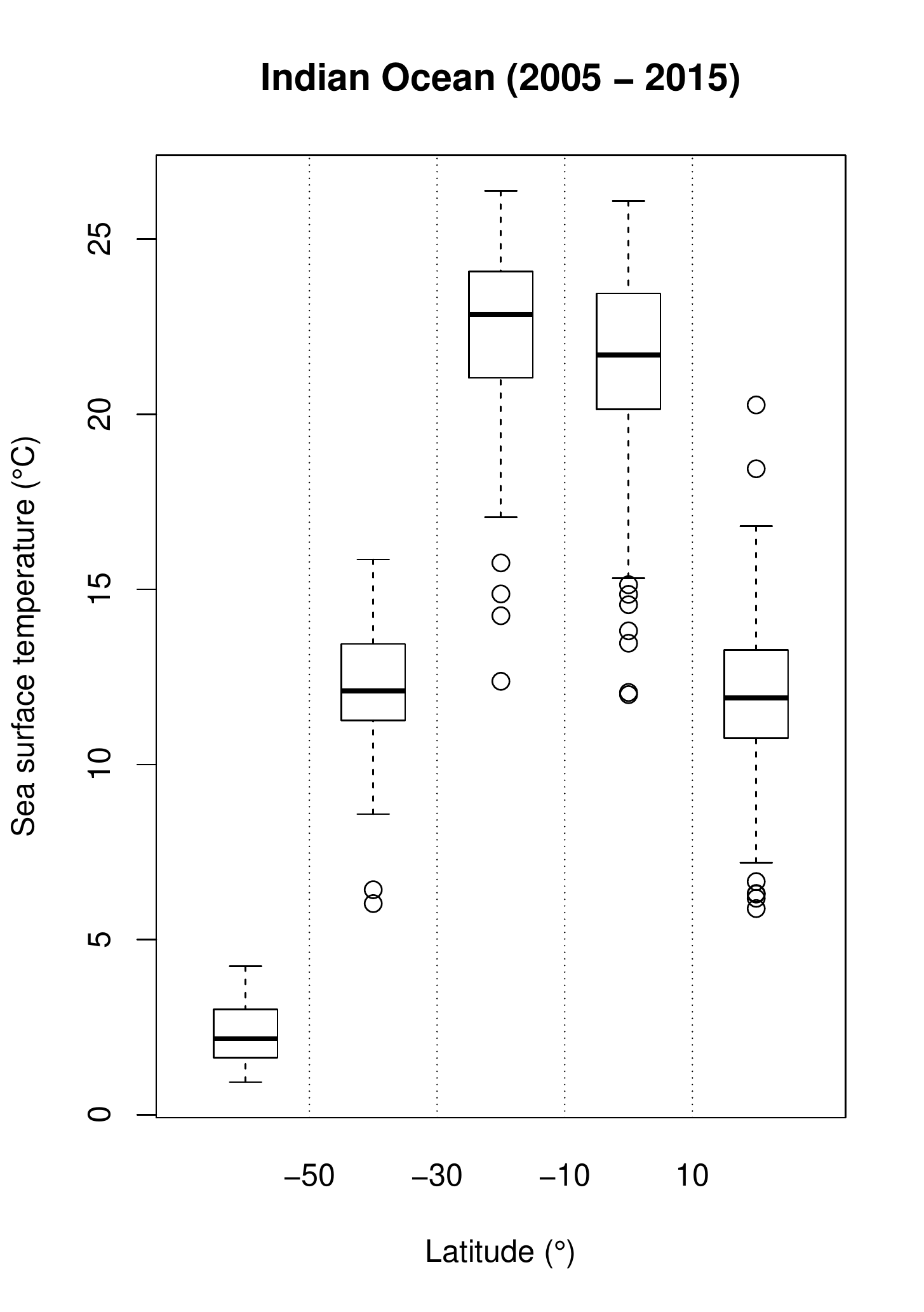}
\caption{Average sea surface temperatures according to the latitude of the considered area for the 3 major oceans. Each boxplot has been computed from $11\times12=132$ estimates of the average temperature for each month of each year between 2005 and 2015.}
\label{fig:profils}
\end{figure}

\begin{figure}[h!]
\centering
   \includegraphics[width=0.33\textwidth,height=7cm]{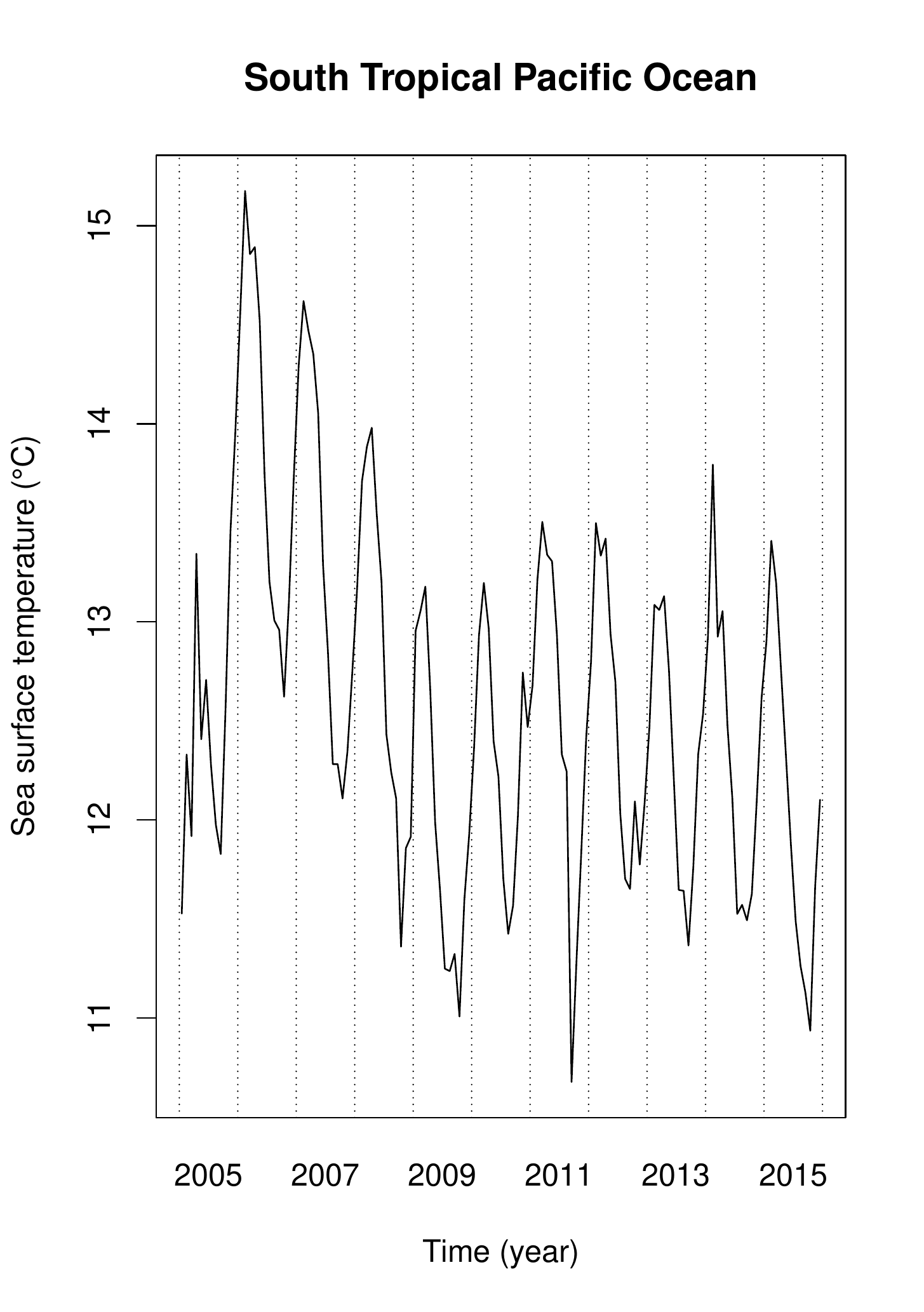}\includegraphics[width=0.33\textwidth,height=7cm]{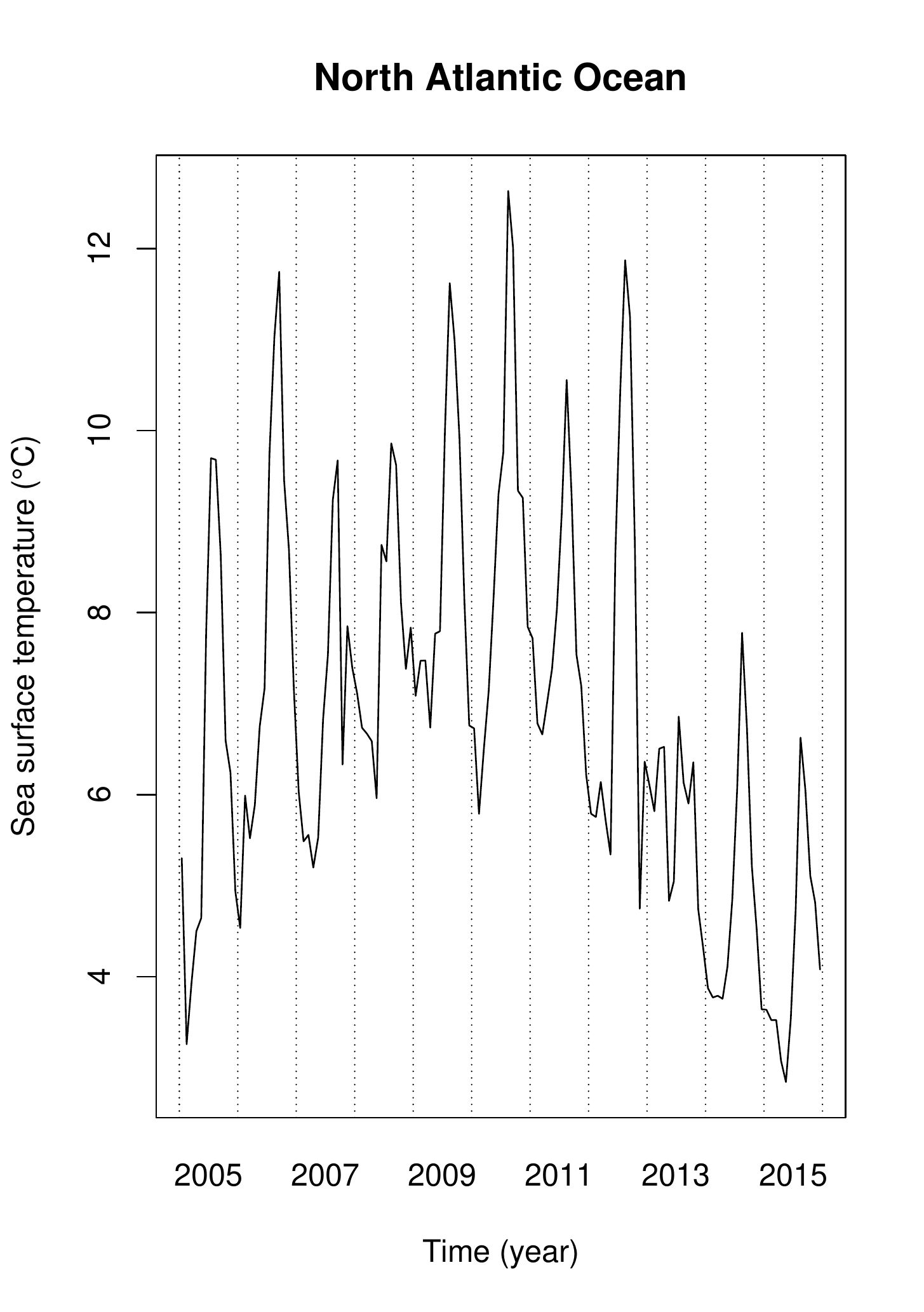}\includegraphics[width=0.33\textwidth,height=7cm]{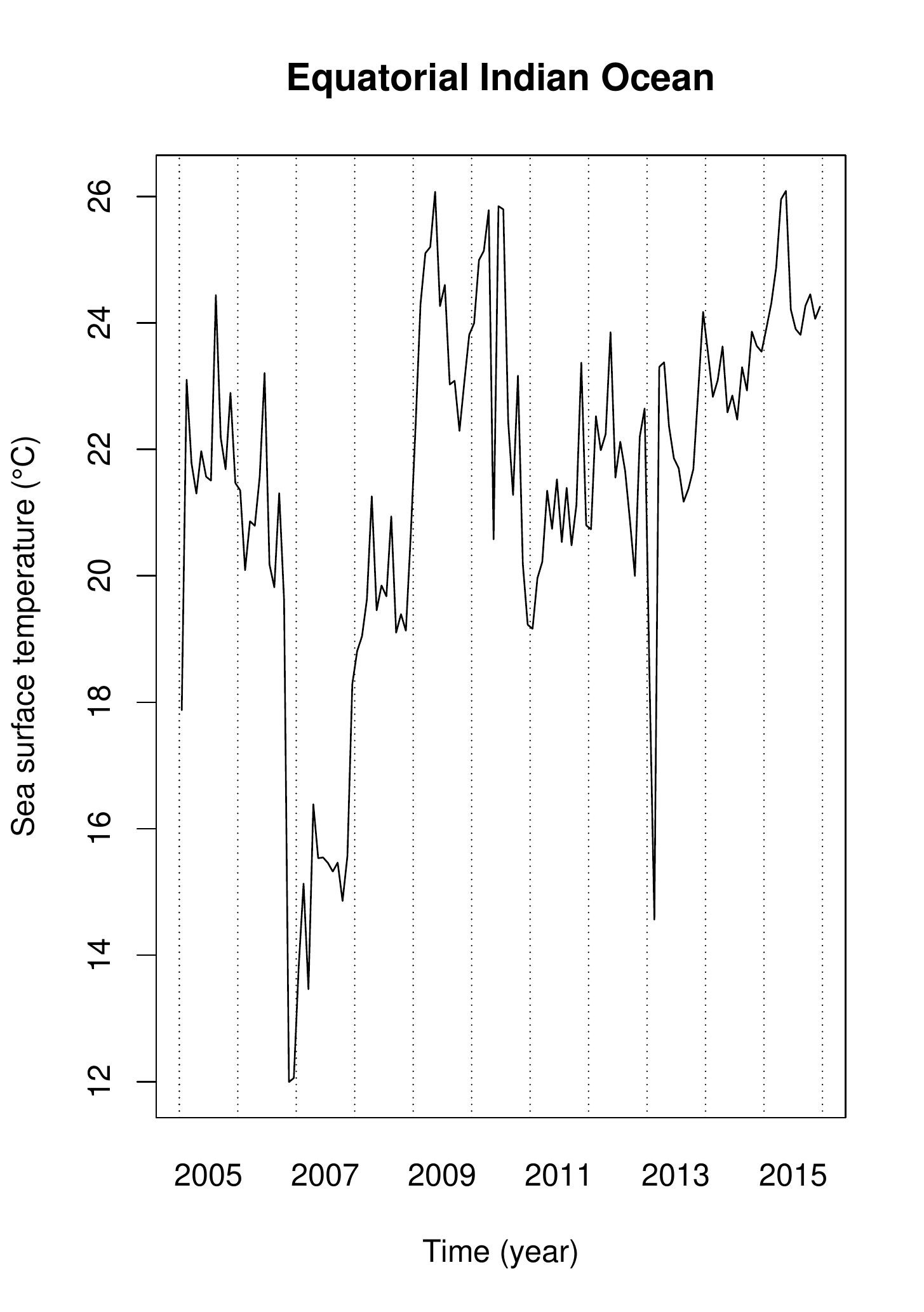}
\caption{Times series of sea surface temperature in some specific areas of the 3 major oceans between 2005 and 2015. Latitude between $-30^\circ$ and $-10^\circ$ for South Tropical Pacific Ocean (top), $50^\circ$ and $60^\circ$ for North Atlantic Ocean (middle), and $-10^\circ$ and $10^\circ$ for Equatorial Indian Ocean (bottom).}
\label{fig:tempo}
\end{figure}


\begin{appendices}



\section{Regeneration-based bounds for expectations}\label{s31}

We employed the Nummelin splitting technique in order to exploit the independence between the blocks $B_k$, 
$k\in \mathbb N^*$, as described in section \ref{s2} of the associated paper. We have however taken care of giving conditions on the moments $\tau_A$ of the original 
chain $(X_i)_{i\in \mathbb N}$ rather than on the moments $\theta_a$ of 
the split chain $(Z_i)_{i\in \mathbb N} $.

Define, for any $p>0$,
\begin{align*}
\xi(p)= \sup_{x\in A}\mathbb {E}_x[\tau_A^p ]  .
\end{align*}
We start with a lemma relating moments of $\theta_a$ to moments of $\tau_A$.

\begin{lemma}\label{lemma:tau} Let $(X_i)_{i\in \mathbb N}$ be a Markov chain satisfying (\ref{ren1}). Then, for any $x_0 \in\EuScript E$, $p\geqslant 1$,
\begin{align}
&\mathbb {E}_{x_0}[\theta_a^{p}]^{1/p}
\leqslant \frac{1}{e^{\lambda_0/{p}}-1}\xi(p)^{1/p}+\mathbb {E}_{x_0}[\tau_A^{p}]^{1/p} \label{c0tau0} \\
&\mathbb {E}_a[ \theta_a^{p}]^{1/p}
\leqslant\lambda_0^{-1} \frac{ e^{\lambda_0/{p}}}{ e^{\lambda_0/{p}}-1} \xi(p)^{1/p}   .\label{c0tau0prime}
\end{align}
\end{lemma}
\begin{proof} 
We start by showing (\ref{c0tau0}). Suppose that $\mathbb {E}_{x_0}[\tau_A^p]<+\infty$ 
and $\sup_{x\in A}\mathbb {E}_x[\tau_A^p]<+\infty$, if not, the stated inequality is obviously satisfied. 
By the Minkowski inequality, we have
\begin{align*}
\mathbb {E}_{x_0}[\theta_a^p]^{1/p}
&\leqslant \mathbb {E}_{x_0}[(\theta_a-\tau_A)^p]^{1/p}+\mathbb {E}_{x_0}[\tau_A^p]^{1/p}\\
&= \mathbb {E}_{x_0}[(\theta_a-\tau_A)^p  \mathbb{1}_{\{\theta_a>\tau_A \}}]^{1/p}+\mathbb {E}_{x_0}[\tau_A^p]^{1/p}.
\end{align*}
Let $\mathscr F_{\tau_A}$ denote the $\sigma$-field of the past before $\tau_A$ and note that ${\{\theta_a>\tau_A \}}$ is $\mathscr F_{\tau_A}$-measurable. By the strong Markov property, it holds
\begin{align*}
\mathbb {E}_{x_0}[(\theta_a-\tau_A)^p \mathbb{1}_{\{\theta_a>\tau_A \}} |\mathscr F_{\tau_A}]= \mathbb {E}_{x_0}[(\theta_a-\tau_A)^p   |\mathscr F_{\tau_A}]\mathbb{1}_{\{\theta_a>\tau_A \}} \leqslant \mathbb{1}_{\{\theta_a>\tau_A \}} \sup_{x\in A}\mathbb {E}_x[\theta_a^p].
\end{align*}
Hence, setting $\gamma=\sup_{x\in A }\mathbb {E}_x[\theta_a^p]^{1/p}$, and because $\lambda_0 = \mathbb {P}_{x_0}(\theta_a= \tau_A)=\mathbb {P}_{x_0}(Y_{\tau_A}=1)$,
\begin{align}\label{mink}
\mathbb {E}_{x_0}[\theta_a^p]^{1/p}&\leqslant \gamma (1-\lambda_0)^{1/p}+\mathbb {E}_{x_0}[\tau_A^p]^{1/p}.
\end{align}
In particular, it follows that
\begin{align*}
(1-(1-\lambda_0)^{1/p})\gamma &\leqslant \sup_{x\in A}\mathbb {E}_x[\tau_A^p]^{1/p}.
\end{align*}
Thus, (\ref{mink}) becomes
\begin{align}
\mathbb {E}_{x_0}[\theta_a^p]^{1/p}&\leqslant (1-\lambda_0)^{1/p}(1-(1-\lambda_0)^{1/p})^{-1}\sup_{x\in A}
\mathbb {E}_x[\tau_A^p]^{1/p}+\mathbb {E}_{x_0}[\tau_A^p]^{1/p},
\end{align}
and we obtain (\ref{c0tau0}) by using $1-\lambda_0\leqslant e^{-\lambda_0}$. To get (\ref{c0tau0prime}), 
note that for every $x_0\in A$,
\begin{align*}
\mathbb {E}_{x_0} [ \theta_a^{p}   \mathbb{1}_{\{Y_0=1\}}] = \lambda_0 \mathbb {E}_{a} [ \theta_a^{p} ].
\end{align*} 
It follows that $ \mathbb {E}_{a} [ \theta_a^{p} ]\leqslant\lambda_0^{-1}\mathbb {E}_{x_0}[\theta_a^{p}]$ 
and we get the result from (\ref{c0tau0}), taking the supremum over $A$.
\end{proof}

We shall need also the following extension of (\ref{bert0}).
\begin{lemma}\label{lemma:invariant_meas_formula_extended}
Let $(X_i)_{i\in \mathbb N}$ be a Markov chain satisfying (\ref{taupr}), (\ref{tauint}) and (\ref{ren1}).
For any measurable function $h:\cup_{n\geqslant 1}\mathbb R^n\rightarrow \mathbb R$ such that $\mathbb {E}_\pi[h(X_1,...X_{\theta_a})]<+\infty$, 
(for any $n$ the restriction of $h$ to $\mathbb R^n$ is measurable), we have
\begin{align}\label{taupui}
\alpha_0 \mathbb {E}_\pi[h(X_1,...X_{\theta_a})]&=\mathbb {E}_a\Big[\sum_{i=1}^{\theta_a}h(X_i,...X_{\theta_a})\Big].
\end{align}
In particular, for any $p>0$,
\begin{align}\label{eq:inequality_invariant_measu}
\alpha_0 \mathbb {E}_\pi[\theta_a^p]
&\leqslant \mathbb {E}_a[\theta_a^{p+1}]\leqslant (p+1)\alpha_0 \mathbb {E}_\pi[\theta_a^p].
\end{align}
\end{lemma}
\begin{proof} Having (\ref{taupr}), (\ref{tauint}) and (\ref{ren1}) we can use the formula (\ref{bert0}). 
Define $g(x)=\mathbb {E}_x[h(X_1,...X_{\theta_a})] $ and remark that, by the Markov property and the fact that $\{i< \theta_a\}$ is $\mathscr F_i$-measurable,
\begin{align*}
&\mathbb {E}_a(g(X_i) \mathbb{1}_{\{i< \theta_a\}}) = \mathbb {E}_a(h(X_{i+1},...X_{\theta_a}) \mathbb{1}_{\{i< \theta_a\}} ),\\
&g(X_{\theta_a})=\mathbb {E}_a(h(X_{1},...X_{\theta_a})) .
\end{align*}
 Then using (\ref{bert0}) with $g$, we get
\begin{align*}
\alpha_0\mathbb {E}_\pi[h(X_1,...X_{\theta_a})]
&=\alpha_0 \pi(g) \\
&=\mathbb {E}_a\Big[\sum_{i=1}^{\theta_a}g(X_i) \Big]\\
&=\mathbb {E}_a\Big[\sum_{i=1}^{{\theta_a}-1}h(X_{i+1},...X_{\theta_a})\Big]
+\mathbb {E}_a\Big[h(X_{1},...X_{{\theta_a}})\Big]\\
&=\mathbb {E}_a\Big[\sum_{i=1}^{{\theta_a}}h(X_i,...X_{\theta_a})\Big].
\end{align*}
Concerning the second statement, we use the fact that
$1+2^p+\dots \theta_a^p\geqslant \int_0^{\theta_a} x^pdx=\frac{\theta_a^{p+1}}{p+1}$ to write
\begin{align*}
 \tfrac{1 }{p+1}\mathbb {E}_a[\theta_a^{p+1}]
 \leqslant \mathbb {E}_a\big[\sum_{i=1}^{\theta_a} i^p\big] \leqslant   \mathbb {E}_a[\theta_a^{p+1}].
\end{align*}
We conclude by using (\ref{taupui}) with $h(x_1,\dots x_k)=k^p$, to show that the middle term is 
$\alpha_0\mathbb {E}_\pi[\theta_a^p]$.
\end{proof}

\begin{lemma}\label{lemma:insideblock}
Let $(X_i)_{i\in \mathbb N}$ be a Markov chain satisfying (\ref{taupr}), (\ref{tauint}) and (\ref{ren1}).  For any $p>2$, there exists $C>0$ (depending on $p$, $\lambda_0$, $\alpha_0$) such that for 
any measurable function $f$,
\begin{align*}
\mathbb {E}_a\Big[\Big(\sum_{i=1}^{\theta_a} f(X_i)\Big)^2\Big]
&\leqslant   C \left(\xi(p)^2\pi(f^2)+ \xi(p) \mathbb E_\pi[f(X_0)^2 \tau_A^p]\right) .
\end{align*}
\end{lemma}

\begin{proof}
Suppose that $f\geqslant 0$. If not, take $|f|$ instead of $f$. 
In what follows, we use the convention that empty sums equal $0$. 
Applying Lemma~\ref{lemma:invariant_meas_formula_extended} with 
\begin{align*}
h(x_i,\ldots x_k) =\Big(\sum_{j=1}^k f(x_j)\Big)^2-\Big(\sum_{j=2}^k f(x_j)\Big)^2
= f(x_i)^2 +2 f(x_i)\sum_{j=2}^k f(x_j),
\end{align*}
we get that 
\begin{align*}
\mathbb {E}_a\Big[\Big(\sum_{i=1}^{\theta_a} f(X_i)\Big)^2\Big]& =\mathbb {E}_a\Big[\sum_{i=1}^{\theta_a}h(X_i,\ldots X_{\theta_a})\Big]\\
&= \alpha_0\mathbb {E}_\pi\Big[ f(X_1)\Big(f(X_1)+2\sum_{i=2}^{\theta_a} f(X_i)\Big)\Big]\\
&= \alpha_0\left( \pi(f^2)+2 \mathbb {E}_\pi\Big[ f(X_1)\sum_{i=2}^{\theta_a} f(X_i)\Big]\right).
\end{align*}
For any $p>2$, the second term is bounded as follows
\begin{align*}
\mathbb {E}_\pi\Big[f(X_1)\sum_{i=2}^{{\theta_a}} f(X_i)\Big]
&=\sum_{i\geqslant 2} \mathbb {E}_\pi\Big[\mathbb{1}_{i\leqslant {\theta_a}}f(X_1)f(X_i)\Big]\\
&\leqslant \sum_{i\geqslant 2} \mathbb {E}_\pi\Big[i^{-p/2}\theta_a^{p/2} \,f(X_1)f(X_i)\Big]\\
&\leqslant \sum_{i\geqslant 2} i^{-p/2} \mathbb {E}_\pi\Big[f(X_1)^2\theta_a^{p}\Big]^{1/2}\mathbb {E}_\pi\Big[f(X_i)^2\Big]^{1/2}\\
&= \left(  \sum_{i\geqslant 2} i^{-p/2}\right) \mathbb {E}_\pi\Big[f(X_1)^2\theta_a^p\Big]^{1/2}\mathbb {E}_\pi\Big[f(X_1)^2\Big]^{1/2} \\
&\leqslant  \left( \frac2{p-2}\right)  \mathbb {E}_\pi\Big[f(X_1)^2\theta_a^p\Big],
\end{align*}
where we have used $\sum_{i\geqslant 2} i^{-p/2}\leqslant \int_1^{+\infty}x^{-p/2}dx$. 
If $\widetilde{\theta}_a$ is the first time $k\geqslant 2$ such $Z_k\in a$, it holds 
\begin{align*}
\mathbb {E}_\pi\Big[f(X_1)^2\theta_a^{p}\Big]\leqslant
\mathbb {E}_\pi\Big[f(X_1)^2\widetilde \theta_a^{p}\Big]
=\mathbb {E}_\pi\Big[f(X_0)^2(\theta_a^{p}+1)\Big]
\leqslant 2\mathbb {E}_\pi\Big[f(X_0)^2\theta_a^{p}\Big] .
\end{align*}
Applying Lemma \ref{lemma:tau}, equation (\ref{c0tau0}), and using that for every $a,b\geqslant 0$, 
and $p>1$, $(a+b)^p\leqslant 2^{p-1} (a^p+b^p)$, we get
\begin{align*}
\mathbb {E}_\pi\Big[f(X_0)^2\theta_a^{p}\Big]
\leqslant 2^{p-1}  \left( \frac{ 1 }{ (e^{\lambda_0/{p}}-1)^p}   \xi(p)  \pi(f^2)  
+  \mathbb E_\pi\left[ f(X_0)^2 \tau_A^{p} \right] \right).
\end{align*} 
Bringing everything together, we get 
\begin{align*}
\mathbb {E}_a\Big[\Big(\sum_{i=1}^{\theta_a} f(X_i)\Big)^2\Big]\leqslant \alpha_0
\left( \pi(f^2)+ \frac{2^{p+2}}{p-2}  \left( \frac{ 1 }{ (e^{\lambda_0/{p}}-1)^p} \xi(p)  \pi(f^2)  
+  \mathbb E_\pi\left[ f(X_0)^2 \tau_A^{p} \right] \right)\right).
\end{align*}
This leads to the stated result.
\end{proof}

\begin{theorem}\label{rosenth}
Let $(X_i)_{i\in \mathbb N}$ be a Markov chain satisfying (\ref{taupr}), (\ref{tauint}) 
and (\ref{ren1}). 
There exists $C>0$ (depending on $p$, $\lambda_0$, $\alpha_0$) such that, for any measurable function $g$ with $\pi(g)=0$, any $n\geqslant 1$ and $p>2$,
\begin{align*}
&\mathbb {E}_\pi\Big[\Big(\sum_{i=1}^n g(X_i)\Big)^2\Big]
\leqslant n C \left(\xi(p)^2\pi(g^2)+ \xi(p) \mathbb E_\pi[g(X_0)^2 \tau_A^p]\right).
\end{align*}
\end{theorem}

\begin{proof}
Defining the blocks sums as (see equation (\ref{block}))
\begin{align*}
G_k&=\sum_{i={\theta_a}(k)+1}^{{\theta_a}(k+1)}g(X_i),
\end{align*}
(in this whole section we set $\sum\nolimits_a^b=0$ if $b<a$)
 $G_k$ is an i.i.d. sequence and one has 
\begin{align*}
\sum_{i=1}^ng(X_i)
&=\sum_{i=1}^{{\theta_a}\wedge n} g(X_i)+\sum_{k=1}^{l_n-1}G_k+\mathbb{1}_{{\theta_a}\leqslant n} \sum_{i={\theta_a}(l_n)+1}^n g(X_i)
\end{align*}
where $l_n$ is the number of times $Z_i$ visits $a$ before $n$, i.e.,
\begin{align}\label{def:l_n}
l_n=\sum_{i=1}^n  \mathbb{1}_{\{Z_i\in a\}}.
\end{align} 
As the chain has been split into independent blocks, the process  $L\mapsto \sum_{k=1}^L G_k$ 
is a martingale. The sequence $(l_n)$ is random and is expected to be of order $n$. 
Since $l_n\leqslant n$, following \cite{bertail2011}, page\,21, we have
 \begin{align*}
\Big| \sum_{i=1}^n g(X_i)\Big|
\leqslant\sum_{i=1}^{{\theta_a}\wedge n} f(X_i)+\max_{1\leqslant L\leqslant n} \Big|\sum_{k=1}^L G_k \Big|+\mathbb{1}_{{\theta_a}\leqslant n}\sum_{i={\theta_a}(l_n)+1}^{n} f (X_i),
 \end{align*}
 where $f=|g|$ (considering $f$ instead of $g$ will help later
for the treatment of the concerned terms). 
By the Minkowski inequality, denoting by $\|\cdot\|_2$ the $L_2(\mathbb {P}_\pi)$ norm, we have
\begin{align}\label{rosen1}
\Big\|\sum_{i=1}^ng(X_i)\Big\|_2
&\leqslant \Big\|\sum_{i=1}^{{\theta_a}\wedge n} f(X_i)\Big\|_2 
+\Big\|\max_{1\leqslant L\leqslant n}\Big|\sum_{k=1}^L G_k\Big|\Big\|_2
+\Big\| \sum_{i={\theta_a}(l_n)+1}^n f(X_i)\Big\|_2.
\end{align}
Using Doob's inequality, we have
\begin{align*}
\mathbb {E}_\pi\max_{1\leqslant L\leqslant n} \Big|\sum_{k=1}^L G_k\Big|^2
\leqslant 4 n\mathbb {E}_\pi[|G_1|^2]=  4 n\mathbb {E}_a\Big[\Big(\sum_{i=1}^{\theta_a} g(X_i)\Big)^2\Big],
\end{align*}
then, from Lemma~\ref{lemma:insideblock}, we get for every $p>2$ that there exist $\widetilde C$ such that
\begin{align*}
\mathbb {E}_\pi\max_{1\leqslant L\leqslant n} \Big|\sum_{k=1}^L G_k\Big|^2
\leqslant 4 n \widetilde C \left(\xi(p)^2\pi(g^2)+ \xi(p) \mathbb E_\pi[g(X_0)^2 \tau_A^p]\right)
\end{align*}
This is also a crude bound for the third term in (\ref{rosen1}) since
\begin{align*}
\mathbb {E}_\pi\Big[\Big(\sum_{i={\theta_a}(l_n)+1}^n f(X_i)\Big)^2\Big]
&\leqslant \mathbb {E}_\pi\Big[\Big(\sum_{i={\theta_a}(l_n)+1}^{{\theta_a}(l_n+1)} f(X_i)\Big)^2\Big]
= \mathbb {E}_a\Big[\Big(\sum_{i=1}^{\theta_a} f(X_i)\Big)^2\Big].
\end{align*}
Now we consider the first term in (\ref{rosen1}).  
Using Lemma~\ref{lemma:invariant_meas_formula_extended} with
\begin{align*}
h(x_1,\ldots x_k) =\Big(\sum_{j=1}^{k\wedge n} f(x_j)\Big)^2,
\end{align*}
we get
\begin{align*}
\mathbb {E}_\pi\Big[\Big(\sum_{j=1}^{\theta_a\wedge n} f(X_j)\Big)^2\Big]
&=  \alpha_0^{-1}\mathbb {E}_a\Big[\sum_{i=1}^{\theta_a}\Big(\sum_{j=i}^{\theta_a\wedge n} f(X_j)\Big)^2\Big]\\
&=    \alpha_0^{-1}\mathbb {E}_a\Big[\sum_{i=1}^{\theta_a\wedge n}\Big(\sum_{j=i}^{\theta_a\wedge n} f(X_j)\Big)^2\Big]\\
&\leqslant n   \alpha_0^{-1}\mathbb {E}_a\Big[\Big(\sum_{j=1}^{\theta_a} f(X_j)\Big)^2\Big]\\
& \leqslant n    \mathbb {E}_a\Big[\Big(\sum_{j=1}^{\theta_a} f(X_j)\Big)^2\Big].
\end{align*}
We conclude again with Lemma~\ref{lemma:insideblock}.
\end{proof}


\section{Proofs of section \ref{s3}}
\label{s32}

\subsection{Proof of Lemma \ref{lemma:regularizationrates}}
We start by proving (\ref{eq:L2regularization}). Define $ k =\lfloor s\rfloor$. From the Taylor formula with integral remainder applied to $g(t)=\psi(x-tu)$, we get
\begin{align*}
\psi(x-hu)-\psi(x)&=\sum_{j=1}^{k-1}\frac{h^j}{j!}g^{(j)}(0)+
\int_0^{h}g^{(k)}(t)\frac{(h-t)^{k-1}}{(k-1)!}dt\\
&=\sum_{j=1}^k\frac{h^j}{j!}g^{(j)}(0)+
\int_0^{h}(g^{(k)}(t)-g^{(k)}(0))\frac{(h-t)^{k-1}}{(k-1)!}dt.
\end{align*}
The first term is a polynomial in $u$ which vanishes after integration with respect to $K$ as by assumption, $K$ is orthogonal to the first non-constant polynomial of degree
$j\leq \lfloor  s\rfloor $. Using the chain rule to compute $g^{(k)}$ and using basic inequalities with some combinatorics, we obtain that there exists a constant $C$ (depending only on $k$ and $d$) such that for every $t\in \mathbb R$,
\begin{align*}
|g^{(k)}(t)-g^{(k)}(0)| \leqslant C |u|_1^k \sum_{l\in \mathcal P_k} | \psi^{(l)}(x-tu)-\psi^{(l)}(x)|  ,  
\end{align*}
where $\mathcal P_k=\{(l_1,\ldots l_d)\in \mathbb N^d\,:\, \sum_{i=1}^d l_i =  k\}$. It follows that
 \begin{align*}
\Big|\int_0^{h}(g^{(k)}(t)-g^{(k)}(0))\frac{(h-t)^{k-1}}{(k-1)!}dt\Big|
&\leqslant  \frac{h^{k-1}C}{(k-1)!} \sum_{l\in \mathcal P_k }  \int_0^{h}|\psi^{(l)}(x-tu)-\psi^{(l)}(x)|\,|u|_1^kdt.
\end{align*}
Hence
\begin{align*}
\Big|\int\left(\psi(x-hu)-\psi(x)\right)K(u)du\Big|
\leqslant  \frac{h^{k-1}C}{(k-1)!} \sum_{l\in \mathcal P_k }\int_0^{h}\int|\psi^{(l)}(x-tu)-\psi^{(l)}(x)|\, |u|_1^k \,|K(u)|\,du\,dt
\end{align*}
and by the generalized Minkowski inequality \citep[page 194]{folland:1999}\footnote
{For any nonegative measurable function $g(.,.)$ on $\mathbb R^{k+d}$, any $\sigma$-finite measures $\mu$ and $\nu $, and any $q\geqslant 1$,
\begin{align*}
\left(\int\left(\int g(y,x)d\mu (y)\right)^q d\nu (x)\right)^{1/q}&\leqslant \int\left(\int g(y,x)^qd\nu (x)\right)^{1/q} d\mu (y).
\end{align*}
}, 
\begin{align*}
\big\| \psi -\psi_{h} \big\|_{L_q(\pi)} 
&\leqslant \frac{h^{k-1}C}{(k-1)!} \sum_{l\in \mathcal P_k } \int \int\left(\int|\psi^{(l)}(x-tu)-\psi^{(l)}(x)|^q|u|_1^{qk}|K(u)|^q \mathbb{1}_{0\leqslant t\leqslant h}\pi(x)dx\right)^{1/q} \!\!\!\! du\,dt\nonumber\\
&\leqslant   \frac{h^{k-1}C}{(k-1)!}M_1\pi_\infty^{1/q}   \sum_{l\in \mathcal P_k}\int\left(|tu|_1^{q(s-k)} |u|_1^{qk}|K(u)|^q \right)^{1/q} \mathbb{1}_{0\leqslant t\leqslant h} du\,dt\nonumber\\
&=  \frac{ h^{s}C}{(k-1)!(s-k+1)} M_1\pi_\infty ^{1/q} \,  \#\{\mathcal P_k \} \int |u|_1^{s}|K(u)| du\nonumber.
\end{align*}
This implies (\ref{eq:L2regularization}). 

 To show (\ref{eq:L1regularization}), it suffices to provide an upper-bound proportional to $h^{s}$ and another one proportional to $h^r$. Because $|\pi(\psi-\psi_h)| \leqslant \pi(|\psi-\psi_h|)$, applying (\ref{eq:L2regularization}) with $q=1$, we obtain the upper-bound $C_1M_1 \pi_\infty h^{s}$. By Fubini's theorem and using the symmetry about $0$ of $K$, it holds
 \begin{align}\label{eq:fubini}
\int \pi(x)\psi_{h}(x)dx= \int \psi(x)\pi_{h}(x)dx.
\end{align} 
Hence, introducing the probability density $\widetilde \psi(y) = \left(\int |\psi(x)|dx\right)^{-1} |\psi(y)|$, $y\in\mathbb R^d$, we find
\begin{align*}
\left|  \int \pi(x)(\psi(x)-\psi_{h}(x))dx\right| 
&=\left|  \int \psi(x)(\pi(x)-\pi_{h}(x))dx\right| \\
&\leqslant  \left(\int |\psi(x)|dx\right)    \int \widetilde \psi(x) \left| \pi(x)-\pi_{h}(x) \right| dx\\
&=  \left(\int |\psi(x)|dx \right)      \left\|\pi -\pi_{h} \right\|_{L_1(\widetilde \psi) }.
\end{align*}
Applying (\ref{eq:L2regularization}) with $\widetilde \psi$ and $\pi$ in place 
of $\pi$ and $\psi$ respectively, we get the bound $\widetilde C_1 M_2 \psi_\infty  h^r$, 
for some $\widetilde C_1>0$ depending on $K$ and $r$. Equation (\ref{eq:L1regularization})
is then deduced from these two bounds.

\subsection{Proof of Proposition \ref{theorem:preservationprop}} 
For any $f$ and $\widetilde f$ in $  \mathcal F_1\times \ldots \times  \mathcal F_d$, we have
\begin{align}
|\Psi(f) - \Psi(\widetilde f))|&\leqslant\sum_{j=1}^d  C_j(F) |f_j-\widetilde f_{j}| .\label{eq:boundL2Q}
\end{align}
Let us first prove that $G$ is an envelope for $\mathcal G$. Applying (\ref{eq:boundL2Q}) with $f_0$ in place of $\widetilde f$, we get that $ 2 \sum_{j=1}^d C_j(F) F_j $ is an envelope for the class 
$\mathcal G - \Psi (f_0)$. As a result $ G$ is an envelope for the class $\mathcal G$. 
The envelope property is proved.

Let $Q$ be such that $Q (G^2) <+\infty$. Define the following probability measures on $\mathcal X$,
\begin{align*}
d Q_j = q_j^{-2}   C_j(F)^2 \, dQ ,\qquad \text{with } q_j^2 = \int  C_j(F)^2\,  dQ .
\end{align*} 
Note that $q_j<+\infty $ is implied by $Q (G^2) <+\infty$. 
Let $\mathcal C_j$ denote a set of functions forming an $\epsilon \|F_j\|_{L_2(Q_j)} $-covering 
of the metric space $(\mathcal F_j, L_2(Q_j))$. 
For $f=(f_1,\ldots f_d) \in\mathcal F_1\times \ldots \times  \mathcal F_d$, 
there exists $\widetilde f= (\widetilde  f_{1},\ldots\widetilde f_{d})\in \mathcal C_1\times\ldots\times\mathcal C_d$ 
such that, using (\ref{eq:boundL2Q}) and the Minkowski inequality,
\begin{align*}
\|\Psi(f)-\Psi(\widetilde f) \|_{L_2(Q)}& \leqslant\sum_{j=1}^d \| (f_j-\widetilde f_{j}) {C_j}(F)\|_{L_2( Q)} \\
&\leqslant \sum_{j=1}^d  q_j \|f_j -\widetilde f_{j}  \|_{L_2(Q_j)} \\
&\leqslant\epsilon\sum_{j=1}^d q_j\|F_j\|_{L_2(Q_j)}.
\end{align*}
The number of possible $d$-uplets $ (\widetilde  f_{1},\ldots\widetilde f_{d})$ is at most 
$\prod_{j=1}^d \#\{ \mathcal C_j\}$, thus
\begin{align*}
\mathcal N \Big(\mathcal G,\,L_2(Q),   \epsilon  \sum_{j=1}^d q_j\|F_j\|_{L_2(Q_j)}\Big)
\leqslant \prod_{j=1}^d \mathcal N \left(\mathcal F_j ,\, L_2(Q_j),\, \epsilon\|F_j\|_{L_2(Q_j)}\right).
\end{align*}
We have
\begin{align*}
\int G(x)^2dQ&\geqslant \int |\Psi (f_0)|^2 dQ+4 \sum_{j=1}^d \int  F_j^2 {C_j}(F)^2\, dQ\\
&\geqslant  \sum_{j=1}^d \int  F_j^2 {C_j}(F)^2\, dQ\\
&=\sum_{j=1}^d q_j^2\|F_j\|_{L_2(Q_j)}^2.
\end{align*}
Combining this with the Schwartz inequality gives 
\begin{align*}
\sum_{j=1}^d  q_j  \|F_j\|_{L_2(Q_j)} 
&\leqslant d ^{1/2} \left(\sum_{j=1}^d  q_j ^2 \|F_j\|_{L_2(Q_j)}^2  \right)^{1/2} 
\leqslant d^{1/2} \|G\|_{L_2(Q)}.
\end{align*}
Hence
\begin{align*}
\mathcal N \left(\mathcal G, \, L_2(Q),\epsilon d^{1/2}  \|G\|_{L_2(Q)}\right)\leqslant  \prod_{j=1}^d 
\mathcal N \left(\mathcal F_j ,\, L_2(Q_j),\,\epsilon\|F_j\|_{L_2(Q_j)}\right).
\end{align*} 
The VC class assumption on $\mathcal F_j$, with characteristics $(A_j,v_j)$, implies that the right hand side 
is smaller than $\varepsilon^{-(v_1+\dots +v_d)}A_1^{v_1}\dots A_d^{v_d}$. This concludes the proof.

\subsection{Proof of Proposition \ref{proposition:pollard+nollan}}
The first statement is proved in \cite{wellner1996}, Example~2.5.4.
The second statement, under (\ref{assump:kernel})$(i)$, is given by Lemma 22, (i), in \cite{nolan+p:1987} (the definitions are different than the ones we use; as stated page 789, their ``Euclideanity'' implies VC). 
Under (\ref{assump:kernel}) $(ii)$, invoking Lemma 22, (ii), in \cite{nolan+p:1987}, the class of real valued functions 
$\{ x \mapsto K^{(0)} ( h^{-1}(y_1-x_1))\, : \, y_1\in \mathbb R,\, h>0 \}$ 
is a uniformly bounded VC class of function. 
Then, since $\Psi (z)=z_1\dots z_d$ satisfies (\ref{def:locally_lipschitz}), Proposition \ref{theorem:preservationprop} implies the conclusion.

\subsection{Proof of Proposition \ref{corollary:vcclass_example}}
We begin by applying Proposition \ref{theorem:preservationprop} to  
$\mathcal F_1=\{ (t,x) \mapsto  \mathbb{1}_{t\leqslant M} \, : \, M\in\mathbb R \}$ 
and $\mathcal F_2=\{ (t,x) \mapsto  K ( h^{-1}(y-x))\, : \, y\in \mathbb R^d,\, h>0  \}$ (both classes are VC by Proposition \ref{proposition:pollard+nollan}),
with $\Psi(z_1,z_2)=z_1z_2$ which satisfies (\ref{def:locally_lipschitz}). The resulting class
\begin{align*}
\{ (t,x) \mapsto  \mathbb{1}_{t\leqslant M} K ( h^{-1}(y-x))\, : \, y\in \mathbb R^d,\, h>0 ,\, M\in\mathbb R \}
\end{align*}  
is uniformly bounded VC. 
Then we can consider the product of $\{(t,x)\mapsto t\}$ and $\mathcal F_3$. 
As for every $z_1,\tilde z_1\in [- A_1,A_1] $ and  $z_2,\tilde z_2\in [- A_2,A_2] $, we have
\begin{align*}
| z_1z_2 - \tilde z_1\tilde z_2| \leqslant A_2 |z_1 -\tilde z_1|+  A_1 |z_2 - \tilde z_2| ,
\end{align*}
this yields a VC class with envelope $(t,x)\mapsto 2 ((1\vee K_\infty  ) |t| +  (1\vee |t|) K_\infty )$.

\subsection{Proof of Theorem \ref{prop:unif_conv_density}}
We have to study
\begin{align*}
\widehat \pi(y) &=n^{-1}\sum_{i=1}^nK_i(y),
\end{align*}
where 
\begin{align*}
K_i(y)= K_{h_n}(y-X_i).
\end{align*}
As in the proof of Theorem~\ref{rosenth}, we will use the split chain defined in section~\ref{s2}, 
$\theta_a(k)$ will stand for the time of the $k$-th return to the set $a$ ($\theta_a(1)>0$), and $l_n$, defined in (\ref{def:l_n}), is the number
of such returns before $n$. 

Recall that $ \alpha_0=\mathbb {E}_a[\theta_a]$. Using the stationarity and equation~(\ref{bert0}),  its expectation under $\pi$ 
can be computed as
\begin{align*}
\mathbb E_\pi [l_n] =\sum_{k=1}^n\mathbb E_\pi [\mathbb{1}_{Z_k\in a}]= n\mathbb E _\pi[ \mathbb{1}_{\{Z_0\in a\}}] =\frac{n}{\alpha_0} .
\end{align*} 
Let us now evaluate the variance of $l_n$. From Theorem \ref{rosenth} with with $g(z)=(\mathbb{1}_{\{z\in a\}}-\alpha_0^{-1} )/n$, there exists $C>0$ such that, for any $n\geqslant 1$,
\begin{align*}
&\mathbb {E}_\pi\Big[\Big(\sum_{i=1}^n g(X_i)\Big)^2\Big]
\leqslant n C \left(\pi(g^2)+  \mathbb E_\pi[g(X_0)^2 \tau_A^{p_0}]\right).
\end{align*}
Because 
\begin{align*}
\mathbb E_\pi [\mathbb{1}_{\{Z_0\in a\}} \tau_A^{p_0}] &= \int \mathbb E_z[\tau_A^{p_0}] \mathbb{1}_{\{ z\in a\}} d\pi(z) = \mathbb E_a [\tau_A^{p_0}] \pi(a)<+\infty,
\end{align*}
we conclude that there exists some constant $\widetilde C>0$ such that
\begin{align}\label{lnconv}
 \mathbb {E}_\pi[ (l_n/n - \alpha_0^{-1})^2]  \leqslant \widetilde  C n^{-1}.
\end{align} 
Consequently,
\begin{align*}
\sup_{y\in\mathbb R^d} \left|\left(1-\frac {\alpha_0 (l_n-1)}{n}\right) \pi_{h_n}(y) \right| 
\leqslant \left| 1-\frac {\alpha_0 (l_n-1)}{n}\right|   
\sup_{y\in\mathbb R^d} \left|\pi(y) \right|  \longrightarrow 0,\quad \text{in $\mathbb P_\pi$-probability.}
\end{align*}
Hence, in place of $\widehat \pi(y) -\pi_{h_n}(y)$, we can rather study 
\begin{align*}
\widehat T (y) =  \widehat \pi(y)- \frac {\alpha_0 (l_n-1) }{n} \pi_{h_n}(y)
\end{align*}
which will have a simpler expansion. 
The idea of the proof is to use the results available for the independent case.
Since terms inside one block are not independent, the trick is to notice that we can consider 
the case when only one term in each block is picked at random. 
More precisely if $\Delta_k={\theta_a}(k+1)-{\theta_a}(k)$ and $I_k$ is a uniformly chosen point 
among $\{{\theta_a}(k)+1,\dots,{\theta_a}(k+1)\}$, the variables
\begin{align*}
\widetilde K_{k}(y)=K_{I_k}(y),\qquad k=1,\ldots l_n-1,
\end{align*}
satisfy
\begin{align*}
\mathbb E[\widetilde K_k(y)|\mathscr F_\infty]=\Delta_k^{-1} \sum_{i={\theta_a}(k)+1}^{{\theta_a}(k+1)} K_i(y),
\end{align*}
where $\mathscr F_\infty $ denote the $\sigma$-field generated by the whole chain. We can rewrite
\begin{align*}
\widehat T(y)
&=n^{-1}\sum_{i=1}^{{\theta_a}(1)} K_i(y)+n^{-1}\sum_{k=1}^{l_n-1}
\left(\left( \sum_{i={\theta_a}(k)+1}^{{\theta_a}(k+1)} K_i(y)\right) -\alpha_0 \pi_{h_n}(y)\right)
+n^{-1}\sum_{i={\theta_a}(l_n)+1}^n K_i(y)\\
&=n^{-1}\sum_{i=1}^{{\theta_a}(1)} K_i(y)
+\mathbb E\left\{n^{-1}\sum_{k=1}^{l_n-1} \left(\Delta_k \widetilde K_k(y) -\alpha_0 \pi_{h_n}(y) \right) 
\Big |\mathscr F_\infty \right\}
+n^{-1}\sum_{i={\theta_a}(l_n)+1}^n K_i(y)\\
&=\widehat T_1(y)+\mathbb E[Z_n(y)|\mathscr F_\infty]+\widehat T_2(y).
\end{align*}
Concerning the boundary terms $\widehat T_1$ and $\widehat T_2$, we have
\begin{align*}
\mathbb E_\pi\Big[\sup_{y\in\mathbb R^d} |\widehat T_1(y)|\Big] \leqslant  n^{-1} 
\mathbb {E}_\pi\Big[\sup_{y\in \mathbb R^d} \sum_{i=1}^{\theta_a} |K_{h_n}(y-X_i)|\Big]
\leqslant  n^{-1}h_n^{-d} K_\infty \mathbb {E}_\pi[{\theta_a}]   ,  
\end{align*}
and similarly,
\begin{align*}
\mathbb E_\pi\Big[\sup_{y\in\mathbb R^d} |\widehat T_2(y)|\Big]&\leqslant n^{-1} 
\mathbb {E}_a\Big[\sup_{y\in \mathbb R^d} \sum_{i=1}^{\theta_a} |K_{h_n}(y-X_i)|\Big]
=n^{-1}h_n^{-d} K_\infty \mathbb {E}_a[{\theta_a}].
\end{align*}
We now consider the term $\mathbb E [Z_n(y)|\mathscr F_\infty]$. From the definition of $I_1$ and using (\ref{bert0}), for any measurable function $g$ with $\pi(g)<+\infty$, we have
\begin{align}\label{eux}
\mathbb E_a [\Delta_1 g(X_{I_1})]=\mathbb {E}_a\Big[{\theta_a} \frac1{\theta_a}\sum_{i=1}^{\theta_a} g(X_i)\Big]=\alpha_0\pi(g).
\end{align}
In particular, 
$\alpha_0 \pi_{h_n}(y) =\mathbb E_a[\Delta_1 \widetilde K_1(y) ]$. It follows that
\begin{align*}
Z_n(y)  = n^{-1}\sum_{k=1}^{l_n-1} \left( \Delta_k \widetilde K_k(y)  - \mathbb E_a[\Delta_1 \widetilde K_1(y) ]  \right)  .
\end{align*}
We are planning to apply Theorem~\ref{theorem:gine+g}, but the problems for now are that $l_n$
is random and $\Delta_k$ is not bounded. Define
\begin{align}\label{choiceofM}
 m_n = (nh_{n}^{-d}/\log (n))^{1/(2p_0-1)}.
\end{align} 
We shall analyse the terms when $\Delta_k\leqslant m_n$ and $\Delta_k>m_n$ independently. The reason why such a value of $m_n$ is considered shall be made clear in the next few lines (below equation (\ref{decideM})). We have
\begin{align}
Z_n(y)\label{znz1nz2n}
&= n^{-1}\sum_{k=1}^{l_n-1} \left( \mu_k \widetilde K_k(y) -\mathbb E _a [\mu_1 \widetilde K_1(y) ] \right)
+ n^{-1}\sum_{k=1}^{l_n-1}\left( \nu_k \widetilde K_k(y)-\mathbb E_a [\nu_1 \widetilde K_1(y)]\right) \\
\mu_k&=\Delta_k\mathbb{1}_{\Delta_k\leqslant m_n}\nonumber\\
\nu_k&=\Delta_k\mathbb{1}_{\Delta_k>m_n}\nonumber.
\end{align}
 Choose $\eta_n= \sqrt {\log(n) /n}$, and set $l_n^0=\lfloor n\alpha_0^{-1}\rfloor $, $l_n^-=\lfloor n (\alpha_0^{-1}-\eta_n )\rfloor $,  $l_n^+=\lfloor n(\alpha_0^{-1}+\eta_n )\rfloor $. 
By construction, as $n\rightarrow +\infty$,
\begin{align*}
n^{1/2} (l_n^+-\alpha_0^{-1}) \rightarrow +\infty, \quad   n^{1/2} (l_n^--\alpha_0^{-1}) \rightarrow -\infty.
\end{align*}
Therefore, from (\ref{lnconv}), we obtain that the event $l_n^-\leqslant l_n-1\leqslant l_n^+$ 
has probability going to $1$. Suppose from now on this event is realized. The number 
\begin{align*}
l'_n=\big((l_n-1)\wedge l_n^+)\big)\vee l_n^-
\end{align*}
is equal to $l_n-1$. Since $l'_n$ and $l^0_n$ both belong to $[l^-_n,l^+_n]$, 
for every sequence $A_k$, $k=1,2,\ldots$, 
it holds that
\begin{align*}
\Big|n^{-1}\sum_{k=1}^{l_n'}  A_k\Big|
& \leqslant n^{-1} \Big| \sum_{k=1}^{l_n^0} A_k\Big|  + n^{-1} \sum_{k=l_n^-}^{l_n^+} |A_k|.
\end{align*}
Taking $A_k = \mu_k \widetilde K_k(y) -\mathbb E_a[\mu_k \widetilde K_k(y) ]$, this gives
\begin{align}\nonumber
n^{-1}\sum_{k=1}^{l_n'}  &\left( \mu_k \widetilde K_k(y) -\mathbb E_a[\mu_k \widetilde K_k(y) ] \right)\\
& \leqslant  n^{-1} \Big| \sum_{k=1}^{l_n^0} \big( \mu_k \widetilde K_k(y) -\mathbb E_a[\mu_k \widetilde K_k(y) ]\big) \Big| 
+ n^{-1} \sum_{k=l_n^-}^{l_n^+} |\mu_k \widetilde K_k(y) -\mathbb E_a[\mu_k \widetilde K_k(y) ] |.
\label{bound:br1}
\end{align}
We treat the first term of (\ref{bound:br1}) by applying Theorem \ref{theorem:gine+g} 
with $\xi_i = (\Delta_i, X_{I_i} )  $, $i=1,2,\ldots$, and the class of functions 
$\{(t,x)\mapsto t \mathbb{1}_{\{t\leqslant m_n \}} K(h_n^{-1}(x-y) )\, :\, y\in \mathbb R^d \}$. 
This class being  a subclass of (\ref{grosvc}) which is VC with 
envelope $F(t,x) = 2((1\vee K_\infty  )|t|+(1\vee |t| )  K_\infty)$ and characteristic $(A,v)$ (in virtue of Proposition \ref{corollary:vcclass_example}).
Hence we can apply Theorem~\ref{theorem:gine+g}. We have to estimate the various quantities 
involved in (\ref{einmasas1}) and (\ref{einmasas2}). On the first hand,
\begin{align*}
 \sup_{f\in \mathcal F} \mathbb E[f(\xi_1)^2])
&= \sup_{y\in\mathbb R^d}\mathbb E_\pi  [\Delta_1^2\mathbb{1}_{\Delta_1\leqslant m_n}K(h_n^{-1}( X_{I_1}-y )  )^2]\\
&\leqslant m_n\sup_{y\in\mathbb R^d}\mathbb E_\pi [\Delta_1K(h_n^{-1}( X_{I_1}-y ))^2]\\
&= m_n\sup_{y\in\mathbb R^d}\mathbb E_a \Big[\sum_{i=1}^{\theta_a}K(h_n^{-1}(X_{i}-y ))^2\Big]\qquad  \text{(cf.  (\ref{eux})})\\
&=m_n \alpha_0 \sup_{y\in\mathbb R^d}\mathbb {E}_{\pi}[K(h_n^{-1}( X_{1}-y )  )^2]\qquad \text{(cf.  (\ref{taupui})})\\
&\leqslant m_n \alpha_0 h_n^{d}\pi_\infty\int K(x)^2dx\\
&=c^2 m_n h_n^d,\qquad \qquad c^2=\alpha_0 \|\pi\|_\infty\int K(x)^2dx.
\end{align*} 
On the other hand, using $(1\vee |t|)\leqslant 1+|t|$ and then (\ref{c0tau0prime}), we find
\begin{align*}
\mathbb E[F(\xi_1)^2]\leqslant 2( (1+K_\infty ) \mathbb E |\Delta_1| +K_\infty   )  \leqslant C (1+ \sup_{x\in A} \mathbb {E}_{x}[\tau_A^2]),
\end{align*}
for some $C>0$.
We choose  
\begin{align*}
\sigma^2&=c^2 m_n h_n^d.
\end{align*} 
With this choice of $\sigma$, equation~(\ref{einmasas1}) is satisfied and (\ref{einmasas2}) will be satisfied  if
\begin{align*}
&c^2m_nh_n^d\geqslant  \frac{16 vn^{-1}}{2} \log\Big(A^2\max\big(1,\mathbb E[F(\xi_1)^2]/ c^2m_nh_n^{d}\big)\Big)m_n^2K_\infty^2.
\end{align*}  
Since $h_n\rightarrow 0$ and $m_n\rightarrow +\infty$, 
this condition will be met for $n$ large enough if, as $n\rightarrow\infty$,
\begin{align*}
&m_n\leqslant  \frac{nh_n^d}{\log(h_n^{-1})}.
\end{align*}  
This is equivalent to
\begin{align}
& \frac{nh_n^{-d}}{\log (n)}\ll \Big(\frac{nh_n^d}{\log(h_n^{-1})}\Big)^{2p_0-1}
\end{align} 
which is
\begin{align}
& 1 \ll \Big(\frac{nh_n^{dp_0/(p_0-1)}}{\log( n)}\Big)^{2(p_0-1)}\Big(\frac{\log (n)}{\log(h_n^{-1})}\Big)^{2p_0-1}.
\end{align} 
This is satisfied indeed since the first term tends to infinity by assumption, and the fact that
$nh_n^{dp_0/(p_0-1)}\rightarrow+\infty$ implies that the second one is bounded from below.

Computing the bound given in Theorem~\ref{theorem:gine+g}, multiplying by $(nh_n^d)^{-1}$, we obtain that
\begin{align}\nonumber
\mathbb E_\pi \sup_{y\in\mathbb R^d}
\Big| n^{-1}\sum_{k=1}^{l_n^0}\mu_k\widetilde K_k(y)-\mathbb E_a[\mu_k\widetilde K_k(y)]\Big| 
&\leqslant (nh_n^d)^{-1}C_0\sqrt {vl_n^0c^2m_nh_n^d \log\Big(A\big(1\vee\tfrac{\beta}{cm_n^{1/2}h_n^{d/2}}\big)\Big)}
\end{align} 
But since
\begin{align*}
 m_nh_n^d = \left(\frac n{\log( n)}\right)^{1/(2p_0-1)}h_n^{2d(p_0-1)/(2p_0-1)},
\end{align*} 
this quantity is larger than some negative power of $n$ (cf. (\ref{nhdp})) 
and using this for bounding the logarithm, we get
\begin{align}
\mathbb E_\pi \sup_{y\in\mathbb R^d}
\Big| n^{-1}\sum_{k=1}^{l_n^0}\mu_k\widetilde K_k(y)-\mathbb E_a[\mu_k\widetilde K_k(y)]\Big| 
&\leqslant C' B(n,h_n,m_n) \label{bound:br2}
\end{align} 
for some $C'>0$ and where
\begin{align*}
B(n,h,m) = \sqrt{\frac{m\log(n)}{n  h^d}}.
\end{align*}
The second term of (\ref{bound:br1})  is smaller than
\begin{align*}
\Big| n^{-1} \sum_{k=l_n^-}^{l_n^+} |\mu_k \widetilde K_k(y) -\mathbb E_a [\mu_k \widetilde K_k(y) ] |
& - \mathbb E_a|\mu_1 \widetilde K_1(y) -\mathbb E_a[\mu_1 \widetilde K_1(y) ]| \Big| \\
&+ n^{-1} (l_n^+-l_n^-) \mathbb E_a(|\mu_1 \widetilde K_1(y) -\mathbb E_a[\mu_1 \widetilde K_1(y) ]|).
\end{align*}
Consider the class 
\begin{align*}
\left\{ (\beta,x)\mapsto \left| \beta \mathbb{1}_{\{\beta\leqslant m_n \}}  K(h^{-1}( x-y )  ) 
-\mathbb E_a[\mu_1  K(h^{-1}( X_1-y )  )]  \right| \, : \, y\in \mathbb R^d, \, h>0  \right\}.
\end{align*}
This class is included in the larger class of functions $z\mapsto | f(z) - w|$, where $f$ describes the VC class (\ref{grosvc}), and $w\in\mathbb R$ 
is ranging over the segment $A= [-\alpha_0  K_\infty, \alpha_0 K_\infty]$.
This larger class is VC because, (i) the class $\{ f(z) - w \}$ remains VC and (ii) 
the transformation $x\mapsto |x|$ being Lipschitz, we can apply Proposition~\ref{theorem:preservationprop}. 
This is basically the same as before, with the only difference that now 
$l_n^+-l_n^-\leqslant 3\eta_n n  $, we obtain that there exists a constant $C>0$ such that 
\begin{align}
\mathbb E_\pi\sup_{y\in \mathbb R^d}  \Big| n^{-1}  
\sum_{k=l_n^-}^{l_n^+} |\mu_k \widetilde K_k(y) -\mathbb E_\pi[\mu_1 \widetilde K_1(y) ] | \Big|
&\leqslant C\left( \sqrt\eta_n   B(n,h_n,m_n) +\eta_n  \mathbb E_\pi|\mu_1 \widetilde K_1(y) | \right).\label{bound:br3}
\end{align}
From (\ref{eux}), we know that
\begin{align*}
\mathbb E_\pi |\mu_1  \widetilde K_1(y) |\leqslant \mathbb E_a\left[\Delta_1  |\widetilde K_1(y)|\right]
= \alpha_0\int |K_{h_n}(y-x)| \pi(x)dx \leqslant \alpha_0 \pi_\infty  \int |K(u)|du.
\end{align*}
 Then, bringing together (\ref{bound:br1}), (\ref{bound:br2}) and (\ref{bound:br3}) gives that, for some $C>0$,
\begin{align}
\mathbb E_\pi \sup_{y\in \mathbb R^d}  \Big  |n^{-1}\sum_{k=1}^{l_n'}  \left( \mu_k \widetilde K_k(y) -\mathbb E_a[\mu_k \widetilde K_k(y) ] \right)\Big|
& \leqslant C  B(n,h_n,m_n)  \label{bound:br4}
\end{align}
because $\eta_n\ll B(n,h_n,m_n)$ and $\eta_n\ll 1$.
Concerning the second term in (\ref{znz1nz2n}), since $l'_n\leqslant n$ and by Lemma \ref{lemma:tau}, we have
\begin{align}\nonumber
\mathbb E_\pi\Big[\sup_{y\in \mathbb R^d} \big| n^{-1} \sum_{k=1}^{l_n'} \nu_k \widetilde K_k(y)\big|\Big]
&\leqslant K_\infty h_n^{-d}\mathbb E_\pi\Big[ n^{-1} \sum_{k=1}^{n} \nu_k \Big]\\
&= K_\infty h_n^{-d}\mathbb E\pi\Big[ {\theta_a}\mathbb{1}_{\theta_a>m_n} \Big]\nonumber\\
&\leqslant K_\infty h_n^{-d}m_n^{-(p_0-1)}\mathbb E_\pi\Big[ {\theta_a}^{p_0} \Big]\nonumber\\
&\leqslant K_\infty h_n^{-d}m_n^{-(p_0-1)}\lambda_0^{-p_0}\frac{e^{\lambda_0}}{ (e^{\lambda_0/p_0} -1)^{p_0}}  \sup_{x\in A} \mathbb {E}_{x}[\tau_A^{p_0}]. \label{bound:br5}
\end{align}
Bringing together (\ref{znz1nz2n}), (\ref{bound:br4}), (\ref{bound:br5}), we finally get, for some $C>0$,
\begin{align}\label{decideM}
\mathbb E_\pi \Big[\sup_{y\in \mathbb R^d} \big| n^{-1} \sum_{k=1}^{l_n'} \Delta_k \widetilde K_k(y)\big|\Big]
&\leqslant   C \left( B(n,h_n,m_n) +h_n^{-d}m_n^{-(p_0-1)}\right) .
\end{align}
The value of $m_n$ that balances these terms together is given by (\ref{choiceofM})
and we  obtain that there exists $C>0$ such that
\begin{align*}
\mathbb E_\pi \Big[\sup_{y\in \mathbb R^d} \big| n^{-1} \sum_{k=1}^{l_n'} \Delta_k \widetilde K_k(y)\big|\Big]
&\leqslant   C\left( \frac { \log (n) }{ n h_n ^{dp_0/(p_0-1)}  }  \right)^{(p_0-1)/(2p_0 - 1)} .
\end{align*}
By assumption, this term goes to $0$ as $n\rightarrow +\infty$. Let $\epsilon>0$, we have that
\begin{align*}
\mathbb P_\pi\Big(\sup_{y\in \mathbb R^d}\big|\mathbb E[Z_n(y)|\mathscr F_\infty]\big|\geqslant\epsilon\Big) 
&\leqslant \mathbb {P}_\pi\Big(\mathbb E[\sup_{y\in \mathbb R^d}|Z_n(y)|\ |\mathscr F_\infty]  
\geqslant \epsilon \Big) \\
&\leqslant\mathbb P_\pi\Big(
\mathbb E\big[\sup_{y\in \mathbb R^d}|Z_n(y)|\,|\mathscr F_\infty\big]\geqslant\epsilon,\, l_n-1=l'_n\Big) 
+ \mathbb {P}_\pi( l_n-1\neq  l'_n )\\
&\leqslant \epsilon ^{-1} 
\mathbb E_\pi\Big[\mathbb E[\sup_{y\in \mathbb R^d}|Z_n(y)|\,|\mathscr F_\infty]\mathbb{1}_{\{ l_n-1= l'_n\}} \Big] 
+ \mathbb {P}_\pi( l_n-1\neq  l'_n )\\
&= \epsilon ^{-1}\mathbb E_\pi \Big[\sup_{y\in \mathbb R^d} |Z_n(y)| \mathbb{1}_{\{ l_n-1= l'_n \}}\Big] 
+ \mathbb {P}_\pi( l_n-1\neq  l'_n )\\
&\leqslant\epsilon ^{-1}\mathbb E_\pi
\Big[\sup_{y\in \mathbb R^d} \big|n^{-1}\sum_{k=1}^{l_n'} \Delta_k \widetilde K_k(y)\big|\Big] 
+ \mathbb {P}_\pi( l_n-1\neq  l'_n ).
\end{align*}
Then we finish the proof by recalling that $  l_n- 1=    l'_n$ whenever $l_n^- \leqslant l_n-1\leqslant l_n ^+ $, which has probability going to $1$.

\subsection{Proof of Corollary \ref{prop:coro:lowerbound_density}}
Without loss of generality, because $h_n\rightarrow 0$, we can assume that $K(u)=0$ for every $|u|\geqslant 1$.  Theorem~\ref{prop:unif_conv_density} implies that 
\begin{align*}
\inf_{y \in  Q}  \widehat \pi(y) \geqslant \inf_{y\in  Q }  \pi_{h_n}(y) - \epsilon_n ,
\end{align*} 
where $\epsilon_n = \sup_{y\in \mathbb R^d} |\widehat \pi(y)- \pi_{h_n}(y) | \rightarrow 0$, in $\mathbb P_\pi$-probability. Define, for any $x\in Q$ and $h>0$,
\begin{align*}
&b(x,h)=\inf_{y\in Q , \, |y-x|\leqslant h } \pi(y),\\
&M(x,h)=\sup_{y\in Q,\, |y-x|\leqslant h} \pi(y).
\end{align*}
Let $K = K_+ +K_-$ be the decomposition of $K$ with respect to the non-negative part and the negative part. Let $x\in Q$, for every $h>0$, we have
\begin{align*}
\pi_h (x)&=\int \pi(x-hu)  K(u) du \\
&\geqslant b(x,h) \int  \mathbb{1}_{\{x-hu\in Q \}} K_+(u)du +M(x,h) \int  \mathbb{1}_{\{x-hu\in Q \}} K_-(u) du\\
&= b(x,h) \int \mathbb{1}_{\{x-hu\in Q \}} K(u) du +(M(x,h)-b(x,h)) \int \mathbb{1}_{\{x-hu\in Q \}} K_-(u) du\\
&\geqslant  b \int  \mathbb{1}_{\{x-hu\in Q \}} K(u) du -\sup_{x\in Q} |M(x,h)-b(x,h)|,
\end{align*}
By virtue of Heine's theorem, $\pi$ is uniformly continuous on $Q$, hence $\sup_{x\in Q} |M(x,h)-b(x,h)|\rightarrow 0 $ as $h\rightarrow 0$. Consequently, as $h_n\rightarrow 0$, we have for every $\epsilon>0$, that $\inf_{x\in Q} \pi_{h_n} (x)\geqslant b c -\epsilon$. Choosing $\epsilon$ small enough and using that $\epsilon_n\rightarrow 0 $, in $\mathbb P_\pi$-probability, gives the statement.


\section{Changing the initial measure}\label{s33}

Appendix \ref{s31} focuses on Markov chains that either starts from their atom $a$, e.g., Lemma \ref{lemma:insideblock}, or from their invariant measure $\pi$, e.g., Theorem \ref{rosenth}. Some link between the underlying probabilities $\mathbb P_a$ and $\mathbb P_\pi$  is provided in Lemma \ref{lemma:invariant_meas_formula_extended}.  
The following lemma turns out to be a useful ingredient to extend convergences in $\mathbb P_\pi$-probability 
to convergences in $ \mathbb P_\nu$, $\nu$ being any measure absolutely continuous with respect to $\pi$.

\begin{lemma}\label{lemma:initial_measure}
Let $(X_i)_{i\in\mathbb N}$ be a Markov chain and let $\nu$ be a probability measure absolutely continuous with respect to $\pi$. Suppose that $f: \cup_{n\geqslant 1} \mathbb R^n \rightarrow \mathbb R^+$ is a bounded measurable function such that $\mathbb E_\pi f(X_1,\ldots X_n) \rightarrow 0$ as $n\rightarrow +\infty$, then
\begin{align*}
\mathbb E_\nu f(X_1,\ldots X_n) \rightarrow 0.
\end{align*}
\end{lemma}

\begin{proof}
Denote by $q$ the Radon–Nikodym derivative of $\nu $ with respect to $\pi$. Let
$$g_n(x) = \mathbb E_x [ f(X_1,\ldots X_n)],$$
and $M>0$ be such that $\sup_{n\geqslant 1} f(x_1,\ldots x_n) <M$ for every sequence $(x_n)_{n\in\mathbb N^*}$. We have
\begin{align*}
\mathbb E_\nu f(X_1,\ldots X_n) &= \int  g_n(x) d\nu (x)\\
&= \int   g_n(x) q(x) d\pi (x)\\
&\leqslant A  \int g_n(x) d\pi(x) +\int   g_n(x) q(x) \mathbb{1}_{q(x)>A} d\pi (x) \\
&= A  \mathbb E_\pi f(X_1,\ldots X_n) +\mathbb E_{\nu} [  g_n(X_0)  \mathbb{1}_{q(X_0)>A}]\\
&\leqslant  A  \mathbb E_\pi f(X_1,\ldots X_n) +M\mathbb P_{\nu} (q(X_0)>A),
\end{align*}
for any $A>0$. In the previous display, the term on the right-hand side can be made arbitrarily small by taking $A$ large and for any such $A$, the term on the left-hand side goes to $0$ by assumption. 
\end{proof}

For application purposes, this simple lemma is fine. 
Notice however that
by Corollary~6.9 of \cite{nummelin:1984}, under an additional aperiodicity assumption, 
the distribution of our Harris chain converges in total variation to $\pi$ as soon as $\mathbb{E}_\pi[\tau_A]<\infty$ (see also Definition~5.5 and Proposition~5.15). In view of the equations (\ref{c0tau0}) and (\ref{eq:inequality_invariant_measu}),
this means that $\sup_{x\in A}E_x[\tau_A^2]<\infty$. The control of the bound in Theorem \ref{rosenth}
already requires this.
Given this, it is not difficult to 
check that the conclusion of Lemma \ref{lemma:initial_measure} holds true even if $\nu$ is 
a Dirac measure $\delta_x$,
under the additional assumption that for all $k\in \{1,\ldots n\}$
\begin{align*}
\sup_{(x_1,\dots x_n,y) \in \mathbb R^{n+1}}|f(x_1,\ldots x_n) - f(x_1,\dots x_{k-1},y,x_{k+1},\ldots x_n)|
= \varepsilon_n\rightarrow 0.
\end{align*}
This is obviously satisfied when $f$ is an empirical mean over uniformly bounded terms. We have indeed for any fixed $x_0$
\begin{align*}
\mathbb E_x f(X_1,\ldots X_n) 
&=\mathbb E_x [f(x_0,\dots x_0,X_{k+1},\ldots X_n)]+k O(\varepsilon_n)\\
&=\int \mathbb E_y[f(x_0,\dots x_0,X_{1},\ldots X_{n-k})]P^k(x,dy)+k O(\varepsilon_n)\\
&= \mathbb E_\pi [f(x_0,\dots x_0,X_{k+1},\dots X_n)] +O(\|\pi-P^k(x,.)\|)f_\infty+k O(\varepsilon_n)\\
&= \mathbb E_\pi [f(X_1,\dots X_n)]+O(\|\pi-P^k(x,.)\|)f_\infty +2k O(\varepsilon_n).
\end{align*}
This remark is of course not new, and is related to the coupling properties of the Harris chains, e.g., Proposition~29 in \cite{robros}.

\end{appendices}


\bibliographystyle{chicago}
\bibliography{biblio_accel}

\end{document}